\documentclass[11pt,leqno]{article}

\usepackage{amsmath, amsfonts, theorem,color,
}
\usepackage[latin1]{inputenc}

\usepackage[T1]{fontenc}

\usepackage[english]{babel}

\usepackage{lmodern}

\usepackage{amsmath}

\usepackage{amssymb}

\usepackage{appendix}

\usepackage{mathrsfs}
\usepackage{latexsym}
\usepackage{graphicx}
\usepackage{cancel}
\usepackage[normalem]{ulem}
\makeatletter
\def\@maketitle{\newpage
    \null
    \vskip .8truein
    \begin{center}%
     {\bf \@title \par}%
     \vskip 1.5em
     {\small
      \lineskip .5em
      \begin{tabular}[t]{c}\@author
      \end{tabular}\par}%
    \end{center}%
    \par
    \vskip .4truein}
\@addtoreset{equation}{section} \@addtoreset{theorem}{section}
\@addtoreset{lemma}{section} \@addtoreset{proposition}{section}
\@addtoreset{definition}{section} \@addtoreset{corollary}{section}
\@addtoreset{remark}{section}

\def\dfrac#1#2{\ds{\frac{#1}{#2}}}

\let\nn=\nonumber

\newcommand{\re}{{\mathbb R}}

\newcommand{\Z}{{\mathbb Z}}
\newcommand{\He}{{\mathbb H}}
\newcommand{\cH}{{\mathcal H}}
\newcommand{\Te}{{\mathbb T}_{\cH}}

\let\ds=\displaystyle

\def\R{{\mathbb R}}
\def\N{{\mathbb  N}}
\def\N{{\mathbb N}}

\newtheorem{theorem}{Theorem}[section]
\newtheorem{lemma}{Lemma}[section]
\newtheorem{proposition}{Proposition}[section]
\newtheorem{definition}{Definition}[section]
\newtheorem{corollary}{Corollary}[section]
\newtheorem{remark}{Remark}[section]
\newtheorem{example}{Example}[section]
\newtheorem{hypothesis}{Hypothesis}[section]

\newenvironment{Proofc}[1]{\smallskip\par\noindent\textsc{#1}\quad}%
  {\hfill$\Box$\bigskip\par}

\DeclareMathOperator{\diver}{div}
\DeclareMathOperator{\tr}{tr}
\def\proof{\list{}{\setlength{\leftmargin}{0pt}
                      \parskip=0pt\parsep=0pt\listparindent=2em
                      \itemindent=0pt}\item[]\futurelet\testchar\@maybe}

\def\@maybe{\ifx[\testchar \let\next\@Opt
          \else \let\next\@NoOpt \fi \next}
\def\@Opt[#1]{{\it Proof of #1.\ }}\def\@NoOpt{{\it Proof.\ }}
\begin{document}
\title{\Large \bf First order Mean Field Games in the Heisenberg group: periodic and non periodic case}

\author{{\large \sc Paola Mannucci\thanks{Dipartimento di Matematica ``Tullio Levi-Civita'', Universit\`a di Padova, mannucci@math.unipd.it}, Claudio Marchi \thanks{Dipartimento di Ingegneria dell'Informazione \& Dipartimento di Matematica ``Tullio Levi-Civita'', Universit\`a di Padova, claudio.marchi@unipd.it}, Nicoletta Tchou \thanks{Univ Rennes, CNRS, IRMAR - UMR 6625, F-35000 Rennes, France, nicoletta.tchou@univ-rennes1.fr}}} 
\maketitle
\begin{abstract}
In this paper we study evolutive first order Mean Field Games in the Heisenberg group~$\He^1$; each agent can move only along ``horizontal'' trajectories which are given in terms of the vector fields generating~$\He^1$ and the kinetic part of the cost depends only on the horizontal velocity. The Hamiltonian is not coercive in the gradient term and the coefficients of the first order term in the continuity equation may have a quadratic growth at infinity.
The main results of this paper are two: the former is to establish the existence of a weak solution to the Mean Field Game system while the latter is to represent this solution following the Lagrangian formulation of the Mean Field Games.
We shall tackle both the Heisenberg-periodic and the non periodic case following two different approaches. 
To get these results, we prove some properties which have their own interest: uniqueness results for a second order Fokker-Planck equation and a probabilistic representation of the solution to the continuity equation.
\end{abstract}

%

\noindent {\bf Keywords}: Mean Field Games, first order Hamilton-Jacobi equations, continuity equation, Fokker-Planck equation, noncoercive Hamiltonian, Heisenberg group, degenerate optimal control problem.

\noindent  {\bf 2010 AMS Subject classification:} 35F50, 35Q91, 49K20, 49L25.

\section{Introduction}
In this paper we study evolutive first order Mean Field Game (briefly, MFG) in the Heisenberg group~$\He^1$. Let us recall that the MFG theory started with the works by Lasry and Lions \cite{LL1,LL2,LL3} and by Huang, Malham\'e and Caines \cite{HMC} (see the monographs \cite{AC, C, BFY, GPV, GS} for the many developments in recent years) and studies Nash equilibria when the number of agents tends to infinity and each agent's aim is to control its dynamics so to minimize a given cost which depends on the distribution of the whole population. On the other hand, the Heisenberg group can be seen as the first non-Euclidean space which is still endowed with nice properties as a (noncommutative) group operation, a family of dilations and a manifold structure (see the monographs \cite{BLU, Montg} for an overview).
From the viewpoint of a single agent, the Heisenberg's framework entails that its state cannot change isotropically in all the directions but it can move only along {\it admissible} trajectories.

We shall consider systems of the form
\begin{equation}\label{eq:MFGintrin}
\left\{\begin{array}{lll}
(i)&\quad-\partial_t u+\frac{|D_{\cH}u|^2}{2}=F[m(t)](x)&\qquad \textrm{in }\He^1\times (0,T)\\
(ii)&\quad\partial_t m-\diver_{\cH}  (m D_{\cH}u)=0&\qquad \textrm{in }\He^1\times (0,T)\\
(iii)&\quad m(x,0)=m_0(x), u(x,T)=G[m(T)](x)&\qquad \textrm{on }\He^1,
\end{array}\right.
\end{equation}
where $D_{\cH}$ and $\diver_{\cH}$ are respectively the {\it horizontal gradient} and the {\it horizontal divergence} while $F$ and $G$ are strongly regularizing coupling operators. We shall consider both the Heisenberg-periodic and non periodic case because they give rise to different mathematical issues so, in our opinion, it is worth studying both of them. In the former one, $m_0$, $F$ and~$G$ are $Q_\cH$-periodic w.r.t. $x$ with $Q_\cH:=[0,1)^3$ and we shall focus our attention on $Q_\cH$-periodic solution to~\eqref{eq:MFGintrin}. 
It is worth to stress that in the periodic setting we take advantage of the definition of a periodicity cell introduced in~\cite{BMT2,BMT} and 
the invariance of the operators in the Heisenberg group: see Section~\ref{sub:periodicity} below for the precise definitions and properties.
In the non periodic case, we obviously drop all these periodicity assumptions.\\
For readers which are not familiar with intrinsic calculus, in Euclidean coordinates, system~\eqref{eq:MFGintrin} becomes 
\begin{equation}\label{eq:MFG1}
\left\{\begin{array}{lll}
(i)&\quad-\partial_t u+H(x, Du)=F[m(t)](x)&\qquad \textrm{in }\re^3\times (0,T)\\
(ii)&\quad\partial_t m-\diver  (m\, \partial_pH(x, Du))=0&\qquad \textrm{in }\re^3\times (0,T)\\
(iii)&\quad m(x,0)=m_0(x), u(x,T)=G[m(T)](x)&\qquad \textrm{on }\re^3,
\end{array}\right.
\end{equation}
where, for $p=(p_1,p_2, p_3)$ and $x=(x_1, x_2,x_3)$, the Hamiltonian $H(x,p)$ is
\begin{equation}\label{Hd}
H(x,p):=\frac{1}{2}((p_1-x_2p_3)^2+(p_2+x_1p_3)^2)=\frac{|pB(x)|^2}{2}\ \textrm{with }\ 
 B(x):=  \begin{pmatrix}
\!\!&1& 0\!\\
\!\!&0&1&\!\\
\!\!&-x_2&x_1\!
\end{pmatrix}\in M^{3\times 2}
\end{equation}
while the drift $\partial_pH(x,p)$ is
\begin{equation}\label{Hp}
\partial_pH(x,p)=pB(x)B(x)^T=(p_1-x_2p_3, p_2+x_1p_3, -p_1x_2+p_2x_1+p_3(x_1^2+x_2^2)).
\end{equation}

These MFG systems arise when the generic player with state~$x$ at time~$t$ must follow {\it horizontal curves} with respect to the two vector fields $X_1$ and $X_2$ generating the Heisenberg group (see \eqref{vectorFields} below):
\begin{equation}\label{DYNH}
x'(s)=\alpha_1(s)X_1(x(s))+\alpha_2(s)X_2(x(s))
\end{equation}
namely
\begin{equation*}
x_1'(s)=\alpha_1(s),\quad x_2'(s)=\alpha_2(s),\quad x_3'(s)=-x_2(s)\alpha_1(s)+x_1(s)\alpha_2(s).
\end{equation*}
Each agent wants to choose the control $\alpha=(\alpha_1, \alpha_2)$ in $L^2([t,T];\R^2)$ in order to minimize the cost
\begin{equation}\label{Jgen}
J_t^{m}(x(\cdot),\alpha(\cdot)):=\int_t^T\left[\frac12 |\alpha(\tau)|^2+F[m(\tau)](x(\tau))\right]\,d\tau+G[m(T)](x(T))
\end{equation}
where $m(\cdot)$ is the evolution of the whole population's distribution while $(x(\cdot),\alpha(\cdot))$ is a trajectory obeying to~\eqref{DYNH}.

Let us observe two important issues of these MFG systems: $(i)$ the Hamiltonian~$H$ is not coercive in~$p$ uniformly in~$x$, $(ii)$ in equation~\eqref{eq:MFG1}-(ii) the coefficient of the first order term may have quadratic growth in~$x$.\\
Point~$(i)$ prevents the application of standard approaches for first order MFG (for instance, see~\cite{BFY,C,CGPT,GS}). Moreover, we recall that the papers~\cite{AMMT, CM, MMMT} already tackled MFG systems with noncoercive Hamiltonians for first order MFG while papers \cite{DF, FGT} dealt with second order hypoelliptic MFG. However, the results in~\cite{CM,AMMT} do not apply to the present setting because these papers consider a different kind of admissible trajectories.
Note that the present case is neither encompassed in our previous work~\cite{MMMT} because the matrix $B$ in \eqref{Hd} does not fulfill the assumptions of \cite[Section 5]{MMMT}, in particular $\det B(x)B^T(x)= 0$ for any $x\in \re^3$. 
The degeneracy of the matrix $B(x)B^T(x)$ implies that we cannot prove the uniqueness of optimal trajectories for a.e. starting points with respect to the initial distribution of players and hence to get a representation formula as in \cite{C, MMMT}.
The issue of finding necessary or sufficient conditions ensuring the uniqueness of the optimal trajectories for a.e. starting points is challenging and open; we hope to study it in a future work.\\
On the other hand, point~$(ii)$ gives rise to some difficulties for applying the vanishing viscosity method, especially for the well-posedness of the Cauchy problem for equation~\eqref{eq:MFG1}-(ii) with the viscosity term and also for its stochastic interpretation.

The aims of this paper are two; the former one is to prove the existence of a weak solution to system~\eqref{eq:MFGintrin} (see Definition~\ref{defsolmfg} for the periodic case and respectively Definition~\ref{defsolmfg_illi} for the non periodic case). To this end, we establish several properties of the solution to the Hamilton-Jacobi equation~\eqref{eq:MFGintrin}-(i) (as semiconcavity, Lipschitz continuity, regularity of the optimal trajectories for the associated optimal control problem). Afterwards, we adapt the techniques introduced by PL. Lions in his lectures at Coll\`ege de France \cite{C, LL3} (see also \cite{AMMT, MMMT} for similar approaches for some noncoercive Hamiltonians). In the periodic case we perform a vanishing viscosity procedure with the {\it horizontal} Laplacian which preserves the periodic structure of the problem. In the non periodic case, we perform a vanishing viscosity procedure with the {\it Euclidean} Laplacian and a truncation argument. The vanishing viscosity procedures permit to exploit the regularity results of the Laplacian while the truncation argument permits to avoid parabolic Cauchy problems with coefficients growing ``too much'' at infinity.
Let us emphasize that our approach for the non periodic case relies on some compactness of the initial distribution of players and on the sublinear growth of the coefficients of~$B$ but it does not need the H\"ormander condition. 
The latter, and main, aim of this paper is to prove that this weak solution is also a {\it mild} solution in the sense of the definition introduced by Cannarsa and Capuani~\cite{CC} for the case of state-constrained MFG where the agents control their velocity.
Roughly speaking, as in the Lagrangian approach for MFG (see~\cite{BCS,CC}), this property means that, for a.e. starting state, the agents follow optimal trajectories for the optimal control problem associated to the Hamilton-Jacobi equation.
In order to prove that our solution is in fact a mild solution, we shall use the superposition principle~\cite[Theorem 8.1.2]{AGS}.
Unfortunately, this result does not apply directly to our Heisenberg-periodic setting; we overcome this issue providing its adaptation to the Heisenberg periodic case.
In particular we shall prove in the Heisenberg framework: $(1)$ an optimal synthesis result, $(2)$ a superposition principle, $(3)$ a uniqueness result in a viscous setting for the Fokker-Planck equation with unbounded coefficients. In our opinion, these results can have an interest  independent of the MFG framework.

Let us stress that our results in the Heisenberg-periodic setting could be generalized with suitable adaptations to the case of invariant operators in other Lie groups~\cite{MS, STROF}.

This paper is organized as follows. Section \ref{Preliminaries} is devoted to define the Heisenberg group, the periodicity and the convolution in the group. Sections from~\ref{sect:mainthm} to~\ref{sect:dim3.1} are devoted to the periodic case: Section \ref{sect:mainthm} contains the assumptions and the statement of our main result, Theorem~\ref{thm:main}, whose proof is given in Section~\ref{sect:dim3.1}. In  Section \ref{OC} we study several properties of the solution of the optimal control problem associated to the Hamilton-Jacobi equation \eqref{eq:MFGintrin}-(i). The main contribution of Section~\ref{sect:continuity} is the proof of Theorem~\ref{prp:m} which establishes the needed regularity of the solution $m$ to prove Theorem~\ref{thm:main}. Sections~\ref{sect:nonper} and~\ref{sect:dim7.1} are devoted to the non periodic case: the former one contains the assumptions and the statement of our main result, Theorem~\ref{thm:main_illi}, while the latter one is devoted to the proof of a preliminary result for an approximation of the continuity equation which provides the main step in the proof of Theorem~\ref{thm:main_illi}. The three appendices deal with A) the definition of the differential operators in the Heisenberg group, B) uniqueness results for a general degenerate Fokker-Plank equation with unbounded coefficients, C) an adaptation of the probabilistic representation result of \cite{AGS} to a generic ``intrinsic'' continuity equation in a periodic setting for the Heisenberg group.

\noindent\underline{Notations.}
For any function $u:\re^3\times\re\ni (x,t)\to u(x,t)\in \re$, $Du$ and~$D^2u$ stand for the Euclidean gradient and respectively Hessian matrix with respect to~$x$.
For any compact set $A$ of $\re^3$, we denote $C^2(A)$ the space of functions with continuous second order derivatives endowed with the norm $\|f\|_{C^2(A)}:=\sup_{x\in A}[|f(x)|+|Df(x)|+|D^2f(x)|]$.\\
For any complete separable metric space $X$, ${\mathcal M}(X)$ and ${\mathcal P}(X)$ denotes the set of nonnegative Radon measures on~$X$, and respectively of Borel probability measures on~$X$. For any complete separable metric spaces $X_1$ and $X_2$, any measure $\eta\in{\mathcal P}(X_1)$ and any function $\phi:X_1\to X_2$, we denote $\phi\#\eta\in{\mathcal P}(X_2)$ the {\it push-forward} of~$\eta$ through~$\phi$, i.e.  $\phi\#\eta(B):=\eta(\phi^{-1}(B))$, for any $B$ measurable set, $B\subset X_2$ (see~\cite[section~5.2]{AGS} for the precise definition and main properties).
For a function $m\in C([0,T],{\mathcal P}(X))$, $m_t$ stands for the probability~$m(t,\cdot)$ on~$X$.


\section{Preliminaries: The Heisenberg group}
\label{Preliminaries}

We introduce the following noncommutative group structure on $\re^3$. We refer to \cite{BLU} for a complete overview on the Heisenberg group. 

\begin{definition}
\label{Hei}
The $3$-dimensional Heisenberg group $\He^1$ is the vector space $\R^3$,
endowed with the following noncommutative group operation, denoted by $\oplus$:
\begin{equation}
\label{Group_Law}
x\oplus y=(x_1, x_2, x_3)\oplus (y_1, y_2, y_3):= (x_1+y_1, x_2+y_2, x_3+y_3-x_2y_1+x_1y_2).
\end{equation}
for all $x=({x}_1,{x}_2,x_3)$, $y=({y}_1,{y}_2,y_3)\in \R^{3}$.
\end{definition}
The law  $x\oplus y$ is called the $x$ {\em {left translation of}} $y$.
We call  $x^{-1}$ the point such that $x^{-1}\oplus x=x\oplus x^{-1}=0$. Note that 
$x^{-1}=(-x_1, -x_2, -x_3)$. Hence we define 
$$x\ominus y := x\oplus y^{-1}$$

In $\He^1$ we define a dilations' family as follows. 
\begin{definition}
\label{DilH}
The dilations in the Heisenberg group are the family of group homeomorphisms defined as, for all $\lambda>0$,  
$\delta_{\lambda}:\He^1\to \He^1$  with
\begin{equation}
\label{Dialtions}
\delta_{\lambda}(x)=(\lambda\, {x}_1,\lambda\, {x}_2, \lambda^2\,x_3),
\quad \forall\;
x=({x}_1,  {x}_2, x_3)\in \He^1.
\end{equation}
\end{definition}

We define the two vector fields 
\begin{equation}\label{vectorFields}
 X_1(x)=\left(\begin{array}{c}1 \\0 \\
 -x_2\end{array}\right)\quad
\textrm{and}
\quad 
X_2(x)=\left(\begin{array}{c}0 \\1 \\x_1\end{array}\right),
\quad \forall\, x=(x_1,x_2,x_3)\in \He^1.
\end{equation}
By these vectors we define the linear differential operators, still called $X_1$ and $X_2$
\begin{equation}\label{VFD}
X_1=\partial_{x_1}-x_2\partial_{x_3},\ X_2=\partial_{x_2}+x_1\partial_{x_3}.
\end{equation}
Note that their commutator $[X_1,X_2]:=X_1\,X_2-X_2X_1$ verifies: $[X_1,X_2]=-2\partial_{x_3}$; hence, together with their commutator $[X_1,X_2]$, they span all $\re^3$. We say that the vectors $X_1(x)$ and $X_2(x)$ are the generators of $\He^1$.
The fields $X_1$ and $X_2$ are left-invariant vector fields, i.e. 
for all $u\in C^{\infty}(\He^1)$ and for all fixed $y\in \He^1$
\begin{equation}\label{LIVF}
X_i(u(y\oplus x))=(X_iu)\,(y\oplus x),\ i=1,2.
\end{equation}
Note that the matrix $B(x)$ defined in \eqref{Hd} is the matrix associated to the vectors $X_1$ and $X_2$. For any regular real-valued function $u$, we shall denote its horizontal gradient and its horizontal Laplacian by $D_\cH u:= (X_1u, X_2u)$ and respectively $\Delta_\cH:=X_1^2u+X_2^2u$ and we observe $D_\cH u=Du\, B(x)$ and $\Delta_\cH u=\textrm{tr}(D^2u\,BB^T)$ where $Du$ and $D^2u$ denote the Euclidean gradient and respectively the Hessian matrix of~$u$.
For any regular $u=(u_1,u_2):\He^1\to \re^2$, we denote its horizontal divergence by $\diver_{\cH}u:=X_1u_1+X_2u_2$ and we note that the left-invariance of $X_i$ ($i=1,2$) entails the left-invariance of $\diver_{\cH}$. We have: $\diver_{\cH}(D_\cH u)=\Delta_\cH u$.

Let us define 
\begin{equation}\label{norm}
\|x\|_\cH:=((x_1^2+x_2^2)^2+x_3^2)^{1/4}
\end{equation}
and the distance associated by the group law
\begin{equation}\label{dist}
d_\cH(x,y):=\|x\oplus y^{-1}\|_\cH.
\end{equation} 
\begin{remark}\label{lemmaagosto} Let us recall that there holds
\begin{equation*}
d_\cH(x,y)\leq |x-y|+(1+|x_1|^{1/2}+|x_2|^{1/2})|x-y|^{1/2} \qquad \forall x,y\in\He^1.
\end{equation*}
For the sake of completeness, let us briefly recall the proof. We have
\begin{eqnarray*}
d_\cH(x,y)&\leq &[(x_1-y_1)^2+(x_2-y_2)^2]^{1/2}+|x_3-y_3+x_2y_1-x_1y_2|^{1/2}\\
&\leq &|x-y|+|x_3-y_3+x_2y_1-x_1y_2|^{1/2}.\end{eqnarray*}
On the other hand, there holds
\begin{eqnarray*}
|x_3-y_3+x_2y_1-x_1y_2|^{1/2}&=&|x_3-y_3+x_2(y_1-x_1)+x_1(x_2-y_2)|^{1/2}\\
&\leq &|x_3-y_3|^{1/2}+|x_2|^{1/2}|y_1-x_1|^{1/2}+|x_1|^{1/2}|x_2-y_2|^{1/2}\\
&\leq& (1+|x_1|^{1/2}+|x_2|^{1/2})|x-y|^{1/2}.
\end{eqnarray*}
Replacing the last inequality in the previous one, we accomplish the proof.
\end{remark}

Using the definitions \eqref{VFD} and \eqref{norm} we easily prove the following equalities:
\begin{lemma}\label{Xcalc} The following equalities hold
\begin{eqnarray*}
(i)&&X_1(\|x\|^2_\cH)= \frac{2x_1(x^2_1+x^2_2)-x_2x_3}{\|x\|^2_\cH},\quad 
X_2(\|x\|^2_\cH)=\frac{2x_2(x^2_1+x^2_2)+x_1x_3}{\|x\|^2_\cH},\\
(ii) &&|D_\cH(\|x\|^2_\cH)|^2=\frac{4(x^2_1+x^2_2)^3+(x^2_1+x^2_2)x_3^2}{\|x\|^4_\cH},\\
(iii)&&X_1^2(\|x\|^2_\cH)= \frac{6x^2_1+3x^2_2}{\|x\|^2_\cH}-\frac{|X_1(\|x\|^2_\cH)|^2}{\|x\|^2_\cH},\quad 
X_2^2(\|x\|^2_\cH)= \frac{6x^2_2+3x^2_1}{\|x\|^2_\cH}-\frac{|X_2(\|x\|^2_\cH)|^2}{\|x\|^2_\cH},\\
(iv)&&\Delta_\cH(\|x\|^2_\cH)= \frac{9(x^2_1+x^2_2)}{\|x\|^2_\cH}-\frac{|D_\cH(\|x\|^2_\cH)|^2}{\|x\|^2_\cH}.
\end{eqnarray*}
\end{lemma}

\subsection{Periodicity in the Heisenberg group}\label{sub:periodicity}

The notion of periodicity is introduced by the group law $\oplus$. We follow the definition and the results given in~\cite{BMT2,BMT} (see also  \cite{BW}).\\
Let $Q_\cH =[0,1)^3$.
We can construct a tiling of $\He^1$ by the {\em property of pavage}:
for every $x\in\He^1$ there exists a unique $z\in \Z^3$ such that there exists a unique $q_\cH =q_\cH (x)\in Q_\cH $ such that $z\oplus q_\cH =x$.

We can now define the {\em $Q_\cH $-periodicity} on $\He^1$ with respect to this reference pavage.
\begin{definition}\label{QHper}
A function $f$ defined on $\He^1$ is said $Q_\cH $-periodic if  for any $x\in\He^1$,
$$f(x)=f(q_\cH (x)),$$
where $q_\cH (x)$ is the unique element of $Q_\cH $ such that $x=z\oplus q_\cH (x)$ for the unique $z\in \Z^3$.
\end{definition}
We will denote by $C^{\infty}_{Q_\cH , per}$ the set of the functions $f\in C^{\infty}(\He^1)$ that are $Q_\cH $-periodic.
The definition of $Q_\cH $-periodicity is equivalent to the following definition of {\em $1_\cH $-periodicity}:
\begin{definition}
A function $f$ defined on $\He^1$ is said $1_\cH $-periodic if for any $x\in\He^1$ and any $z\in \Z^3$ there holds
$$f(z\oplus x)=f(x).$$
\end{definition}
\begin{lemma}\label{lemma:period}
A function $f$ is $Q_\cH $-periodic if and only if is $1_\cH $-periodic.
\end{lemma}
\begin{proof}
Note that by the pavage property if $f$ is $1_\cH $-periodic then is $Q_\cH $-periodic. Conversely, for any $x\in\He^1$ there exist unique $z$ and $q_\cH $ such that $x=z\oplus q_\cH $. For any $z'\in\Z^3$ we write
$z'\oplus x= z'\oplus z\oplus q_\cH $. Since $z'\oplus z\in\Z^3$ then $q_\cH (z'\oplus x)= q_\cH (x)$ and by the definition of $Q_\cH $-periodicity we get $f(z'\oplus x)=f(q_\cH (z'\oplus x))= f(q_\cH (x))=f(x)$, for any $z'\in\Z^3$.
\end{proof}
\begin{definition}\label{Htoro}
We denote with $\Te$ the torus in the Heisenberg group~$\He^1$, namely $\He^1/\Z^3$ using the following equivalence law: $x\sim y$ if there exists $z\in \Z^3$ such that $z\oplus x=y$.
The torus is naturally endowed with the distance induced by $d_\cH$:
\begin{equation*}
d_{\Te}(x,y):=\inf d_\cH(x',y')\qquad\forall x,y\in \Te
\end{equation*}
where the infimum is performed over all the couple $(x',y')\in \He^1\times \He^1$ with $x\sim x'$, $y\sim y'$.
\end{definition}
\begin{remark} Lemma~\ref{lemma:period} ensures that $x\sim x'$ if and only of $q_\cH (x)=q_\cH (x')$. It is worth to observe that the Heisenberg torus~$\Te$ does not coincide with the Euclidean torus; especially, $\Te$ is not obtained identifying the points of two opposite faces of~$\overline{Q_\cH}$ with the same two coordinates. As a matter of facts, this happens between the two faces given by~$x_3=0$ and~$x_3=1$. For completeness, let us write the identification of points $(1,x_2,x_3)$ with $(x_2,x_3)\in [0,1]^2$ with points $(0,x_2',x_3')$ with $(x_2',x_3')\in [0,1]^2$: we have
\begin{equation*}
(1,x_2,x_3)\sim\left\{\begin{array}{ll}
(0,x_2,x_3-x_2) &\quad \textrm{for } x_3-x_2\in [0,1)\\
(0,x_2,x_3-x_2+1) &\quad \textrm{for } x_3-x_2\in (-1,0];
\end{array}\right.
\end{equation*}
actually, for $x_3-x_2\in [0,1)$ there holds $(-1,0,0)\oplus (1,x_2,x_3)= (0,x_2,x_3-x_2)$ while for $x_3-x_2\in [-1,0)$ there holds $(-1,0,1)\oplus (1,x_2,x_3)= (0,x_2,x_3-x_2+1)$. Moreover, $(1,1,x_3)\sim (0,0,x_3)$ because $(-1,-1,0)\oplus (1,1,x_3)=(0,0,x_3)$ for every $x_3\in[0,1)$ and $(1,1,1)\sim(0,0,0)$ because $(-1,-1,-1)\oplus(1,1,1)=(0,0,0)$. And similarly for the remaining cases.
\end{remark}
\begin{remark}\label{rmk:misureper}
With a slight abuse of notations, throughout this paper we shall identify any measure $\eta\in{\mathcal M}(Q_\cH)$ with the same measure on the torus~$\Te$ and also with the measure $\eta'\in{\mathcal M}(\He^1)$ such that $\eta'(z\oplus A)=\eta (A)$ for any measurable set $A\subset Q_\cH$ and $z\in\Z^3$.
\end{remark}
\begin{remark}
We recall from \cite[Proposition 1.3.21]{BLU} that the Haar measure associated to the Heisenberg group coincides with the Lebesgue measure.
\end{remark}
\subsection{Convolution on Heisenberg group}\label{sec:conv}
We define the convolution in Heisenberg group of a function $\psi\in L^1_{loc}(\He^1)$ by a function $\rho\in C^{\infty}_c(\He^1)$ as
\begin{equation}\label{conH}
(\psi \ast\rho)(x)= \int_{\He^1} \psi(y)\rho(x\ominus y) dy.
\end{equation}
In the proof of Theorem~\ref{thm:main}, we will use the convolution by the regularizing kernel 
\begin{equation}\label{rhoeps}
\rho_{\epsilon}(x)= C(\epsilon) \rho_0(\|x/\epsilon\|^4_\cH )
\end{equation}
where $\rho_0(t)=e^{-t}$ and the constant $C(\epsilon)$ is chosen such that 
$\int_{\He^1}\rho_{\epsilon}(x)dx=1$. 
This convolution has the following properties
\begin{proposition}\label{propconvH}
We have
\begin{enumerate}
\item [(i)] $\psi \ast\rho_{\epsilon}=\rho_{\epsilon}\ast\psi$;
\item [(ii)] If $\psi$ is $Q_\cH$-periodic then also $\psi \ast\rho_{\epsilon}$ is $Q_\cH$-periodic;
\item [(iii)] If $\psi$ is $L^p(\He^1)$ for some $p\geq 1$, then  $\psi \ast\rho_{\epsilon}$ is $C^{\infty}(\He^1)$;
\item [(iv)] If $\psi$ is $L^1_{loc}(\He^1)$ then $\psi\ast\rho_{\epsilon}\to \psi$ in $L^1_{loc}(\He^1)$ as $\epsilon\to 0$;
\item [(v)] If $\psi$ is differentiable then $$X_i\psi \ast\rho_{\epsilon}=(X_i\psi) \ast\rho_{\epsilon}= \psi \ast X_i\rho_{\epsilon},$$ where the vectors $X_i$ are defined in \eqref{VFD};
\item[(vi)] If $\psi\geq 0$ in $\He^1$ and $\int_{\He^1}\psi(x)dx=C>0$ then $\psi \ast\rho_{\epsilon}(x)>0$ for any $x\in\He^1$.
\end{enumerate}
\end{proposition}
\begin{proof}
All the proofs are easy and standard using the fact that the Haar measure for the Heisenberg group coincides with the Lebesgue measure. For the sake of completeness, we only provide the detailed proof of $(v)$ for $X_1$ as an example
\begin{eqnarray*}
X_1(\psi \ast\rho_{\epsilon})(x)&=&\partial_{x_1}(\psi \ast\rho_{\epsilon})(x)-x_2\partial_{x_3}(\psi \ast\rho_{\epsilon})(x)
= \int \psi(y)[\partial_{x_1}\rho_{\epsilon}+y_2\partial_{x_3}\rho_{\epsilon}-x_2\partial_{x_3}\rho_{\epsilon}]dy\\
&=&
\int \psi(y)[\partial_{x_1}\rho_{\epsilon}-(x_2-y_2)\partial_{x_3}\rho_{\epsilon}]dy=
\int \psi(y)X_i\rho_{\epsilon}(x\ominus y)dy=
\psi \ast (X_i\rho_{\epsilon})(x).
\end{eqnarray*}
\end{proof}

\section{Definitions, assumptions and main results}\label{sect:mainthm}
In this section, we introduce the functional spaces  needed for the definition of solution to system~\eqref{eq:MFGintrin}, our assumptions and we state the main results of this paper for the periodic case.
Following \cite{BK}, we adapt the classical notion of Kantorovich-Rubinstein distance to~the set~$\Te$ in terms of the distance~$d_{\Te}$ introduced in Definition~\ref{Htoro}:
$$
{\bf d_{1}}(m, m'):= \inf_{\pi\in\Pi(m,m')}\int_{\Te\times \Te}d_{\Te}(x,y)d\pi(x,y)\qquad \forall m,m'\in {\mathcal P}(\Te)
$$
where 
\begin{equation}\label{Pi}
\Pi(m, m'):=\{\pi {\text { Borel prob. meas. on }} \Te\times \Te: \pi(A\times \Te)=m(A), \pi(\Te \times A)=m'(A)\},
\end{equation}
where $A$ is any Borel set $A\subset \Te$.\\
For the sake of completeness, let us recall that:
${\bf d_{1}}(m,m')=\sup \int_{\Te}f(x)d(m-m')(x)$, 
where the supremum is taken over the set of all maps $f:\Te\to \re$ which are $1$-Lipschitz continuous with respect to $d_{\Te}$ (see \cite[Theorem 1.1.5]{BK}).

We set
\begin{equation*}
{\mathcal P}_{per}(\He^1):=\left\{m\in {\mathcal M}(\He^1):\ m_{\mid Q_\cH}\in {\mathcal P} (Q_\cH),\quad\textrm{$m$ is $Q_\cH$-periodic}\right\}
\end{equation*}
where for ``$m$ is $Q_\cH$-periodic'' we mean $m(z\oplus A)=m(A)$ for every $z\in \Z^3$ and every measurable $A\subset \He^1$.
By Remark~\ref{rmk:misureper}, we identify ${\mathcal P}_{per}(\He^1)$ with ${\mathcal P}(\Te)$.
We assume that the set ${\mathcal P}_{per}(\He^1)$ is endowed with the distance~${\bf d_{1}}$.

\vskip .1cm

Throughout this section, unless otherwise explicitly stated, we shall require the following hypotheses:
\begin{enumerate}
\item[(H1)]\label{H1} the functions~$F$ and $G$ are real-valued function, continuous on ${\mathcal P}_{per}(\He^1)\times\He^1$, moreover, for any fixed $m\in{\mathcal P}_{per}(\He^1)$, $F[m](\cdot)$ and $G[m](\cdot)$ are $Q_\cH $-periodic;
\item[(H2)]\label{H2} the map $m\to F[m](\cdot)$ is Lipschitz continuous from~${\mathcal P}_{per}(\He^1)$ to $C^{2}(\re^3)$; moreover, there exist~$C\in \mathbb R$ and $\delta_0\in(0,1]$ such that
$$\|F[m](\cdot)\|_{C^{2+\delta_0}(\re^3)}, \|G[m](\cdot)\|_{C^{2}(\re^3)}\leq C,\qquad \forall m\in {\mathcal P}_{per}(\He^1);$$
\item[(H3)]\label{H4} the distribution~$m_0:\He^1\to \re$ is a nonnegative $C^0$ function, $Q_\cH$-periodic with $\int_{Q_\cH }m_0dx=1$.
\end{enumerate}
\begin{example} Easy examples of $F$ and $G$ are given by the convolution of a regular kernel (as the one defined in \eqref{rhoeps}) with~$m$.
In this case, Proposition~\ref{propconvH} ensures that assumptions~(H1) and~(H2) are satisfied.
\end{example}

\vskip 3mm

We now introduce our definitions of solution of the MFG system~\eqref{eq:MFGintrin} and state the main result concerning its existence.
\begin{definition}\label{defsolmfg}
A couple $(u,m)$ of $Q_\cH$-periodic functions on~$\He^1\times [0,T]$ is a solution of system~\eqref{eq:MFGintrin} if:
\begin{itemize}
\item[1)] $u$ belongs to $W^{1,\infty}(\He^1\times[0,T])$;
\item[2)] $m$ belongs to $C^0([0,T];{\mathcal P}_{per}(\He^1))$ and for all $ t\in [0,T]$, $m_t$ is absolutely continuous w.r.t. the Lebesgue measure. Let $m(\cdot, t)$ denote the density of $m_t$. The function $(x,t)\mapsto m(x,t)$ is bounded;
\item[3)] Equation~\eqref{eq:MFGintrin}-(i) is satisfied by $u$ in the viscosity sense in~$\He^1\times(0,T)$;
\item[4)] Equation~\eqref{eq:MFGintrin}-(ii) is satisfied by $m$ in the sense of distributions  in~$\He^1\times(0,T)$.
\end{itemize}
\end{definition}
\begin{remark}
Any solution $(u,m)$ of the MFG system~\eqref{eq:MFGintrin} is also a solution in~$\Te\times [0,T]$ by the identification of~${\mathcal P}_{per}(\He^1)$ with ${\mathcal P}(\Te)$.
\end{remark}
\begin{remark}
From Lemma \ref{8.1.2} in Appendix C, we get that the distributional solution of \eqref{eq:MFGintrin}-(ii) stated in point 4) of the definition \ref{defsolmfg}
is automatically continuous in the sense of point 2) of the same definition.
\end{remark}

In order to give a more detailed description of our solution, it is expedient to use the notion of {\it mild} solution introduced by~\cite{CC}.
This notion is reminiscent of the Lagrangian approach to MFGs (see~\cite{BCS}) and it relies on replacing probability measures on the state space with probability measures on arcs on the state space.\\
We define the set of AC arcs in $\He^1$
\begin{equation}\label{gamma}
\Gamma:=\{\gamma\in AC((0,T), \He^1)\}
\end{equation}
 and 
the evaluation map $e_t: \Gamma\to \He^1$ as
\begin{equation}\label{eval}
e_t(\gamma)=\gamma(t).
\end{equation}
For any $x\in \He^1$, we define the set of arcs starting at~$x$
\begin{equation*}
\Gamma_0[x]:=\{\gamma\in \Gamma,\ \gamma(0)=x\}
\end{equation*}
and the set of horizontal arcs starting at~$x$ with an associated control law
\begin{equation*}
\mathcal A(x,0):=\{(\gamma, \alpha) : \gamma\in \Gamma_0[x],\ \alpha\in L^2([0,T], \re^2), (\gamma, \alpha) {\text { solves }}\eqref{DYNH}\}.
\end{equation*}
Given $m_0\in{\mathcal P}_{per}(\He^1)$,  we define 
$$\mathcal P_{m_0}(\Gamma)=\{\eta\in{\mathcal M}(\Gamma):\ m_0 =e_0\# \eta\quad \textrm{and }e_t\# \eta\in {\mathcal P}_{per}(\He^1) \quad \forall t\in[0,T] \}.$$
For any $\eta\in \mathcal P_{m_0}(\Gamma)$ and for any $x\in \He^1$, 
we consider the cost
\begin{equation}\label{Jeta}
J^{\eta}_x(\gamma(\cdot), \alpha):=\int_0^T\left[\frac12 |\alpha(\tau)|^2+F[e_\tau\#\eta](\gamma(\tau))\right]\,d\tau+G[e_T\#\eta](\gamma(T))
\end{equation}
where $(\gamma, \alpha) \in \mathcal A(x,0)$.
For any $\eta\in \mathcal P_{m_0}(\Gamma)$ and for any $x\in \He^1$ we define the set of optimal horizontal arcs starting at~$x$
\begin{equation}\label{Gammaeta}
\Gamma^{\eta}[x]:= \{\overline\gamma:\ (\overline\gamma, \overline\alpha) \in \mathcal A[x,0]: 
J^{\eta}_x(\overline\gamma(\cdot), \overline\alpha)= \min_{(\gamma, \alpha) \in \mathcal A(x,0)} J^{\eta}_x(\gamma, \alpha)\}.
\end{equation}
\begin{definition}\label{mfgequil}
A measure $\eta\in \mathcal P_{m_0}(\Gamma)$ is a {\em MFG equilibrium} for $m_0$ if 
$$supp\, \eta\subseteq \bigcup_{x\in\He^1} \Gamma^{\eta}[x].$$
\end{definition}

This means that the support of $\eta $ is contained in the set $\cup_{x\in\He^1}\{\gamma\in \Gamma_0[x]: \textrm{$\gamma$ is a minimizer of $J^{\eta}_x$}\}$ (see also~\cite{CC}).

\begin{definition}\label{mild}
A couple $(u,m)\in C^0([0,T]\times \He^1)\times C^0([0,T]; \mathcal P_{per}(\He^1))$ is called {\em {mild solution}} if there exists a MFG equilibrium $\eta$ for $m_0$ such that:
\begin{itemize}
\item[i)]
$m_t =e_t\# \eta$;
\item[ii)] $u$ is given by
$$
u(x,t)= \inf_{(\gamma, \alpha) \in \mathcal A(x,0)}\int_t^T\left[\frac12 |\alpha(\tau)|^2+F[e_\tau\#\eta](\gamma(\tau))\right]\,d\tau+G[e_T\#\eta](\gamma(T))
.$$
\end{itemize}
\end{definition}
Now we can state the main result of this paper.
\begin{theorem}\label{thm:main}
Under the above assumptions:
\begin{itemize}
\item[i)] System \eqref{eq:MFGintrin} has a solution $(u,m)$;
\item[ii)] $(u,m)$ is a mild solution.
\end{itemize}
\end{theorem}
\begin{remark}
As a matter of fact, from the proof of this theorem we get that any solution, as in Definition \ref{defsolmfg} is a mild solution.
\end{remark}

\begin{remark} Uniqueness holds under classical hypothesis on the monotonicity of $F$ and $G$ as in \cite{C}.
\end{remark}

%
%
\section{Formulation of the optimal control problem}\label{OC}
In this section, we tackle the optimal control problem associated to the Hamilton-Jacobi equation~\eqref{eq:MFGintrin}-(i); in particular we shall show that the value function solves this equation, is $Q_\cH$-periodic, Lipschitz continuous and semiconcave in~$x$. Throughout this section we shall assume the following hypothesis
\begin{hypothesis} \label{BasicAss}
\begin{enumerate}
\item $f$, $g$ are $Q_\cH $-periodic w.r.t. x;
\item $f\in C^{0}([0,T],C^2(\re^3))$, $g\in C^2(\re^3)$;
so 
there exists a constant $C$ such that
$$\sup_{t\in[0,T]}\|f(\cdot,t)\|_{C^2(\re^3)} + \|g\|_{C^2(\re^3)} \leq C.$$
\end{enumerate}
\end{hypothesis}

\begin{definition}\label{def:OCD} We consider the following optimal control problem:
\begin{equation}\label{def:OC}
\text{minimize } J_t(x(\cdot),\alpha(\cdot)):
=\displaystyle\int_t^T\dfrac12|\alpha(s)|^2+f( x(s),s)\,ds+g(x(T))
\end{equation}
subject to $(x(\cdot), \alpha(\cdot))\in \mathcal A(x,t)$, where
\begin{equation} \mathcal A(x,t):=\left\{(x(\cdot), \alpha(\cdot))\in AC([t,T]; \re^3)\times L^2([t,T]; \re^2):\, \textrm{\eqref{DYNH} holds a.e. with } x(t)= x\right\}.
\end{equation}
A couple $(x(\cdot),\alpha(\cdot))\in\mathcal A(x,t)$ is said to be admissible. We say that $x^*(\cdot)$ is an optimal trajectory if there is a control $\alpha^*(\cdot)$ such that $(x^*(\cdot),\alpha^*(\cdot))\in  \mathcal A(x,t)$ is optimal for the optimal control problem in~\eqref{def:OC}. Also, we shall refer to the system~\eqref{DYNH}
as to the dynamics of the optimal control problem in~\eqref{def:OC}.
\end{definition}
\begin{remark}\label{rem:uniquex} Notice that, given a control law $\alpha\in L^2([t,T]; \re^2)$ and an initial point~$x$,  there is a \emph{unique} trajectory $x(\cdot)$ such that $(x(\cdot),\alpha)\in\mathcal A(x,t)$.
\end{remark}
\begin{remark} \label{2.3} 
  Hypothesis~\ref{BasicAss} ensures that, for any $(x,t)\in \He^1\times(0,T)$, the optimal control problem in definition \ref{def:OCD} admits a solution $(x^*(\cdot),\alpha^*)$ thanks to the LSC with respect to the weak $L^2$ topology. Moreover, just testing $J_t(x^*(\cdot),\alpha^*)$ against $J_t(x,0)$, we get
\begin{equation}\label{bd:alpa_L2}
\|\alpha^*\|_{L^2(t,T)}\leq C_1:=C[(T-t)+1],
\end{equation}
where $C$ is the constant introduced in Hypotheses~\ref{BasicAss}.
In particular, by H\"older inequality,
\begin{equation}\label{bd:x_1,x_2}
x^*\in C^{1/2}([t,T], \He^1).
\end{equation}

\end{remark}
\begin{definition} The value function for the cost $J_t$ defined in \eqref{def:OC} is
\begin{equation}\label{repr}u(x,t):=\inf\left\{ J_t(x(\cdot), \alpha):\, (x(\cdot),\alpha)\in \mathcal A(x,t)\right\}.
\end{equation}
An optimal couple $(x^*(\cdot), \alpha^*)$ for the control problem in definition \ref{def:OCD} is also said to be optimal for $u(x,t)$.
\end{definition}
The following lemma states that, under Hypothesis~\ref{BasicAss}, the value function $u$ is $Q_\cH$-periodic in~$x$ hence we can restrict our study to $Q_\cH$.
\begin{lemma}\label{uper}
Let $u$ be the value function introduced in \eqref{repr}. Then $u$ is $Q_\cH $-periodic in~$x$.
\end{lemma}
\begin{proof}
We have to prove that $u(z\oplus x, t)=u(x,t)$ for any $z\in\Z^3$ and for any $x\in\He^1$.
Note that if $x(s)$ and $y(s)$ solves \eqref{DYNH} with the same law of control $\beta$ and with respectively $x(t)=x$ and $y(t)=z\oplus x$, then 
$y(s)= z\oplus x(s)$; actually there hold
\begin{eqnarray*}
y_i(s) &=& z_i+x_i+\int_t^s\beta_i(\tau)d\tau= z_i+ x_i(s),\qquad \textrm{for } i=1,2,\\
y_3(s) &=& z_3+x_3-z_2x_1+z_1x_2+\int_t^s(z_2+ x_2(\tau))(-\beta_1(\tau))+ (z_1+ x_1(\tau))\beta_2(\tau)d\tau\\
&=&z_3+\left(x_3- \int_t^sx_2(\tau)\beta_1(\tau)+ x_1(\tau)\beta_2(\tau)d\tau\right)-z_2\left(x_1+\int_t^s\beta_1(\tau)d\tau\right)\\
&&+z_1\left(x_2+\int_t^s\beta_2(\tau)d\tau\right)\\
&=&z_3 +x_3(s)-z_2 x_1(s)+z_1 x_2(s).
\end{eqnarray*}
Taking advantage of the $Q_\cH$-periodicity of~$f$ and~$g$, we deduce
\begin{eqnarray*}
u(z\oplus x,t)
&=&\inf_{\beta}\displaystyle\int_t^T\dfrac12|\beta(s)|^2+f(y(s),s)\,ds+g(y(T))\\
&=&\inf_{\beta}\displaystyle\int_t^T\dfrac12|\beta(s)|^2+f(z\oplus x(s),s)\,ds+g(z\oplus x(T))\\
&=&
\inf_{\beta}\displaystyle\int_t^T\dfrac12|\beta(s)|^2+f(x(s),s)\,ds+g(x(T))= u(x,t)
\end{eqnarray*}
namely, the value function is $Q_\cH$-periodic.
\end{proof}
The following proposition ensures that we can restrict our study on uniformly bounded controls.

\begin{proposition}\label{boundalfa}
  Let $u$ be the value function introduced in \eqref{repr}. Then, there exists a constant $C_2$ (depending only on $T$ and on the constant $C$ of Hypothesis \ref{BasicAss}) such that there holds
\begin{equation}\label{restric}
u(x,t)=\inf \{J_t(x(\cdot), \alpha):\ (x(\cdot), \alpha) \in\mathcal A(x,t),\ \|\alpha\|_{\infty}\leq C_2\}
\end{equation}
for any $x=(x_1,x_2,x_3)\in Q_\cH $ and $t\in[0,T]$.
Hence, by the $Q_\cH$-periodicity of~$u$, the optimal control $\alpha$ for any point $(x,t)\in \He^1\times [0,T]$ fulfills: $\|\alpha\|_\infty\leq C_2$.
\end{proposition}
\begin{proof}
The idea of the proof is borrowed from \cite[Theorem 2.1]{BCP}. For $x=(x_1,x_2,x_3)\in Q_\cH $ and $t\in[0,T]$, let $\alpha$ be an optimal control for $u(x,t)$.
For $\mu>0$, let $I_{\mu}:=\{ s\in(t,T): |\alpha(s)|> \mu\}$.
Define 
\begin{equation}\label{alfamu}
\alpha^{\mu}(s)=
\left\{\begin{array}{ll}
\alpha(s) &\text { if }|\alpha(s)|\leq \mu,\\
0 &\text { if }|\alpha(s)|> \mu.
\end{array}\right.
\end{equation}
Let $x^{\mu}(s)$ be the trajectory starting from $x\in Q_\cH$ associated to the control~$\alpha^{\mu}(s)$.
We claim that 
\begin{equation}\label{claimmu}
|x^{\mu}(s)-x(s)|\leq K\int_{I_{\mu}} |\alpha(\tau)|d\tau \qquad\forall s\in[t,T]
\end{equation}
where $K$ is a constant depending only on $C_1$ (see \eqref{bd:alpa_L2}) and $T$. Actually, for the first two components of $x^{\mu}(s)-x(s)$ we have 
\begin{equation}\label{primedue}
\left| x_i^{\mu}(s)-x_i(s)\right|\leq \int_t^s \left|\alpha_i^{\mu}(\tau)-\alpha_i(\tau)\right|d\tau= \int_{I_{\mu}} |\alpha_i(\tau)|d\tau \qquad\forall s\in[t,T],\ i=1,2.
\end{equation}
For the third component, there holds 
\begin{eqnarray*}
x_3^{\mu}(s)-x_3(s)&=&\int_t^s \left[-x_2^{\mu}(\tau)\alpha_1^{\mu}(\tau)+x_1^{\mu}(\tau)\alpha_2^{\mu}(\tau)+ x_2(\tau)\alpha_1(\tau)-x_1(\tau)\alpha_2(\tau)\right]d\tau\\
&=&\int_t^s \left[(x_2(\tau)-x_2^{\mu}(\tau))\alpha_1^{\mu}(\tau)+ x_2(\tau)(\alpha_1(\tau)-\alpha_1^{\mu}(\tau)) \right.\\ 
&&\left. + (x_1^{\mu}(\tau)-x_1(\tau))\alpha_2^{\mu}(\tau)  + x_1(\tau)(\alpha_2^{\mu}(\tau)-\alpha_2(\tau))\right] d\tau.
\end{eqnarray*}
Hence from \eqref{bd:x_1,x_2} and \eqref{primedue}, we infer 
\begin{eqnarray*}
|x_3^{\mu}(s)-x_3(s)|&\leq& \int_{I_{\mu}} |\alpha_2(\tau)|d\tau\int_t^s |\alpha_1^{\mu}(\tau)|d\tau + [|x_2|+C_1  (T-t)^{1/2}]
\int_{I_{\mu}} |\alpha_1(\tau)|d\tau\\
&&+\int_{I_{\mu}} |\alpha_1(\tau)|d\tau\int_t^s |\alpha_2^{\mu}(\tau)|d\tau
+ [|x_1|+C_1 (T-t)^{1/2}]
\int_{I_{\mu}} |\alpha_2(\tau)|d\tau.
\end{eqnarray*}
Moreover, by H\"older inequality and \eqref{bd:alpa_L2}, we have
\[
\int_t^s |\alpha_i^{\mu}(\tau)|d\tau\leq \sqrt{s-t}\|\alpha\|_2\leq C_1\sqrt{T-t}, \ i=1,2.
\]
Replacing the last inequality in the previous one, since $|x_i|\leq 1$, we accomplish the proof of the claim~\eqref{claimmu}.

Now, the definition of the cost $J_t(x(s), \alpha(s))$ in \eqref{def:OC} and the Lipschitz continuity of $f$ and $g$ yield
\begin{eqnarray}
&&J_t(x^{\mu}(s), \alpha^{\mu}(s))- J_t(x(s),\alpha(s))=\nn\\
&&=\displaystyle\int_t^T\dfrac12|\alpha^{\mu}(s)|^2+f( x^{\mu}(s),s)\,ds+g(x^{\mu}(T))-
\displaystyle\int_t^T\dfrac12|\alpha(s)|^2+f( x(s),s)\,ds-g(x(T))\nn\\
&&\leq-\int_{I_{\mu}}\dfrac12|\alpha(s)|^2ds+ 
L_f\int_t^T|x^{\mu}(s)-x(s)|ds+L_g|x^{\mu}(T)-x(T)|\nn\\
&&\leq \int_{I_{\mu}}\left(-\dfrac12|\alpha(s)|^2+ K(L_f (T-t)+L_g)|\alpha(s)|\right) ds,\nn
\end{eqnarray}
where the last inequality comes from \eqref{claimmu}.
Hence, if $I_\mu$ has positive measure for $\mu>2K(L_f T+L_g)$, the last integrand is negative for every $s\in I_\mu$ which contradicts the optimality of $\alpha$. This implies that these $I_\mu$ have null measure and, in particular, $\|\alpha\|_\infty\leq 2K(L_f T+L_g)$.
\end{proof}

\subsection{Necessary conditions and regularity for the optimal trajectories}
The application of the Maximum Principle (see \cite[Theorem 22.17]{Cla})
yields the following necessary conditions.
\begin{proposition}\label{prop:pontriagind}
Let $(x^*, \alpha^*)$ be optimal for the optimal control problem in \eqref{def:OCD}. 
Then, there exists an arc $p\in AC([t,T]; \re^3)$, hereafter called the costate,  such that
\begin{enumerate}
\item The pair
$(x^*, p)$ satisfies the system of differential equations  for a.e. $s\in [t,T]$
\begin{equation}\label{tag:12d}
\left\{
\begin{array}{ll}
\quad x_1'= p_1-x_2p_3\\
\quad x_2'= p_2+x_1p_3\\
\quad x_3'= (x_1^2+x_2^2)p_3+x_1p_2-x_2p_1\\
\quad p_1'=-(p_2+x_1p_3)p_3+  f_{x_1}(x,s)\\
\quad p_2'=(p_1-x_2p_3)p_3+  f_{x_2}(x,s)\\
\quad p_3'=f_{x_3}(x,s)
\end{array}\right.
\end{equation}
with the mixed boundary conditions
\begin{equation}\label{tag:bcd} x(t)=x,\quad p(T)=-D g(x(T)).
\end{equation}
\item The optimal control $\alpha^*$ verifies
\begin{equation}\label{tag:alpha*d}
\left\{
\begin{array}{ll}
\alpha_1^*(s)= p_1-x_2^*p_3,\\
\alpha_2^*(s)= p_2+x_1^*p_3,
\end{array}\right. \qquad \text{ a.e on }[t,T].
\end{equation}
\end{enumerate}
\end{proposition}
\begin{remark}\label{rmk:PMPintrin}
Let us observe that equations~\eqref{tag:12d} and \eqref{tag:alpha*d} can be rewritten in terms of the vector fields as follows
\begin{equation*}
\begin{array}{lll}
 x_1'= X_1p,&\quad x_2'= X_2p,&\quad x_3'=x_2X_1p-x_1X_2p,\\
 p_1'=-p_3X_2p+ f_{x_1}(x,s),&\quad  p_2'=p_3X_1p +  f_{x_2}(x,s),&\quad p_3'=f_{x_3}(x,s)
\end{array}
\end{equation*}
and respectively
\begin{equation*}
\alpha_1(s)=X_1p(s),\qquad \alpha_2(s)=X_2p(s).
\end{equation*}
\end{remark}

\begin{corollary}\label{coro:regularityd}
Let $(x^*, \alpha^*)$ be optimal for the optimal control problem in \eqref{def:OCD}. Then:
\begin{itemize}
\item [1.] The unique solution of the Cauchy  problem
$$
\left\{
\begin{array}{ll}
\quad \pi_1'=-(\pi_2+x^*_1\pi_3)p_3+  f_{x_1}(x^*,s),\\
\quad \pi_2'=(\pi_1-x^*_2\pi_3)\pi_3+  f_{x_2}(x^*,s),\\
\quad \pi_3'=f_{x_3}(x^*,s),\\
\quad \pi(T)=-D g(x^*(T)).
\end{array}\right.
$$
is the costate $p$ associated to $(x^*, \alpha^*)$ as in Proposition~\ref{prop:pontriagind}.
\item[2.] The optimal $\alpha^*$ is a feedback control and it is
uniquely expressed by
$$
\left\{
\begin{array}{ll}
\alpha_1^*(s)= p_1-x^*_2p_3\\
\alpha_2^*(s)= p_2+x^*_1p_3
\end{array}\right.
$$
where $p$ is the costate associated to $(x^*, \alpha^*)$.
\item[3.] The functions  $x^*$  and $\alpha^*$ are of class $C^1$. In particular equations \eqref{tag:12d} and \eqref{tag:alpha*d} hold for every $s\in [t,T]$.
\item[4.] Assume that,  for some $k\in\mathbb N$,
$D_xf\in C^k$.
Then, the costate~$p$ and the control~$\alpha^*$ are of class $C^{k+1}$ and $x^*$ is of class $C^{k+2}$.
 \end{itemize}
\end{corollary}
\begin{proof}
The proof follows the same lines as in \cite[Corollary 2.1]{MMMT} and we refer to that paper for the detailed arguments.
\end{proof}
\begin{remark}
The uniqueness of the optimal trajectories after the initial time for a.e. initial data is an open problem.  In \cite{MMMT} this result was obtained thanks to the property $meas\{x: \det B(x)B^T(x)=0\}=0$; now, in the Heisenberg setting, this property fails to be true since $\det B(x)B^T(x)= 0$ for any $x\in \He^1$.
\end{remark}

\subsection{The Hamilton-Jacobi equation and the value function of the optimal control problem}\label{vf}
The aim of this section is to study the Hamilton-Jacobi equation~\eqref{eq:MFGintrin}-(i) with $m$ fixed, namely
\begin{equation}\label{eq:HJ1}
\left\{\begin{array}{ll}
-\partial_t u+\frac12 |D_{\mathcal H}u|^2=f(x, t)&\qquad \textrm{in }\He^1\times (0,T),\\
u(x,T)=g(x)&\qquad \textrm{on }\He^1.
\end{array}\right.
\end{equation}
Under Hypothesis~\ref{BasicAss}, we shall prove Lipschitz continuity and semiconcavity of $u$.
As a first step, in the next lemma we show that the solution $u$ of \eqref{eq:HJ1} can be represented as the value function of the control problem defined in \eqref{repr}. Hence from Lemma \ref{uper} we can restrict to study equation~\eqref{eq:HJ1} in $\Te$.
\begin{lemma}\label{valuefunction}
Under Hypothesis \ref{BasicAss}, the value function $u$, defined in~\eqref{repr}, is the unique continuous bounded viscosity solution to problem~\eqref{eq:HJ1}. Moreover $u$ is $Q_{\cH}$-periodic.
\end{lemma}
\begin{proof}
The proof comes from classical results in viscosity theory, see for example \cite[Proposition 3.5]{BCD}, \cite[Theorem 3.1]{BCP} and \cite[Corollary 2.1]{DL}.
\end{proof}
In the following lemma we prove the Lipschitz continuity in both variables $x$ and $t$ of the value function. 
\begin{lemma}\label{L1}
Under Hypothesis~\eqref{BasicAss}, $u(x,t)$ is Lipschitz continuous with respect to the spatial variable~$x$,
and the time variable $t$.
\end{lemma}
\begin{proof}  In this proof, $C_T$ will denote a constant which may change from line to line but it always depends only on the constants in the assumptions (especially the Lipschitz constants of $f$ and $g$) and on~$T$.\\
We study first the Lipschitz continuity w.r.t. $x$.
Let $t$ be fixed. We follow the proof of \cite[Lemma 4.7]{C}.
From Remark \ref{2.3} we know that there exists
$\alpha(\cdot)$ optimal control for $u(x,t)$ and $x(\cdot)$ optimal trajectory i.e.:
\begin{equation}
\label{eq:HJ31}
u(x_1, x_2, x_3, t)=\int_t^T\frac12 |\alpha(s)|^2+f(x(s),s)\,ds+g(x(T)).
\end{equation}
We consider the path $x^*(s)$ starting from $y=(y_1,y_2, y_3)$, with control $\alpha$.
Hence
\begin{eqnarray*}
x_1^*(s)&=&y_1+\int_t^s\alpha_1(\tau) \,d\tau=y_1-x_1+x_1(s)\\
x_2^*(s)&=&y_2+\int_t^s \alpha_2(\tau)\,d\tau=y_2-x_2+x_2(s)\\
x_3^*(s)&=&y_3-\int_t^s\alpha_1(\tau)x_2^*(\tau)\,d\tau+
\int_t^s\alpha_2(\tau)x_1^*(\tau)\,d\tau\\
&=&y_3- (y_2-x_2)\int_t^s\alpha_1(\tau)\,d\tau+
(y_1-x_1)\int_t^s\alpha_2(\tau)\,d\tau\\
&&+\int_t^s (-\alpha_1(\tau)x_2(\tau)+\alpha_2(\tau)x_1(\tau))\,d\tau\\
&=&x_3(s)+(y_3-x_3)- (y_2-x_2)\int_t^s\alpha_1(\tau)\,d\tau+
(y_1-x_1)\int_t^s\alpha_2(\tau)\,d\tau.
\end{eqnarray*}
Using the Lipschitz continuity of $f$ we get
\begin{multline*}
f(x^*(s),s)
\leq f(x(s), s)
+L(|y_1-x_1|+ |y_2-x_2| + |y_3-x_3|+\\ +|y_2-x_2|\sqrt{s-t}\|\alpha_1\|_2 +|y_1-x_1|\sqrt{s-t}\|\alpha_2\|_2)\end{multline*}
and from the $L^2$ uniform estimate for $\alpha_1$ and $\alpha_2$ in \eqref{bd:alpa_L2} we get 
$$f(x^*(s),s)-f(x(s), s)\leq C_T(\vert y_1-x_1\vert+\vert y_2-x_2\vert +\vert y_3-x_3\vert).$$

By the same calculations for $g$ and substituting equality~\eqref{eq:HJ31} in
\begin{equation*}
u(y_1, y_2, y_3, t)\leq \int_t^T\frac12 |\alpha(s)|^2+f(x^*(s),s)\,ds+g(x^*(T)),
\end{equation*}
we get
\begin{equation*}
u(y_1, y_2, y_3, t)\leq u(x_1, x_2, x_3, t)+C_T(\vert y_1-x_1 \vert +\vert y_2-x_2\vert+ \vert y_3-x_3\vert).
\end{equation*} 
Reversing the role of $x$ and $y$ we get the result.\\
Let us now prove the Lipschitz continuity of~$u$ w.r.t.~$t$. Thanks to the $Q_\cH$-periodicity in~$x$ of~$u$, it is enough to prove the Lipschitz continuity in~$t$ only for~$x\in Q_\cH$. To this end, taking advantage of the $L^\infty$-bound for optimal controls established in Proposition \ref{boundalfa}, we can follow the same arguments as those in the proof of \cite[Lemma 4.7]{C}, noting that
\begin{equation*}
|x(s)-x|\leq C(s-t)(\|\alpha_1\|_{\infty}|x_2|+ \|\alpha_2\|_{\infty}|x_1|)
\leq K(s-t).
\end{equation*}
\end{proof}

In the following lemma we establish the semiconcavity of~$u$ w.r.t. $x$; we recall here below the definition of semiconcavity with linear modulus and we refer the reader to the monograph \cite{CS} for further properties.

\begin{definition}\label{SMC}
Let $u:\re^d\to\re$. We say that $u$ is {\em {semiconcave}} (with linear modulus)
if there exists a constant $C\geq 0$ such that  for all $\lambda\in [0,1]$,
$$\lambda u(y)+(1-\lambda)u(x)-2u(\lambda y+(1-\lambda)x)\leq C\lambda(1-\lambda)|y-x|^2\qquad \forall x, y\in \re^d.$$
\end{definition}

\begin{lemma}\label{semic}
Under Hypothesis~\ref{BasicAss}, the value function~$u$, defined in \eqref{repr}, is semiconcave with respect to the variable~$x$ in~$Q_\cH$ with a semiconcavity constant depending only on the constant~$C$ of hypothesis~\ref{BasicAss}.
\end{lemma}
\begin{proof}
For any~$x,y\in Q_\cH $ and $\lambda\in[0,1]$, consider~$x_{\lambda}:=\lambda x+(1-\lambda)y$. Let $\alpha(s)$ and $x_{\lambda}(s)$ be an optimal control and respectively the corresponding optimal trajectory for~$u(x_{\lambda}, t)$; for $s\in[t,T]$ there holds

\begin{eqnarray*}
x_{\lambda,i}(s)&=&x_{\lambda,i}+\int_t^s\,\alpha_i(\tau)\,d\tau,\qquad i=1,2\\
x_{\lambda,3}(s)&=&x_{\lambda,3}-\int_t^s\alpha_1(\tau)x_{\lambda,2}(\tau)\,d\tau+
\int_t^s\alpha_2(\tau)x_{\lambda,1}(\tau)\,d\tau.
\end{eqnarray*}
Let $x(s)$ and $y(s)$ satisfy \eqref{DYNH} with initial condition respectively $x$ and $y$ still with the same control $\alpha$, optimal for~$u(x_{\lambda}, t)$.
We have to estimate~$\lambda u(x,t) +(1-\lambda)u(y,t)$ in terms of $u(x_{\lambda}, t)$. To this end, arguing as in the proof of~\cite[Lemma 4.7]{C},  we have to estimate the terms $\lambda f(x(s), s) +(1-\lambda)f(y(s), s)$ and $\lambda g(x(T))+(1-\lambda)g(y(T)).$\\
We explicitly provide the calculations for the third component $x_3(s)$ since the calculations for $x_1(s)$ and $x_2(s)$ are the same as in \cite{C}.
We have
\begin{eqnarray*}
x_3(s)&=&x_3-\int_t^s\alpha_1(\tau)x_2(\tau)\,d\tau+
\int_t^s\alpha_2(\tau)x_1(\tau)\,d\tau\\
&=&x_3-x_{\lambda,3}+x_{\lambda,3}(s)- \int_t^s\alpha_1(\tau)(x_2(\tau)-x_{\lambda,2}(\tau))\,d\tau+
\int_t^s\alpha_2(\tau)(x_1(\tau)-x_{\lambda,1}(\tau))\,d\tau.
\end{eqnarray*}
Since $x_3-x_{\lambda,3}=(1-\lambda)(x_3-y_3)$ and
\begin{equation}\label{diff1}x_i(\tau)-x_{\lambda,i}(\tau)=(1-\lambda)(x_i-y_i)\qquad \textrm{for }i=1,2,
\end{equation}
we get
\begin{equation}\label{diff}
x_3(s)-x_{\lambda,3}(s)=(1-\lambda)\left[x_3-y_3- (x_2-y_2)\int_t^s\alpha_1(\tau)d\tau+ (x_1-y_1)\int_t^s\alpha_2(\tau)d\tau\right].
\end{equation}
Analogously for $y(s)$: since $y_3-x_{\lambda,3}=\lambda(y_3-x_3)$ and
\begin{equation}\label{diff2}
y_i(\tau)-x_{\lambda,i}(\tau)=\lambda(y_i-x_i)\qquad\textrm{for }i=1,2,
\end{equation}
we get
\begin{equation}\label{diff3}
y_3(s)-x_{\lambda,3}(s)=\lambda\left[(y_3-x_3)+ (x_2-y_2)\int_t^s\alpha_1(\tau)d\tau- (x_1-y_1)\int_t^s\alpha_2(\tau)d\tau\right].
\end{equation}

For the sake of brevity we provide the explicit calculations only for $f$ omitting the analogous ones for $g$; and we write $f(x_1,x_2, x_3):=f(x_1, x_2, x_3, s)$. 
We have
\begin{equation*}
\begin{array}{l}
\lambda f(x(s))+(1-\lambda)f(y(s))=\\
 \lambda f(x_1(s),x_{2}(s), x_{\lambda,3}(s)+(1-\lambda)(x_3-y_3- (x_2-y_2)\int_t^s\alpha_1(\tau)d\tau+ (x_1-y_1)\int_t^s\alpha_2(\tau)d\tau))+\\
+(1-\lambda)f(y_1(s), y_{2}(s), x_{\lambda,3}(s)+\lambda(y_3-x_3+ (x_2-y_2)\int_t^s\alpha_1(\tau)d\tau- (x_1-y_1)\int_t^s\alpha_2(\tau)d\tau).
\end{array}
\end{equation*}
Since for $i=1,2$ there holds
$$\lambda \partial_{x_i}f(x_{\lambda}(s))(x_i(s)-x_{\lambda,i}(s))+(1-\lambda) \partial_{x_i}f(x_{\lambda}(s))(y_i(s)-x_{\lambda,i}(s))=0,$$ 
the Taylor expansion of $f$ centered in  $x_{\lambda}(s)$ gives:
\begin{multline*}
\lambda f(x(s))+(1-\lambda)f(y(s))=\\
\lambda(f(x_{\lambda}(s))+ Df(x_{\lambda}(s))(x(s)-x_{\lambda}(s))
+R_1)+ 
(1-\lambda)(f(x_{\lambda}(s))+ Df(x_{\lambda}(s))(y(s)-x_{\lambda}(s)) +R_2)\\
=\lambda\left(f(x_{\lambda}(s))+ \partial_{x_3}f(x_{\lambda}(s))
(1-\lambda)(x_3-y_3- (x_2-y_2)\int_t^s\alpha_1(\tau)d\tau + (x_1-y_1)\int_t^s\alpha_2(\tau)d\tau) +
R_1\right)\\
+(1-\lambda)\bigg(f(x_{\lambda}(s))+ \partial_{x_3}f(x_{\lambda}(s))
\lambda(y_3-x_3+ (x_2-y_2)\int_t^s\alpha_1(\tau)d\tau- (x_1-y_1)\int_t^s\alpha_2(\tau)d\tau) +R_2\bigg)=\\
=f(x_{\lambda}(s))+ \lambda R_1+(1-\lambda)R_2,
\end{multline*}
where $R_1$ and $R_2$ are the error terms of the expansion, namely
\begin{multline*}
\lambda R_1+(1-\lambda)R_2=
\frac{1}{2}\lambda ( (x(s)-x_{\lambda}(s))D^2f(\xi_1)(x(s)-x_{\lambda}(s))^T\\
+\frac{1}{2}(1-\lambda)( (y(s)-x_{\lambda}(s))D^2f(\xi_2)(y(s)-x_{\lambda}(s))^T,
\end{multline*}
for suitable $\xi_1, \xi_2\in Q_\cH$.

Using relations~\eqref{diff1}-\eqref{diff3} and the $L^2$ uniform estimate of $\alpha$ in \eqref{bd:alpa_L2}, we obtain
\begin{equation*}
\left\{\begin{array}{ll}
|x_i(s)-x_{\lambda,i}(s)|\,|x_j(s)-x_{\lambda,j}(s)|\leq C(1-\lambda)^2|x-y|^2&\qquad i,j=1,2,3\\
|y_i(s)-x_{\lambda,i}(s)|\,|y_j(s)-x_{\lambda,j}(s)|\leq C\lambda^2|x-y|^2&\qquad i,j=1,2,3
\end{array}\right.
\end{equation*}
for some positive constant $C$. Then, possibly increasing $C$, we get
$$
\lambda R_1+(1-\lambda)R_2\leq C\lambda(1-\lambda) |x-y|^2,
$$
and, in particular,
$$\lambda f(x(s))+(1-\lambda)f(y(s))\leq f(x_{\lambda}(s))+ C\lambda(1-\lambda) |x-y|^2$$
which amounts to the semiconcavity of $u$.
\end{proof}
We state the optimal synthesis principle:
\begin{lemma}\label{BB}
Let $x(\cdot)$ be an absolutely continuous function such that~$x(t)=x\in\He^1$
and for almost every $s\in (t,T)$,
\begin{equation}
  \label{eq:3}
u(\cdot,s) \hbox{ is $\cH$-differentiable at } x(s),
\end{equation}
(see Definition~\ref{def:Hdiff} in Appendix~\ref{app:Hdiff} for the precise definition of $\cH$-differentiability and some of its properties) and $x(\cdot)$ satisfies the ODE
\begin{equation}\label{OS}
x'(s)=-D_\cH u(x(s),s)B^T(x(s)), \qquad\textrm{a.e. }s\in(t,T)
\end{equation}
where $u$ is the value function defined in~\eqref{repr}.
Then the control law $\alpha(s)$, given by 
\begin{equation}\label{OS2}
\alpha(s)=-D_\cH u(x(s),s),
\end{equation}
is optimal for $u(x,t)$.
\end{lemma}

\begin{proof}
We adapt the arguments of \cite[Lemma 3.6]{MMMT} and \cite[Lemma 4.11]{C}. 
Fix $(x,t)\in\He^1\times (0,T)$ and consider an absolutely continuous solution~$x(\cdot)$ to~\eqref{OS}; note that this implies that $D_\cH u$ exists at $(x(s),s)$ for a.e. $s\in (t,T)$. 
We claim that $x(\cdot)$ is Lipschitz continuous. 
Indeed system \eqref{OS} reads
\begin{equation}\label{OSexpl}
\left\{
\begin{array}{l}
x_1'(s)=-X_1u(x(s),s)\\
x_2'(s)=-X_2u(x(s),s)\\
x_3'(s)= x_2(s)X_1u(x(s),s)-x_1(s)X_2u(x(s),s)
\end{array}\right.
\end{equation}
for a.e. $s\in(t,T)$. By Lemma~\ref{L1} and Lemma~\ref{uper}, there exists $C>0$ such that $\|D_\cH u\|_\infty\leq C$; hence, $x_1(\cdot)$ and $x_2(\cdot)$ are both Lipschitz continuous and, in particular they are also bounded. By the third equation in~\eqref{OSexpl}, we also obtain that $x_3(\cdot)$ is Lipschitz continuous. Hence our claim is proved.\\
Consequently, from the Lipschitz continuity of $u$ and of $x(\cdot)$ we get that
also $u(x(\cdot),\cdot)$ is Lipschitz. For a.e. $s\in(t,T)$ there hold:
$i$) $D_\cH u(x(s),s)$ exists, $ii$) equation \eqref{OS} holds, $iii$) the function~$u(x(\cdot),\cdot)$ admits a derivative at $s$. Fix such a $s$.

The Lebourg Theorem for Lipschitz function (see \cite[Thm 2.3.7]{Cla90} and \cite[Thm 2.5.1]{Cla90}) ensures that, for any $h\in\R$ small, there exists $(y_h,s_h)$ in the segment $((x(s),s), (x(s+h),s+h))$ and $(\xi^h_x,\xi^h_t) \in co D_{x,t}^*u(y_h,s_h)$ such that
\begin{equation}\label{31}
u(x(s+h),s+h)-u(x(s),s)= \xi^h_x\cdot (x(s+h)-x(s)) +\xi^h_t h
\end{equation}
(here, ``$co$'' stands for the convex hull and $D_{x,t}^*u$ is the Euclidean reachable gradient both in $x$ and in $t$).
 The Caratheodory theorem (see \cite[Thm A.1.6]{CS}) guarantees that there exist $(\lambda^{h,i},\xi^{h,i}_x, \xi^{h,i}_t)_{i=1,\dots,5}$ such that $\lambda^{h,i}\geq0$, $\sum_{i=1}^5\lambda^{h,i}=1$, $(\xi^{h,i}_x, \xi^{h,i}_t)\in D_{x,t}^*u(y_h,s_h)$ and $(\xi^h_x,\xi^h_t) = \sum_{i=1}^5\lambda^{h,i}(\xi^{h,i}_x, \xi^{h,i}_t)$. We claim that there holds
\begin{equation}\label{claim:os}
\lim_{h\to 0}\xi^{h,i}_x B(y_h)=D_\cH u(x(s),s) \qquad\forall i=1,\dots,5.
\end{equation}
Indeed, for any $i=1,\dots,5$ fixed, let $\xi$ be any cluster point of $\{\xi^{h,i}_x\}_h$ (which must be finite because $u$ is Lipschitz continuous). Then, by a diagonal extraction, there exist $(x_n,t_n)$ such that~$u$ is differentiable at $(x_n,t_n)$, $(x_n,t_n)\to (x(s),s)$ and $D_xu(x_n,t_n)\to \xi$ as $n\to\infty$.
The results in~\cite[Lemma 4.6]{C}, applied to $w_n(\cdot):=u(\cdot,t_n)$ and $w(\cdot):=u(\cdot, s)$, infer: $\xi\in D^+w(s)$.
Lemma~\ref{lm:Hdiff}-(iii) in the appendix ensures $\xi B(x(s))\in D^+_\cH w(x(s))$; in conclusion, by Proposition~\ref{app:prp3.1.5}, since $w$ is $\cH$-differentiable at $x(s)$, we conclude $\xi B(x(s))=D_\cH w(x(s))=D_\cH u(x(s),s)$ namely our claim~\eqref{claim:os} is completely proved.
In particular, we have
\begin{equation}\label{4.23bis}
\lim_{h\to 0}\xi^{h}_x B(y_h)=D_\cH u(x(s),s).
\end{equation}
On the other hand, since $u$ is a viscosity solution to equation~\eqref{eq:HJ1}, by \cite[Proposition II.1.9]{BCD}, we obtain
\[
- \xi^{h,i}_t +\frac{|\xi^{h,i}_{x}B(y_h)|}{2}^2=f(y_h,s_h);
\]
in particular, as $h\to0$, we deduce
\begin{equation}\label{31bis}
\xi^{h}_t=  \frac12 \sum_{i=1}^5\lambda^{h,i}|\xi^{h,i}_{x}B(y_h)|^2
- f(y_h,s_h)\rightarrow \frac12 |D_\cH u(x(s),s)|^2  - f(x(s),s).
\end{equation}
Dividing \eqref{31} by $h$ and letting $h\to 0$, by equations \eqref{OS}, \eqref{4.23bis} and \eqref{31bis}, we infer
\begin{eqnarray*}
\frac{d}{ds}u(x(s),s)&=&\lim_{h\to0}\xi^h_x\cdot[D_\cH u(x(s),s)B^T(x(s))+ \frac{x(s+h)-x(s)}{h}]\\
&&+\lim_{h\to0}\xi^h_x\cdot[D_\cH u(x(s),s)(B^T(y_h)-B^T(x(s)))]\\
&&-\lim_{h\to0}\xi^h_x\cdot[D_\cH u(x(s),s)B^T(y_h)] +\lim_{h\to0}\xi^h_t\\
&=&-\frac12 |D_\cH u(x(s),s)|^2 - f(x(s),s)\\
&=&-\frac12 |\alpha(s)|^2  - f(x(s),s)\qquad\textrm{a.e. }s\in(t,T)
\end{eqnarray*}
where the last equality is due to our definition of $\alpha$ in~\eqref{OS2}. 
Integrating this equality on $[t,T]$ and taking into account the final datum of~\eqref{eq:HJ1}, we obtain
\[
u(x,t)=\int_t^T\frac{|\alpha(s)|^2}{2}+ f(x(s),s) ds +g(x(T)).
\]
Observe that $x(\cdot)$ satisfies the dynamics~\eqref{DYNH} with the control~$\alpha(\cdot)$ defined in~\eqref{OS2}; therefore, the last equality implies that $x(\cdot)$ is an optimal trajectory with optimal control~$\alpha(\cdot)$ given by~\eqref{OS2}.
\end{proof}

\section{The continuity equation}\label{sect:continuity}

This section is devoted to equation~\eqref{eq:MFGintrin}-(ii), namely
\begin{equation}\label{continuitye}
\left\{
\begin{array}{ll}
\partial_t m-\diver_{\cH}  (m D_{\cH}u)=0 &\qquad \textrm{in }\He^1\times (0,T)\\
m(x,0)=m_0(x) &\qquad \textrm{on }\He^1,
\end{array}\right.
\end{equation}
where $u$ is the solution to problem
\begin{equation}\label{HJe}
\left\{\begin{array}{ll}
-\partial_t u+\frac{|D_{\cH}u|^2}{2}=F[\overline{m}_t](x) &\qquad \textrm{in }\He^1\times (0,T)\\
u(x,T)=G[\overline {m}_T](x)&\qquad \textrm{on }\He^1,
\end{array}\right.
\end{equation}
and the function~$\overline m$ is fixed in $C^{1/4}([0,T],\mathcal P_{per}(\He^1))$.
Let us observe that assumptions (H1)-(H3) and Lemma~\ref{valuefunction} ensure that there is a unique bounded solution~$u$ to~\eqref{HJe} which is moreover~$Q_\cH$-periodic.

Now we deal with the existence, the periodicity and uniform estimates of the solution $m$ of \eqref{continuitye}.
%
%
\begin{theorem}\label{prp:m}
Under assumptions (H1)-(H3), for any $\overline m\in C^{1/4}([0,T],\mathcal P_{per}(\He^1))$, problem~\eqref{continuitye} has a solution $m$ 
in the sense of Definition~\ref{defsolmfg}. Moreover the function $m$ belongs to $C^{1/4}([0,T],\mathcal P_{per}(\He^1))\cap {\mathbb L}^\infty(\He^1\times (0,T))$ and there exist two positive constants $C_0$ and $C_1$   (both independent of~$\overline m$) such that
\begin{equation}\label{stimaC0}
  0\leq m(x,t)\leq C_0 \qquad \forall (x,t)\in\He^1\times (0,T),
\end{equation}
\begin{equation}\label{stima0.5}
{\bf d_{1}}(m_s,m_t)\leq C_1(t-s)^{1/4} \qquad \forall\ 0\leq s\leq t\leq T. 
\end{equation}
\end{theorem}
The proof of this Theorem is postponed at the end of this section. It relies on a suitable adaptation of the arguments of the proof of~\cite[Proposition 3.1]{MMMT} (see also \cite[Theorem 5.1]{C13} and \cite[Theorem 4.20]{C}).

We shall use a vanishing viscosity approach applied to the {\it whole} MFG system in terms of the {\it horizontal Laplacian}~$\Delta_\cH$. We need such ``degenerate'' approximation to ensure that the corresponding solution is still $Q_\cH$-periodic in~$x$.\\
For any $\sigma >0$, we consider the system
\begin{equation}
\label{eq:MFGv}
\left\{
\begin{array}{lll}
&(i)\quad-\partial_t u-\sigma \Delta_\cH u+\frac12 |D_\cH u|^2
=F[\overline {m}_t](x)&\qquad \textrm{in }\He^1\times (0,T),\\
&(ii)\quad \partial_t m-\sigma\Delta_\cH m-\diver_\cH (m D_\cH u)=0&\qquad \textrm{in }\He^1\times (0,T),\\
&(iii)\quad m(x,0)=m_0(x), u(x,T)=G[{\overline {m}}_T](x)&\qquad \textrm{on }\He^1.
\end{array}\right.
\end{equation}
In order to prove Theorem~\ref{prp:m}, it is expedient to establish several properties of the solution~$(u^\sigma,m^\sigma)$ to system~\eqref{eq:MFGv}: the following lemmata collect existence, uniqueness and other properties of $u^\sigma$ and respectively $m^\sigma$.

Let us emphasize some features of equation~\eqref{eq:MFGv}-(ii): the degeneracy of the operator, the unboundedness and the lack of global Lipschitz continuity of the coefficients. These features prevent to apply all the uniqueness result we known in literature. In order to overcome this issue, we shall establish two uniqueness results which are collected in appendix~\ref{app:uniqueness}.
Moreover, $m_0$ is not a probability on~$\He^1$ (but only a nonnegative measure).

For any domain $U\subset \He^1\times[0,T]$, any $k\in\N$ and any $\delta\in(0,1]$, we denote $C^{k+\delta}_\cH(U)$ (resp. $C^{k+\delta}_{\cH, loc}(U)$) the (resp. local) parabolic H\"older space adapted to the vector fields~$X_1$ and~$X_2$ (for instance, see~\cite[Section 4]{BB07} or~\cite[Definition 10.4]{BBLU}).
%
%
\begin{lemma}\label{lm:usigma}
Assume $(H1)-(H3)$ and fix $\overline m\in C^{1/4}([0,T],\mathcal P_{per}(\He^1))$.
The Cauchy problem
\begin{equation} \label{eq:HJv}
\left\{
\begin{array}{ll}
-\partial_t u-\sigma \Delta_\cH u+\frac12 |D_\cH u|^2
=F[\overline {m}](x)&\qquad \textrm{in }\He^1\times (0,T),\\
u(x,T)=G[{\overline {m}}(T)](x)&\qquad \textrm{on }\He^1
\end{array}\right.
\end{equation}
admits exactly one bounded viscosity solution~$u^\sigma$ (with a bound independent of~$\sigma$). Moreover, the function~$u^\sigma$ fulfills the following properties
\begin{itemize}
\item[(i)] $u^\sigma$ is $Q_\cH$-periodic in~$x$, Lipschitz continuous and locally semiconcave in~$x$,
\item[(ii)] there exists a positive constant $C$, independent of $\sigma$ and of $\overline m$, such that:
\[
|D_\cH u^\sigma(x,t)|\leq C\quad\textrm{and}\qquad \Delta_\cH u^\sigma(x,t) \leq C\qquad \forall (x,t)\in \He^1\times[0,T].
\]

\item[(iii)] for every $\tau\in[0,T)$ and $\delta\in(0,1/4]$, there exists a positive constant $C$ (depending on $\tau$, $\delta$ and $\sigma$) such that
\[
\|u^\sigma\|_{C^{2+\delta}_{\cH}(\He^1\times[0,\tau])}+
\sum_{i=1}^2\|X_iu^\sigma\|_{C^{2+\delta}_{\cH}(\He^1\times[0,\tau])}+
\sum_{i,j=1}^2\|X_iX_j u^\sigma\|_{C^{2+\delta}_{\cH}(\He^1\times[0,\tau])}\leq C,
\]
\item[(iv)] the functions~$u^\sigma$ are $1/4$-H\"older continuous in time uniformly in~$\sigma$.
\end{itemize}
\end{lemma}
\begin{proof}
The differential equation in~\eqref{eq:HJv} can be written as
\begin{equation*}
-\partial_t u-\sigma \tr(D^2 u B(x)B(x)^T)+\frac12 |D u B(x)|^2 =F[\overline {m}](x);
\end{equation*}
in particular, it fulfills the assumption for the comparison principle established in~\cite[Theorem 2.1]{DL}. Using $w^\pm(x,t):=\pm C(-t+1)$ as super- and subsolution, we deduce the existence and uniqueness of a viscosity solution~$u^\sigma$ uniformly bounded on $\sigma$, i.e.
there exists $C$ independent on $\sigma$ such that
\begin{equation}\label{boundusigma}
\|u^\sigma\|_{L^{\infty}(\He^1\times[0,T])}\leq C.
\end{equation}
Let us now prove the several properties of~$u^\sigma$.\\
$(i)$. Since the vector fields $X_1$ and $X_2$ are left-invariant and $F[\overline {m}](\cdot)$ and $G[\overline {m}](\cdot)$ are $Q_\cH$-periodic in~$x$, for any $z\in\Z^3$, also the function $w^\sigma(x,t):=u^\sigma(z\oplus x,t)$ is a solution to~\eqref{eq:HJv}. Again the comparison principle in~\cite[Theorem 2.1]{DL} yields $u^\sigma=w^\sigma$, namely $u^\sigma$ is $Q_\cH$-periodic in~$x$.\\
Invoking~\cite[Theorem 2.1]{DL}, we can represent the solution~$u^\sigma$ as the value function of a stochastic optimal control problem:
\begin{equation}\label{reprsigma}
u^\sigma(x,t)=\min \mathbb{E}\bigg(\int_t^T\left[\frac12 |\alpha(\tau)|^2+F[\overline m_\tau](Y(\tau))\right]\,d\tau+g[\overline m_T](Y(T))\bigg)
\end{equation}
where, in $[t,T]$,  $Y(\cdot)$ obeys to a stochastic differential equation
\begin{equation}
\label{1-stoc}
dY=\alpha(t) B(Y_t)^Tdt +\sqrt{2\sigma} B(Y_t) dW_{t},
\end{equation}
where  $Y(t)=x$ and $W_t$ is a standard $3$-dimensional Brownian motion.
Arguing as in~\cite[Theorem 4.20]{C} and following the calculations in the proofs of Lemma~\ref{L1} and Lemma~\ref{semic}, we get the Lipschitz continuity and the local semiconcavity (see \cite[Theorem 4.20 (proof)]{C} for a similar argument).

$(ii)$. 
Taking into account of the representation of $u^\sigma$ \eqref{reprsigma} as the value function of a stochastic optimal control problem, following the procedure used in Lemma \ref{L1} for the deterministic case, we can prove the uniform Lipschitz continuity of $u^\sigma$. Hence $D_\cH u^\sigma$ is uniformly bounded in $Q_\cH$ and by the $Q_\cH$-periodicity of $u^\sigma$ we get the first bound of~$(ii)$.
Still using the representation of $u^\sigma$ \eqref{reprsigma} we can follow the procedure used in Lemma \ref{semic} for the deterministic case, (see also \cite[Lemma 4.1-(c) (proof)]{AMMT}) to get the uniform local semiconcavity of $u^\sigma$, i.e. $D^2u^\sigma\leq CI$. This implies that $\Delta_\cH u^\sigma\leq C(1+x_1^2+x_2^2)$ and using the periodicity of $u^\sigma$ we get the second bound of~$(ii)$.

$(iii)$. We introduce the Cole-Hopf transformation of $u^\sigma$, $w^\sigma(x,t):=\exp\{-u^\sigma(x,t)/(2\sigma)\}$ and we observe that it is bounded and $Q_\cH$-periodic in~$x$ and it fulfills:
\[X_iu^\sigma =-2\sigma \frac{X_iw^\sigma}{w^\sigma}, \qquad
X_i^2u^\sigma = 2\sigma \frac{(X_iw^\sigma)^2}{(w^\sigma)^2}- 2\sigma \frac{X_i^2w^\sigma}{w^\sigma} \qquad (i=1,2).\]
Replacing these relations in~\eqref{eq:HJv}, we infer that $w^\sigma$ is a viscosity solution to the following linear subelliptic parabolic equation
\begin{equation}\label{CHtrans}
-\partial_t w^\sigma -\sigma \Delta_\cH w^\sigma + w^\sigma F[\overline {m}]/(2\sigma)=0;
\end{equation}
by the equivalence between distributional solutions and viscosity solutions established by Ishii~\cite{I95} for the elliptic case but holding also in the evolutive case, we deduce that $w^\sigma$ is also a distributional solution of equation~\eqref{CHtrans}.

We observe that, by its periodicity, the function $F[\overline {m}]$ belongs to $C^{1/4}_{\cH}(\He^1\times[0,T])$. We consider a bounded domain $Q'\subset \He^1$ such that $\overline{Q_\cH}\subset Q'$. Classical results for linear subelliptic operators,~\cite[Theorem 10.7]{BBLU} and~\cite[Theorem 1.1]{BB07} ensure that, for every $\tau\in[0,T)$ and $\delta\in(0,1/4]$, the function~$w^\sigma$ belongs to $C^{2+\delta}_{\cH}(Q'\times[0,\tau])$ and there exists a constant $C$ (depending on~$\tau$ and $\delta$) such that
\begin{equation}\label{primaiii}
\|w^\sigma\|_{C^{2+\delta}_{\cH}(Q_\cH\times[0,\tau])}\leq C.
\end{equation}
Inverting the Cole-Hopf transformation and using \eqref{boundusigma}, we obtain a bound for~$u^\sigma$ as \eqref{primaiii}. Finally, by periodicity of $w^\sigma$, we accomplish the proof of: $\|u^\sigma\|_{C^{2+\delta}_{\cH}(\He^1\times[0,\tau])}\leq C$.\\
Moreover, by assumptions $(H1)$ and~$(H2)$, also the functions $X_iF[\overline {m}]$ and $X_jX_i F[\overline {m}]$ belong to $C^{1/4}_{\cH}(\He^1\times[0,T])$ for $i,j\in\{1,2\}$.
We observe 
\begin{eqnarray}
&&X_1X_2w^\sigma-X_2X_1w^\sigma= 2\partial_{x_3}w^\sigma,
\partial_{x_3}X_iw^\sigma=X_i\partial_{x_3}w^\sigma, \ i=1,2,\label{primo}\\
&&X_1(\Delta_\cH w^\sigma)=\Delta_\cH (X_1w^\sigma)+4X_2\partial_{x_3}w^\sigma, 
X_2(\Delta_\cH w^\sigma)= \Delta_\cH (X_2w^\sigma)-4X_1\partial_{x_3}w^\sigma.\label{secondo}
\end{eqnarray}
First we remark that the function $W_3:=\partial_{x_3}w^\sigma$ is a distributional solution to
\begin{equation*}
\left\{\begin{array}{ll}
-\partial_t W_3 -\sigma \Delta_\cH W_3 + W_3F[\overline {m}]/(2\sigma)=-w^\sigma \partial_{x_3}F[\overline {m}]/(2\sigma) &\ \textrm{in }\He^1\times(0,T)\\
W_3(x,T)= \partial_{x_3} (\exp\{-G[\overline m_T]/(2\sigma)\})&\ \textrm{on }\He^1,
\end{array}\right.
\end{equation*}
hence following the same procedure to obtain \eqref{primaiii}, we get
\begin{equation}\label{partial3}
\|\partial_{x_3}w^\sigma\|_{C^{2+\delta}_{\cH}(\He^1\times[0,\tau])}\leq C.
\end{equation}
Then the functions $W_i:=X_i w^\sigma$, $i=1,2$,
are distributional solution to
\begin{equation}\label{CHtrans_i}
\left\{\begin{array}{ll}
-\partial_t W_1 -\sigma \Delta_\cH W_1 + W_1F[\overline {m}]/(2\sigma)=4\sigma\, X_2\partial_{x_3}w^\sigma
-w^\sigma X_1F[\overline {m}]/(2\sigma) &\ \textrm{in }\He^1\times(0,T)\\
-\partial_t W_2 -\sigma \Delta_\cH W_2 + W_2F[\overline {m}]/(2\sigma)=-4\sigma\, X_1\partial_{x_3}w^\sigma
-w^\sigma X_2F[\overline {m}]/(2\sigma) &\ \textrm{in }\He^1\times(0,T)\\
W_i(x,T)=X_i (\exp\{-G[\overline m_T]/(2\sigma)\})&\ \textrm{on }\He^1.
\end{array}\right.
\end{equation}
The uniqueness of bounded viscosity solutions established in~\cite[Theorem 2.1]{DL} and the result in~\cite{I95} imply the uniqueness of bounded distributional solution of these problems.
Using estimate \eqref{partial3} in system \eqref{CHtrans_i} and repeating the same arguments as before, 
 we get 
$\|X_i u^\sigma\|_{C^{2+\delta}_{\cH}(\He^1\times[0,\tau])}\leq C$ for $i=1,2$.\\
To get the bound for $X_iX_ju^\sigma$ we consider the equation satisfied by 
$W_{ij}:=X_iX_jw^\sigma$, $i,j=1,2$.
We write it explicitly for $W_{11}=X_1^2w^\sigma=X_1W_1$ and $W_{21}=X_2X_1w^\sigma=X_2W_1$, the other cases are similar so we shall omit them.
$W_{11}$ is the distributional solution to
\begin{equation*}
\left\{\begin{array}{l}
-\partial_t W_{11} -\sigma \Delta_\cH W_{11} + W_{11}F[\overline {m}]/(2\sigma)\\
 =4\sigma X_2 \partial_{x_3}W_1+
4\sigma X_1X_2 \partial_{x_3}w^\sigma
 -W_1X_1(F[\overline {m}])/\sigma-w^\sigma X_1^2F[\overline {m}]/(2\sigma), \\
W_{11}(x,T)=X_1^2 (\exp\{-G[\overline m_T]/(2\sigma)\})
\end{array}\right.
\end{equation*}
and $W_{21}$ solves
\begin{equation*}
\left\{\begin{array}{l}
-\partial_t W_{21} -\sigma \Delta_\cH W_{21} + W_{21}F[\overline {m}]/(2\sigma)\\
 =-4\sigma X_1 \partial_{x_3}W_1+
4\sigma X^2_2 \partial_{x_3}w^\sigma
 -W_1X_2(F[\overline {m}])/(2\sigma)-w^\sigma X_2X_1F[\overline {m}]/(2\sigma)-\\
 - X_2w^\sigma X_1(F[\overline {m}])/(2\sigma), \\
W_{21}(x,T)=X_2X_1 (\exp\{-G[\overline m_T]/(2\sigma)\}).
\end{array}\right.
\end{equation*}
Taking into account that $X_i \partial_{x_3}W_1=X_iX_1 W_3$, $i=1,2$ and of \eqref{partial3}, repeating the same arguments as before we get again the uniqueness of bounded distributional solution $W_{ij}$ and
$\|X_jX_i u^\sigma\|_{C^{2+\delta}_{\cH}(\He^1\times[0,\tau])}\leq C$ for $i,j=1,2$.

$(iv)$. We shall follow the arguments of~\cite[Lemma 3.4]{MMMT} (see also~\cite[Theorem 5.1 (proof)]{C13}); hence we only provide the main steps of the proof. By our assumptions on~$G$, there exists a constant~$C_1$, independent of~$\sigma$, such that the functions $w^\pm(x,t):=G[{\overline {m}}_T](x)\pm C_1(T-t)$ are respectively a supersolution and a subsolution to~\eqref{eq:HJv}. The comparison principle in~\cite[Theorem 2.1]{DL} entails
\begin{equation}\label{CP33}
\sup_x|u^\sigma(x,t)-G[{\overline {m}}_T](x)|\leq C_1(T-t) \qquad\forall t\in[0,T].
\end{equation}
On the other hand, assumption~(H2) and the hypothesis on~$\overline m$ yield
\[
\sup_{t\in[h,T]}\|F[{\overline {m}}_t](x)-F[{\overline {m}}_{(t-h)}](x)\|_\infty\leq C_2h^{1/4}.
\]
We deduce that the functions $v^\sigma_h(x,t):=u^\sigma(x,t-h)+C_1 h+C_2h^{1/4}(T-t)$ is a supersolution to the PDE in~\eqref{eq:HJv} and verifies $v^\sigma_h(x,T)\geq u^\sigma(x,T)$. 
Thanks to \eqref{CP33},
again by comparison principle, we get
\[
u^\sigma(x,t-h)- u^\sigma(x,t)\geq-C_1 h-C_2h^{1/4}(T-t).
\]
The other inequality can be obtained in a similar way and we shall omit its proof.
\end{proof}

\begin{lemma}\label{lm:msigma}
Under assumptions $(H1)-(H3)$ we consider 
\begin{equation} \label{eq:FPv}
\left\{
\begin{array}{ll}
\partial_t m-\sigma\Delta_\cH m-\diver_\cH (m D_\cH u^\sigma)=0&\qquad \textrm{in }\He^1\times (0,T),\\
m(x,0)=m_0(x)&\qquad \textrm{on }\He^1.
\end{array}\right.
\end{equation}
where $u^\sigma$ is the solution to problem~\eqref{eq:HJv} found in Lemma~\ref{lm:usigma} with a fixed $\overline m$. Then, problem~\eqref{eq:FPv} admits exactly one bounded classical solution~$m^\sigma$. Moreover, $m^\sigma$ has the following properties:
\begin{itemize}
\item[(i)]  $m^\sigma$ is $Q_\cH$-periodic and there exists $C_0>0$ (independent of~$\sigma$ and of $\overline m$) such that $0\leq m^\sigma\leq C_0$,
\item[(ii)] for every $\tau\in (0,T]$ and $\delta\in(0,1/4]$, 
there exists $C_1>0$ (depending on $\sigma$, $\tau$ and $\delta$) such that
\[
\|m^\sigma\|_{C^{2+\delta}_\cH(\He^1\times[\tau, T])}\leq C_1.  
\]
\end{itemize}
\end{lemma}
\begin{proof}
We observe that the differential equation in~\eqref{eq:FPv} can be written as
\begin{equation*}
\partial_t m^\sigma-\sigma\Delta_\cH m^\sigma-D_\cH m^\sigma\cdot D_\cH u^\sigma - m^\sigma\Delta_\cH u^\sigma=0.
\end{equation*}
Lemma~\ref{lm:usigma}-(iii) ensures that the coefficients of this linear parabolic equation belong to~$C^{\delta}_\cH(\He^1\times[0,\tau))$ for any~$\delta\in(0,1)$ and~$\tau\in(0,T)$; hence the results in~\cite{BBLU} apply to this equation. In particular, \cite[Theorem 10.7]{BBLU} ensures the existence of a bounded distributional solution~$m^\sigma$ to~\eqref{eq:FPv} with $m^\sigma\in C^{2+\delta}_{\cH,loc}$.
On the other hand, since $m^\sigma$ satisfies assumption \eqref{exponential}, then 
Proposition~\ref{prp:!fpv_reg} in the appendix ensures the uniqueness of a bounded classical solution.
 
Let us now prove the properties of~$m^\sigma$.\\
(i). By the left-invariance of the vector fields generating~$\He^1$ and the $Q_\cH$-periodicity of~$u^\sigma$ (see Lemma~\ref{lm:usigma}-(i)), for any $z\in\Z^3$, the function $\tilde m^\sigma(x,t):=m^\sigma(z\oplus x,t)$ is still a solution to~\eqref{eq:FPv}. Applying again Proposition~\ref{prp:!fpv_reg}, we have $m^\sigma=\tilde m^\sigma$, namely $m^\sigma$ is $Q_\cH$-periodic.\\
Moreover, \cite[Theorem 10.7]{BBLU} establishes that the fundamental solution of~\eqref{eq:FPv} is nonnegative; since $m_0\geq 0$, we get: $m\geq 0$.\\
Let us now prove the upper bound for $m^\sigma$. By Lemma~\ref{lm:usigma}-(ii) and~$m^\sigma\geq 0$, we have
\begin{equation*}
\partial_t m^\sigma-\sigma\Delta_\cH m^\sigma-D_\cH m^\sigma\cdot D_\cH u - C m^\sigma\leq 0
\end{equation*}
for a constant~$C$ independent of $\sigma$ and of $\overline m$. By the $L^\infty$ bound of $m_0$, using again the comparison principle we obtain the statement.\\
(ii). It is enough to invoke the results in~\cite[Theorem 10.7]{BBLU} and in~\cite[Theorem 1.1]{BB07} and to use the periodicity of~$m^\sigma$.
\end{proof}

As for the Euclidean case (for instance, see \cite[Lemma 3.4]{C}) it is expedient to interpret $m^\sigma$ as the law of a suitable stochastic process. In fact, we shall adapt this approach for the present setting where $m_0$ is only a nonnegative measure on $\He^1$ (see assumptions H3) and the coefficients in the SDE are unbounded. To this end, we consider a probability space $(\Omega, {\mathcal F}, P)$, equipped with a filtration~$({\mathcal F}_t)_{t\geq 0}$. For any $x\in \He^1$, we introduce the process
\begin{equation}\label{processMarkus}
d Y^x_t= -D_\cH u (Y^x_t,t) B^T(Y^x_t) dt +\sqrt{2\sigma} B(Y^x_t) d W_t,\qquad Y^x_0=x
\end{equation}
where $B(x)$ is the matrix introduced in~\eqref{Hd} and $W_{\cdot}$ is a standard $2$-dimensional $({\mathcal F}_t)$-adapted Wiener process. 
\begin{remark}
By Lemma~\ref{lm:usigma}-(iii), the drift and the diffusion matrix are locally Lipschitz continuous and have an at most linear growth; hence, by standard theory  (for instance, \cite[Theorem 8.10 pag. 201]{Baldi} or \cite [theorem B.3.1]{BGL}) there exists a unique solution to~\eqref{processMarkus}.
\end{remark}

\begin{remark} The process $Y^x_t$ fulfills the following {\it translation formula}
\begin{equation}\label{transY}
z\oplus Y^x_t=Y^{z\oplus x}_t\qquad \forall z\in\Z^3,\, x\in\He^1, \, t\in [0,T].
\end{equation}
Actually, by~\eqref{processMarkus} and the periodicity of $X_i u$ (see Lemma~\ref{lm:usigma}-(i)), the process $Z_t:=z\oplus Y^x_t$ satisfies $Z_0=z\oplus x$ and
\begin{eqnarray*}
d(Z_t)_i&=&d(Y^x_t)_i=X_iu(Y^x_t,t) dt +\sqrt{2\sigma}d(W_t)_i=X_iu(Z_t,t) dt +\sqrt{2\sigma}d(W_t)_i, \quad (i=1,2)\\
d(Z_t)_3&=&d(Y^x_t)_3+z_1d(Y^x_t)_2-z_2d(Y^x_t)_1\\
&=&[(Z_t)_1 X_2u(Y^x_t,t)-(Z_t)_2 X_1u(Y^x_t,t)]dt +\sqrt{2\sigma}[(Z_t)_1d(W_t)_1-(Z_t)_2d(W_t)_2]\\
&=&[(Z_t)_1 X_2u(Z_t,t)-(Z_t)_2 X_1u(Z_t,t)]dt +\sqrt{2\sigma}[(Z_t)_1d(W_t)_1-(Z_t)_2d(W_t)_2]
\end{eqnarray*}
namely, $Z_t$ solves the SDE in~\eqref{processMarkus}.
\end{remark}

We set
\begin{equation}\label{etaMarkus}
\eta_t^\sigma:=\int_{\He^1}{\mathcal L}(Y^x_t)dm_0(x), \quad  t\in [0,T], 
\end{equation}
where ${\mathcal L}(Y^x_t)$ is the law of the process $Y^x_t$. \\
In the following lemma we shall prove that $\eta_t^\sigma$ is a periodic measure on $\He^1$, so using Remark \ref{rmk:misureper} we shall denote by $\eta_t^\sigma$ also the corresponding probability measure on $Q_{\cH}$.
\begin{lemma}\label{lm:etaMarkus}
The function $\eta_{\cdot}^\sigma:[0,T]\rightarrow {\mathcal M}(\He^1)$ fulfills the following properties:
\begin{itemize}
\item[(i)] $\eta_t^\sigma$ is $\Z^3$-periodic, namely
\[
\eta_t^\sigma(z\oplus A)=\eta_t^\sigma(A) \qquad \forall z\in\Z^3,\ A {\text { Borel set }}, A\subset [0,1)^{3}, t\in[0,T];
\]
\item[(ii)] $\eta_t^\sigma(Q_\cH)=1$ for every $t\in[0,T]$ (i.e., $\eta_t^\sigma\in{\mathcal P}(\Te)$);
\item[(iii)] there exists $C_1>0$, independent of $\sigma\in [0,1)$ and~$\overline m$, such that
\[
{\bf d_{1}}(\eta_t^\sigma, \eta_s^\sigma)\leq C_1(t-s)^{1/4} \qquad \forall 0\leq s\leq t\leq T;
\]
\item[(iv)] $\eta_t^\sigma$ is a distributional solution to~\eqref{eq:FPv}, namely it fulfills
\begin{equation}\label{eq:FPv_distr}
\int_{\He^1}\phi(x,t)\eta^\sigma_t(dx)=\int_{\He^1}\phi(x,0)m_0(x)dx+
\iint_{[0,t]\times\He^1} [\partial_t \phi +\sigma \Delta_\cH\phi-D_\cH u^\sigma\cdot D_\cH \phi]\eta^\sigma_s(dx)ds
\end{equation}
for every $\phi\in C^{2,1}([0,T]\times\He^1)$; moreover it coincides with $m_t^\sigma$.
\end{itemize}
\end{lemma}
\begin{proof}
(i). Consider $z$, $t$ and $A$ as in the statement. By the definition~\eqref{etaMarkus} of~$\eta^\sigma$ and the translation formula~\eqref{transY}, we have
\begin{eqnarray*}
\eta_t^\sigma(z\oplus A)&=&\int_{\He^1}P\left\{Y^x_t\in z\oplus A\right\}dm_0(x)
=\int_{\He^1}P\left\{Y^{(-z)\oplus x}_t\in A\right\}dm_0(x)\\
&=& \int_{\He^1}P\left\{Y^{x'}_t\in A\right\}dm_0(z\oplus x')=
\int_{\He^1}P\left\{Y^{x'}_t\in A\right\}dm_0(x')=\eta_t(A)
\end{eqnarray*} 
where the second-last equality is due to the periodicity of~$m_0$.\\
(ii). By the property of pavage and the periodicity of~$m_0$, we have
\begin{eqnarray*}
\eta^\sigma_t(Q_\cH)&=&\sum_{z\in\Z^3}\int_{z\oplus Q_\cH} P\left\{Y^x_t\in  Q_\cH\right\}dm_0(x)=\sum_{z\in\Z^3}\int_{Q_\cH} P\left\{Y^{z\oplus x'}_t\in  Q_\cH\right\}dm_0(x')\\
&=&\sum_{z\in\Z^3}\int_{Q_\cH} P\left\{Y^{x'}_t\in (-z)\oplus Q_\cH\right\}dm_0(x')\\
&=&\int_{Q_\cH} P\left\{Y^{x'}_t\in\cup_{z\in\Z^3} [(-z)\oplus Q_\cH]\right\}dm_0(x')\\
&=&\int_{Q_\cH} P\left\{Y^{x'}_t\in\He^1\right\}dm_0(x')=1.
\end{eqnarray*}
(iii). First of all observe that, using Remark \ref{rmk:misureper}, we shall denote by $\eta_t^\sigma$ also the corresponding probability measure on $Q_{\cH}$.\\
For each $x\in\He^1$, set
\begin{equation}\label{proMarkusper}
Y^{per,x}_\tau:= q_\cH(Y^{x}_\tau) \qquad \forall \tau\in[0,T]
\end{equation}
where $q_\cH(\cdot)$ is the projection introduced in section~\ref{sub:periodicity}.
Fix $0\leq s\leq t\leq T$ and introduce
\begin{equation*}
\tilde \pi:=\int_{\Te} {\mathcal L}(Y^{per,x}_s,Y^{per,x}_t)dm_0(x)
\end{equation*}
where ${\mathcal L}(Y^{per,x}_s,Y^{per,x}_t)$ is the law of the pair $(Y^{per,x}_s,Y^{per,x}_t)$.
We claim that
\begin{equation}\label{claim:pitilde}
\tilde \pi\in \Pi(\eta_{s}^\sigma,\eta_{t}^\sigma)
\end{equation}
where the set $\Pi$ is the one introduced in~\eqref{Pi}. Let us assume for the moment that this claim is true. Then, by~\eqref{claim:pitilde}, there holds
\begin{eqnarray*}
{\bf d_{1}}(\eta_t^\sigma, \eta_s^\sigma)&\leq&\int_{\Te\times \Te}d_{\Te}(z_1,z_2)\tilde \pi(dz_1,dz_2)=\int_{\Te}\mathbb E[d_{\Te}(Y^{per,x}_s,Y^{per,x}_t)]dm_0(x)\\&\leq&\int_{\Te}\mathbb E[d_{\Te}(Y^{per,x}_s,Y^{per,x}_t)]dm_0(x)
\\
&\leq&\int_{\Te}\mathbb E\left[|Y^{per,x}_s-Y^{per,x}_t|^{1/2}\left(1+2|Y^{per,x}_s|^{1/2}+|Y^{per,x}_s-Y^{per,x}_t|^{1/2}\right)\right]dm_0(x).
\end{eqnarray*}
where the last inequality is due to Remark~\ref{lemmaagosto}. Since now on we denote by~$C$ a constant which may change from line to line but which is independent of $x,s,t,\sigma$. Since $|Y^{per,x}_s|^{1/2},|Y^{per,x}_s-Y^{per,x}_t|^{1/2}\leq \sqrt3$,  we get
\begin{eqnarray}\notag
d_{1}(\eta_t^\sigma, \eta_s^\sigma)&\leq& 
C\int_{\Te}\mathbb E\left[\left|\int_s^t -D_\cH u (Y^x_\tau,\tau) B^T(Y^x_\tau) d\tau +\sqrt{2\sigma} B(Y^x_\tau) d W_\tau\right|^{1/2}\right]dm_0(x)\\ \label{stimad1} &\leq&C\int_{\Te}\mathbb E\left[\left(\int_s^t |D_\cH uB^T| d\tau\right)^{1/2}\right]dm_0(x)+C \sigma^{1/4} \int_{\Te}\mathbb E\left[\left|\int_s^t B d W_\tau\right|^{1/2}\right]dm_0(x).
\end{eqnarray}
By standard theory on SDE (see~\cite[Theorem 8.10 pag.201]{Baldi}), since $\mathbb E[|Y^{x}_0|^2]=|x|^2$ for every $x\in \Te$, we obtain that there exists a positive constant $K$, independent of $\sigma$ and~$\overline m$ (by virtue of Lemma~\ref{lm:usigma}-(ii)), such that:
\begin{equation}\label{stimaL2Markus}
\mathbb E[|Y^{x}_\tau|^2]\leq K\qquad\forall x\in \Te, \quad 0\leq \tau\leq T.
\end{equation}
By Jensen inequality and by Fubini theorem, there holds
\begin{eqnarray}\notag
\int_{\Te}\mathbb E\left[\left(\int_s^t |D_\cH uB^T| d\tau\right)^{1/2}\right]dm_0(x) &\leq&
\int_{\Te}\left(\int_s^t \mathbb E [ |D_\cH u (Y^{x}_\tau,\tau)B^T(Y^{x}_\tau)|] d\tau\right)^{1/2}dm_0(x) \\ \notag
&\leq&\int_{\Te}\left(\int_s^t\mathbb E[1+| Y^{x}_\tau|]d\tau \right)^{1/2}dm_0(x) \\\notag
&\leq&\int_{\Te}\left(\int_s^t\mathbb E[1+| Y^{x}_\tau|^2]d\tau \right)^{1/2}dm_0(x)
\end{eqnarray}
where the last two inequalities are due to Lemma~\ref{lm:usigma}-(ii) and the definition of the matrix~$B$ in~\eqref{Hd} and respectively to the Cauchy-Schwarz inequality.
Using estimate~\eqref{stimaL2Markus} in the previous inequality, we obtain
\begin{equation}\label{stimad1A}
\int_{\Te}\mathbb E\left[\left(\int_s^t |D_\cH uB^T| d\tau\right)^{1/2}\right]dm_0(x) \leq C\sqrt{t-s}.
\end{equation}

On the other hand, by Jensen inequality and Cauchy-Schwarz inequality, we get
\begin{eqnarray*}
\int_{\Te}\mathbb E\left[\left|\int_s^t B d W_\tau\right|^{1/2}\right]dm_0(x)&\leq&
\int_{\Te}\left(\mathbb E\left[\left|\int_s^t B(Y^{x}_\tau) d W_\tau\right|\right] \right)^{1/2}dm_0(x)\\
&\leq&\int_{\Te}\left(\mathbb E \left[\left|\int_s^t B(Y^{x}_\tau) d W_\tau\right|^2\right] \right)^{1/4}dm_0(x)\\
&\leq&\int_{\Te}\left(\mathbb E\left[\int_s^t(1+|Y^{x}_\tau|^2)d\tau \right] \right)^{1/4}dm_0(x)
\end{eqnarray*}
where the last inequality is due to standard calculus for Ito's integral. Using again Fubini theorem and estimate~\eqref{stimaL2Markus} in the previous inequality, we get
\begin{eqnarray}\notag
\int_{\Te}\mathbb E\left[\left|\int_s^t B d W_\tau\right|^{1/2}\right]dm_0(x)&\leq&
\int_{\Te}\left(\int_s^t\mathbb E\left[(1+|Y^{x}_\tau|^2)\right]d\tau  \right)^{1/4}dm_0(x)\\\label{stimad1B}
&\leq& C(t-s)^{1/4}.
\end{eqnarray}
Replacing estimates~\eqref{stimad1A} and \eqref{stimad1B} in~\eqref{stimad1}, taking $\sigma \in[0,1)$, we obtain the statement.

It only remains to prove our claim~\eqref{claim:pitilde}: for any measurable subset $A\subset \Te$, arguing as in proof of point~(i) and using the property of pavage, we have
\begin{eqnarray*}
\tilde \pi(A\times \Te)&=& \int_{\Te} P\{Y^{per,x}_s\in A\}dm_0(x)
=\sum_{z\in\Z^3} \int_{\Te} P\{Y^{x}_s\in z\oplus A\}dm_0(x)\\
&=& \sum_{z\in\Z^3} \int_{z\oplus Q_\cH} P\{Y^{x'}_s\in A\}dm_0(x')=\int_{\He^1} P\{Y^{x'}_s\in A\}dm_0(x') \\
&=& \eta_s^\sigma(A)= \eta_{s\mid \Te}^\sigma(A);
\end{eqnarray*}
analogously, we have $\tilde \pi(\Te \times A)=\eta_{t\mid \Te}(A)$. Hence, our claim~\eqref{claim:pitilde} is completely proved.

(iv). The former part of the statement is due to a standard application of Ito's formula as in the Euclidean setting (for instance, see~\cite[Lemma 3.3]{C} and also~\cite[Theorem 5.7.6]{KS}). The latter part of the statement is an immediate consequence of Proposition~\ref{prp:!FPv_distr} in the appendix with $b=-D_\cH u$ and $c=-\Delta_\cH u$ and of Lemma~\ref{lm:usigma}-(iii).
\end{proof}

\begin{Proofc}{Proof of Theorem~\ref{prp:m}}
We shall follow the arguments of the proof of \cite[Theorem 5.1]{C13} (see also \cite[Theorem 4.20]{C}).\\
By the estimates in Lemma~\ref{lm:usigma}-(ii) and (iv), possibly passing to a subsequence (that we still denote by~$u^\sigma$), as $\sigma\to 0^+$, the sequence~$\{u^\sigma\}_\sigma$ uniformly converges to the function~$u$ which solves~\eqref{HJe}, is $1/4$-H\"older continuous in time and horizontally Lipschitz continuous in space, with $D_\cH u^\sigma \to D_\cH u$ a.e. (by \cite[Theorem 3.3.3]{CS}).

On the other hand, since $\mathcal P_{per}(\He^1)$ can be identified with the space of probabilities on the compact set~$\Te$, the estimates for $m^\sigma$ in Lemma~\ref{lm:etaMarkus}-(iii) and in Lemma~\ref{lm:msigma}-(i)  ensure that, as $\sigma\to 0^+$, possibly passing to a subsequence, $\{m^\sigma\}_\sigma$ converges to some $m\in C^{1/4}([0,T],\mathcal P_{per}(\He^1))$ in the $C^{0}([0,T],\mathcal P_{per}(\He^1))$-topology and in the $\mathbb L^\infty(\He^1\times(0,T))$-weak-$*$ topology; $m$ satisfies \eqref{stimaC0} with the same constant $C_0$ of Lemma~\ref{lm:msigma}-(i)
and \eqref{stima0.5} with the same constant $C_1$ of Lemma~\ref{lm:etaMarkus}-(iii).
In conclusion, we accomplish the proof arguing as in \cite[Proposition 3.1(proof)]{MMMT}.
\end{Proofc}

\section{Proof of Theorem \ref{thm:main}}\label{sect:dim3.1}

\begin{Proofc}{Proof of Theorem \ref{thm:main}}\\
(i) 
Consider the set
\[
{\cal C} :=\left\{m\in C^{1/4}([0,T]; {{\cal P}_{per}(\He^1)}): \textrm{$m$ fulfills \eqref{stimaC0}-\eqref{stima0.5} and $m(0)=m_0$} \right\}
\]
endowed with the $C^0([0,T]; {{\cal P}_{per}(\He^1)})$-topology.
Observe that it is a nonempty convex subset of~$C^0([0,T]; {{\cal P}_{per}(\He^1)})$; moreover, by Ascoli-Arzela theorem, it is also compact.
We introduce the set valued map~${\cal T}$ on~${\cal C}$ as follows: for any $\overline m\in {\cal C}$,  we set
\[{\cal T}(\overline m):=\left\{m\in C^{1/4}([0,T]; {\cal P}_{per}(\He^1)):\begin{array}{l}\textrm{$m$ solves \eqref{continuitye} (associated to $\overline m$ through \eqref{HJe})}\\\textrm{and fulfills \eqref{stimaC0}-\eqref{stima0.5}} \end{array}\right\}.
\]  
Let us assume for the moment that the map~${\cal T}$ admits a fixed point~$m$; let $u$ be the corresponding solution to \eqref{HJe} (i.e., the solution to \eqref{HJe} with $\overline m$ replaced by $m$). Then, by the results in Section \ref{OC} and in Section \ref{sect:continuity}, the couple $(u,m)$ is a solution to~\eqref{eq:MFGintrin}.\\
Let us prove the existence of such a fixed point applying the Kakutani's Theorem.
Note that here we cannot use Schauder's theorem as in \cite[Theorem 4.1 (proof)]{C} because we do not have uniqueness of the solution to \eqref{continuitye}. 
We observe that Theorem~\ref{prp:m} ensures $\emptyset\ne{\cal T}(\overline m)\subseteq {\cal C}$, for any $\overline m\in {\cal C}$.
Moreover, ${\cal T}(\overline m)$ is a convex set by the linearity of \eqref{continuitye}.
We claim that ${\cal T}$ has closed graph. Indeed, let us consider $\overline m_n, \overline m\in {\cal C}$
with $\overline m_n\rightarrow \overline m$ in the $C^0([0,T]; {{\cal P}_{per}(\He^1)})$-topology and 
$m_n\in {\cal T}(\overline m_n)$ with $m_n\rightarrow m$ in the $C^0([0,T]; {{\cal P}_{per}(\He^1)})$-topology; we want to prove that $m\in {\cal T}(\overline m)$.
By the periodicity and the bounds in assumptions (H1) and (H2), possibly passing to a subsequence (that we still denote $\overline m_n$),  Ascoli-Arzela theorem guarantees that $F[\overline m_n]$ and $G[\overline m_n(T)]$ converge uniformly to $F[\overline m]$ in  $\Te\times [0,T]$ and, respectively, to $G[\overline m(T)]$ in $\Te$. Moreover, Lemma \ref{valuefunction} and Lemma \ref{L1} ensure that the solutions $u_n$ to problem \eqref{HJe} with $\overline m$ replaced by $\overline m_n$ are $Q_\cH$-periodic, uniformly bounded and uniformly Lipschitz continuous. By standard stability results for viscosity solutions, the sequence $\{u_n\}_{n}$ converges uniformly to the viscosity solution $u$ to problem \eqref{HJe}.
Moreover, by Lemma \ref{semic}, the functions $u_n$ are uniformly semiconcave with a semiconcavity constant depending only on the constant~$C$ in assumption~(H2); hence by \cite[Theorem 3.3.3]{CS} $Du_n$ converges a.e. to $Du$.
On the other hand, by definition of ${\cal T}$, the functions $m_n\in {\cal T}(\overline m_n)$ are uniformly bounded and uniformly $1/4$-H\"older continuous, so by Ascoli-Arzela theorem and Banach-Alaoglu Theorem, there exists a subsequence $\{m_{n_k}\}_k$ which converges to $m$ in the $C^0([0,T]; {{\cal P}_{per}(\He^1)})$-topology and in the $\mathbb L^\infty(\Te\times[0,T])$-weak-$*$ topology. Being a solution to \eqref{continuitye} with $\overline m$ replaced by $\overline m_{n_k}$ in problem \eqref{HJe}, the function $m_{n_k}$ fulfills
\begin{equation}
\int_0^T\int_{\mathbb{H}^{1}}m_{n_k}(-\partial_t\varphi+D_{\cH}u_{n_k}\cdot D_{\cH}\varphi)dxdt=0\qquad \forall \varphi\in C_{c}^{\infty}(\mathbb{H}^{1}\times (0,T)).
\end{equation}
Passing to the limit as $k\rightarrow +\infty$ we get that $m$ is a solution to \eqref{continuitye}. Moreover again by the uniform convergence and the uniform $1/4$-H\"older continuity of $m_{n_k}$, we have that $m$ satisfies the bounds \eqref{stimaC0}-\eqref{stima0.5}. In conclusion $m\in {\cal T}(\overline m)$ and our claim is proved.
Then, Kakutani's Theorem guarantees the existence of a fixed point for~${\cal T}$, namely a solution to~\eqref{continuitye}.

(ii) Consider the function $m$ found in point (i). Since $t\rightarrow m_t$ is narrowly continuous, applying Theorem \ref{821}, we get that there exists a probability measure $\eta^*$ in $\Te\times\Gamma$ which satisfies points (i) and (ii) of Theorem \ref{821}.
We denote $\eta\in {\mathcal P}(\Gamma)$ the measure on $\Gamma$ defined as $\eta(A):=\eta^*(\Te\times A)$ for every $A\subset \Gamma$ measurable. We claim that $\eta$ is a MFG equilibrium. Indeed, by \eqref{eta}, we have $e_0\#\eta=m_0$ and $e_t\#\eta\in\mathcal P_{per}(\He^1)$,
so $\eta\in {\mathcal P}_{m_0}(\Gamma)$. On the other hand, by  \eqref{dis}, $\eta$ is supported on the curves solving \eqref{ODE}. From Lemma \ref{BB} such curves are optimal, i.e. belong to the set $\Gamma^{\eta}[x]$, hence our claim is proved.\\
Let us now prove that $(u,m)$ is a mild solution. By \eqref{eta}, we have $m_t=e_t\#\eta$. Moreover, by Lemma \ref{valuefunction}, the function $u$ found in point~(i) is the value function associated to $m$  as in Definition \ref{mild}-(ii). In conclusion $(u,m)$ is a mild solution to \eqref{eq:MFGintrin}.
\end{Proofc}
Let us provide the sketch of a different proof of Theorem~\ref{thm:main}-(i). 
\begin{Proofc}{Alternative proof of Theorem \ref{thm:main}-($i$)}
We divide the proof in two steps: in the former one, we obtain a solution to the MFG system with viscosity terms
\begin{equation}
\label{eq:MFGvv}
\left\{
\begin{array}{lll}
&(i)\quad-\partial_t u^\sigma-\sigma \Delta_\cH u^\sigma+\frac12 |D_\cH u^\sigma|^2
=F[m^\sigma_t](x)&\qquad \textrm{in }\He^1\times (0,T),\\
&(ii)\quad \partial_t m^\sigma-\sigma\Delta_\cH m^\sigma-\diver_\cH (m^\sigma D_\cH u^\sigma)=0&\qquad \textrm{in }\He^1\times (0,T),\\
&(iii)\quad m^\sigma(x,0)=m_0(x), u^\sigma(x,T)=G[m^\sigma_T](x)&\qquad \textrm{on }\He^1.
\end{array}\right.
\end{equation}
while in the latter one we get a solution to~\eqref{eq:MFGintrin} letting $\sigma\to 0^+$.

Step $1$. We claim that, for each $\sigma >0$, problem~\eqref{eq:MFGvv} admits a solution $(u^\sigma, m^\sigma)$ such that: the functions~$u^\sigma$ are bounded and fulfill the properties in Lemma~\ref{lm:usigma}-(i) uniformly in~$\sigma$ while the functions~$m^\sigma$ fulfill \eqref{stimaC0} and \eqref{stima0.5} uniformly in $\sigma$. Indeed, let ${\cal C}$ be the set introduced in the previous proof, still endowed with the topology of~$C^0([0,T]; {{\cal P}_{per}(\He^1)})$. For any $\sigma>0$, consider the map $\overline{\cal T}:{\cal C}\rightarrow {\cal C}$ defined by $\overline{\cal T}(\overline m)=m$ where $m$ is the solution to~\eqref{eq:FPv} (where $u^\sigma$ solves problem~\eqref{eq:HJv}). By Lemma~\ref{lm:msigma} and Lemma~\ref{lm:etaMarkus}, the function~$m$ is uniquely determined and belongs to~${\cal C}$ so the map~$\overline{\cal T}$ is well defined. Assume for the moment that the map~$\overline{\cal T}$ is continuous. Since ${\cal C}$ is nonempty, convex and compact, by Schauder fixed point theorem, the map $\overline{\cal T}$ admits a fixed point $m^\sigma$. Let $u^\sigma$ be the solution to problem~\eqref{eq:HJv} with $\overline m$ replaced by $m^\sigma$. One can easily check that the couple $(u^\sigma,m^\sigma)$ solves~\eqref{eq:MFGvv} and fulfill the desired bounds.\\
It remains to prove that $\overline{\cal T}$ is continuous. For simplicity, we drop the superscript ``$\sigma$'' because it is fixed. To this end, let $\{\overline m_n\}_n$ be a sequence of functions in ${\cal C}$ such that, as $n\to+\infty$, $\overline m_n\to \overline m\in{\cal C}$ in the $C^0([0,T]; {{\cal P}_{per}(\He^1)})$-topology. We want to prove that $m_n=\overline{\cal T}(\overline m_n)$ converges to $m=\overline{\cal T}(\overline m)$ in the $C^0([0,T]; {{\cal P}_{per}(\He^1)})$-topology. Let us assume by contradiction that there exists a subsequence $\overline{\cal T}(\overline m_{n_k})$ which does not converge to $m$ as $k\to+\infty$.
Since ${\cal C}$ is compact, possibly passing to a subsequence (still denoted $\overline{\cal T}(\overline m_{n_k})$), we can assume that $\overline{\cal T}(\overline m_{n_k})$ converges in the $C^0([0,T]; {{\cal P}_{per}(\He^1)})$-topology to some function $\tilde m\ne m$. As in the proof above, $F[\overline m_{n_k}]$ and $G[\overline m_{n_k}(T)]$ converge uniformly to $F[\overline m]$ in  $\Te\times [0,T]$ and, respectively, to $G[\overline m(T)]$ in~$\Te$.
By Lemma \ref{lm:usigma} and by Ascoli-Arzela theorem, (again possibly passing to a subsequence that we still denote $u_{n_k}$), the solution $u_{n_k}$ to~\eqref{eq:HJv} with $\overline m$ replaced by $\overline m_{n_k}$, converges uniformly with their horizontal gradient to a function $u$. By stability results of viscosity solutions, the function~$u$ is the unique viscosity solution to problem~\eqref{eq:HJv}. On the other hand, by Lemma~\ref{lm:msigma} and Lemma~\ref{lm:etaMarkus}, one can pass to the limit in the weak formulation of problem~\eqref{eq:FPv} with $u^\sigma$ replaced by $u_{n_k}$ (whose solution is $\overline{\cal T}(\overline m_{n_k})$) and we get that $\tilde m$ is a weak solution to~\eqref{eq:FPv} with $u^\sigma$ replaced by the solution $u$ to~\eqref{eq:HJv}. By the uniqueness result in Proposition~\ref{prp:!FPv_distr} we get $m=\tilde m$ which is the desired contradiction.

Step 2. By the bounds of step~$1$ and by Ascoli-Arzela theorem, as $\sigma\to 0^+$, (possibly passing to a subsequence still denoted $(u^\sigma,m^\sigma)$), we have: $u^\sigma$ uniformly converges to a $Q_\cH$-periodic, Lipschitz continuous bounded function~$u$ while $m^\sigma$ converges to some function~$m\in C^{1/4}([0,T]; {{\cal P}_{per}(\He^1)})$ in the $C^{0}([0,T]; {{\cal P}_{per}(\He^1)})$-topology. By arguments similar to the above ones, $u$ is a viscosity solution to~\eqref{eq:MFGintrin}-(i) while $m$ is a weak solution to~\eqref{eq:MFGintrin}-(ii). 
\end{Proofc}
\begin{remark} Differently from \cite{AMMT} and \cite{MMMT}, in this model we cannot obtain the representation of $m$ as the push-forward of $m_0$ by the flow associated to the optimal control problem. This is due to the fact that we cannot prove a uniqueness result of the optimal trajectories and then we cannot say that $\Gamma^{\eta}[x]$ is a singleton, or equivalently that the disintegrated measure $\eta_x$ (see \eqref{dis}) coincides with the Dirac measure $\delta_{\overline \gamma_x}$.
\end{remark}

\section{The non periodic case}\label{sect:nonper}
This section is devoted to system~\eqref{eq:MFG1} without any periodicity condition on the operators and the data. The main result of this section is Theorem~\ref{thm:main_illi} which is the analogue of Theorem~\ref{thm:main} in the non periodic case.
Here we shall follow an approach which relies on the compactness of the initial distribution of players and on the sublinear growth of the coefficients of the matrix~$B$ (see~\eqref{Hd}) but it does not need the H\"ormander condition.
Actually, the growth condition for~$B$ plays a key role only for applying a comparison principle and for representing the solution in terms of a stochastic process.

We first introduce~$\mathcal P_1(\re^3)$ (respectively, $\mathcal P_2(\re^3)$) as the space of Borel probability measures on~$\re^3$ with finite first (resp., second) order moment with respect to the Euclidean distance, endowed with the Monge-Kantorovich distance~${\bf d_{1}}$ (resp.,~${\bf d_{2}}$).



Throughout this section we shall require the following hypotheses, analogous to (H) stated in Section \ref{sect:mainthm} :
\begin{enumerate}
\item[(H1')]\label{H1'} the functions~$F$ and $G$ are real-valued function, continuous on ${\mathcal P}_{1}(\re^3)\times\re^3$,
\item[(H2')]\label{H2'} the map $m\to F[m](\cdot)$ is Lipschitz continuous from~${\mathcal P}_{1}(\re^3)$ to $C^{2}(\re^3)$; moreover, there exists~$C\in \mathbb R$ such that
  $$\|F[m](\cdot)\|_{C^{2}(\re^3)}, \|G[m](\cdot)\|_{C^{2}(\re^3)}\leq C,\qquad \forall m\in {\mathcal P}_{1}(\re^3);$$
\item[(H3')]\label{H4'} the distribution~$m_0:\re^3\to \re$ is a nonnegative $C^{0}$ function with compact support and $\int_{\re^3}m_0dx=1$.
\end{enumerate}
We now introduce our definitions of solution of the MFG system~\eqref{eq:MFGintrin} which are analogous to the corresponding ones in section~\ref{sect:mainthm}  (just dropping any periodicity requirement and replacing ${\mathcal P}_{per}(\He^1)$ with~${\mathcal P}_{1}(\re^3)$).
\begin{definition}\label{defsolmfg_illi}
A couple $(u,m)$ of functions defined on~$\re^3\times [0,T]$ is a solution of system~\eqref{eq:MFGintrin} if it verifies points $1)$-$4)$ of~Definition~\ref{defsolmfg} with ${\mathcal P}_{per}(\He^1)$ replaced by~${\mathcal P}_{1}(\re^3)$.
\end{definition}

\begin{definition}\label{mfgequil_illi}
A measure $\eta$ is a {\em MFG equilibrium} for $m_0$ if it verifies Definition~\ref{mfgequil} with
$$\mathcal P_{m_0}(\Gamma)=\{\eta\in{\mathcal P}(\Gamma):\ m_0 =e_0\# \eta\quad \textrm{and }e_t\# \eta\in {\mathcal P}_{1}(\re^3) \quad \forall t\in[0,T] \}.$$
\end{definition}
\begin{definition}\label{mild_illi}
A couple $(u,m)$ is called {\em {mild solution}} if it verifies Definition~\ref{mild} with ${\mathcal P}_{per}(\He^1)$ replaced by~${\mathcal P}_1(\re^3)$.
\end{definition}
Now we can state the main result for the non periodic case whose proof is postponed at the end of this section.
\begin{theorem}\label{thm:main_illi}
Under the above assumptions:
\begin{enumerate}
\item system \eqref{eq:MFGintrin} has a solution $(u,m)$;
\item $(u,m)$ is a mild solution of the MFG problem.
\end{enumerate}
\end{theorem}

\begin{remark} Uniqueness holds under classical hypothesis on the monotonicity of $F$ and $G$ as in \cite{C}.
\end{remark}

In order to prove Theorem~\ref{thm:main_illi}, it is expedient to introduce the following approximating problems for $\varepsilon\in(0,1]$
%
%
\begin{equation}\label{eq:MFGe}
\left\{\begin{array}{lll}
(i)&\quad-\partial_t u+H^{\varepsilon}(x, Du)=F[m(t)](x)&\qquad \textrm{in }\re^3\times (0,T)\\
(ii)&\quad\partial_t m-\diver  (m\, \partial_pH^{\varepsilon}(x, Du))=0&\qquad \textrm{in }\re^3\times (0,T)\\
(iii)&\quad m(x,0)=m_0(x), u(x,T)=G[m(T)](x)&\qquad \textrm{on }\re^3,
\end{array}\right.
\end{equation}
where
\begin{equation}\label{He}
H^{\varepsilon}(x,p)=\frac{1}{2}|pB^{\varepsilon}(x)|^2\qquad\textrm{with}\qquad
 B^\varepsilon:=  \begin{pmatrix}
      \!\!&1& 0&0&\!\\
     \!\!&0&1&0&\!\\
      \!\!&-x_2&x_1&\epsilon&
\end{pmatrix}.
    \end{equation}
Hence, explicitly, the Hamiltonian and the drift are respectively
\begin{eqnarray*}
H^{\varepsilon}(x,p)&=&\frac{1}{2}((p_1-x_2p_3)^2+(p_2+x_1p_3)^2+(\epsilon p_3)^2)\\
\partial_pH^{\varepsilon}(x,p)&=& p B^{\varepsilon}(x)(B^{\varepsilon}(x))^T=(p_1-x_2p_3, p_2+x_1p_3,x_1p_2-x_2p_1+(x_1^2+x_2^2+\varepsilon^2)p_3)
\end{eqnarray*}
while the dynamics of the generic player at point~$x$ at time~$t$ becomes
\begin{equation}\label{DYNe}
x_1'(s)=\alpha_1(s),\quad x_2'(s)=\alpha_2(s),\quad x_3'(s)=-x_2(s)\alpha_1(s)+x_1(s)\alpha_2(s)+\epsilon \alpha_3(s)
\end{equation}
where the control $\alpha=(\alpha_1, \alpha_2, \alpha_3)$ is chosen in $L^2([t,T];\R^3)$ for minimizing the cost~\eqref{Jgen}.

We now obtain existence, representation formula and suitable estimates of a solution to problem~\eqref{eq:MFGe} which will allow us to prove Theorem \ref{thm:main_illi}. These properties are stated in the following Proposition whose proof is postponed in the next section.
\begin{proposition}\label{mainproe}
For any fixed $\epsilon\in(0,1]$ there exists a solution $(u_{\varepsilon}, m_{\varepsilon})$ of the system 
\eqref{eq:MFGe} such that 
\begin{equation}
\label{ambrosioe}
\int_{\re^3}\phi(x) \, dm_\varepsilon(t)=\int_{\re^3}\phi(\gamma_{\varepsilon,x}(t))\,m_0(x)\, dx \qquad \forall \phi\in C^0_0(\R^3), \, \forall t\in[0,T]
\end{equation}
where, for a.e. $x\in\re^3$,  $\gamma_{\varepsilon,x}$ is the unique solution to
\begin{equation}\label{flowe} 
x'(s)= -Du_{\varepsilon}(x(s),s)B^{\epsilon}(x(s))(B^{\epsilon}(x(s)))^T,\quad x(0)=x.
\end{equation}
Moreover, there exists a positive constant~$C$ (independent of~$\epsilon$) such that
\begin{itemize}
\item[a)] $\|u_\varepsilon\|_\infty\leq C$, $\|D u_\varepsilon\|_\infty\leq C$, $|\partial_t u_\varepsilon(t,x)|\leq C (1+|x_1|^2+|x_2|^2)$, $D^2 u_\varepsilon\leq C$,
\item[b)] $\|m_\varepsilon\|_\infty\leq C$, ${\bf d}_1(m_\varepsilon(t_1),m_\varepsilon(t_2))\leq C|t_2-t_1|^{1/2}$, $\int_{\re^3}|x|^2m_\varepsilon(x,t)dx\leq C$.
\end{itemize}
\end{proposition}
Now, we can prove Theorem~\ref{thm:main_illi}.
\begin{proof}[Theorem \ref{thm:main_illi}]
1. The uniform estimates for~$u_{\varepsilon}$ and for~$m_{\varepsilon}$ in Proposition~\ref{mainproe} ensure that there exist two subsequences, which we will still denote $u_{\varepsilon}$ and respectively $m_{\varepsilon}$ such that, as $\epsilon\to 0^+$, $u_{\varepsilon}$ converge to some function~$u$ locally uniformly in $(x,t)$ and $m_\varepsilon$ converge to some $m\in C^0([0,T], \mathcal P_1(\re^3))$ in the $C^0([0,T],\mathcal P_1(\re^3))$-topology and in the weak-$*$-$L^\infty_{loc}(\re^3\times(0,T))$ topology. 
In particular, we get $m(0)=m_0$ and we deduce that $u$ is Lipschitz continuous in~$x$, locally Lipschitz continuous in~$t$, semiconcave in $x$ and 
$Du_{\varepsilon}\to Du$ a.e. (because of the semiconcavity estimate in Proposition~\ref{mainproe}-(a) and \cite[Theorem 3.3.3]{CS}).
Letting $\epsilon\to0^+$ in system~\eqref{eq:MFGe}, by the same arguments as those in the proof of Theorem~\ref{thm:main}, we infer that $(u,m)$ is a solution to~\eqref{eq:MFGintrin}.\\
2. The proof follows by easy adaptations of the proof of Theorem~\ref{thm:main}-(ii). Indeed it is enough to invoke~\cite[Theorem 8.2.1]{AGS} instead of Theorem~\ref{821} and to observe that the optimal synthesis in
Lemma~\ref{BB} 
still holds for the non periodic case (for a sketch of the proof see Lemma~\ref{lemma:OSeps}-(a) with~$\varepsilon=0$).
\end{proof}

\section{Tools for the non periodic case: optimal synthesis and proof of Proposition~\ref{mainproe}}\label{sect:dim7.1}

This section is devoted to the proof of Proposition~\ref{mainproe} which is given in section~\ref{subsect:pr_illi}.
To this end, we cope the study of equation~\eqref{eq:MFGe}-(i) and the corresponding optimal control problem (see section~\ref{subsect:OC_illi}) and separately of equation~\eqref{eq:MFGe}-(ii) (see section~\ref{subsect:c_illi}). It is worth to note that for the control problem underlying equation~\eqref{eq:MFGe}-(i) we establish an optimal synthesis result (see Lemma~\ref{lemma:OSeps}) exploiting the fact that the matrix $B^\varepsilon(B^\varepsilon)^T$ is invertible.
In all this section we take $\varepsilon\in [0,1]$ fixed and we omit it in most of the section.
\subsection{The associated optimal control problem}\label{subsect:OC_illi}
As in Section \ref{OC} for dynamics \eqref{DYNH}, here we state the optimal control problem associated to the Hamilton Jacobi equation in \eqref{eq:MFGe}. 
Throughout this section we shall require the following hypothesis
\begin{hypothesis} \label{BasicAss_illi}
$f\in C^{0}([0,T],C^2(\re^3))$, $g\in C^2(\re^3)$ and there exists a constant $C$ such that
$$\|f(\cdot,t)\|_{C^2(\re^3)} + \|g\|_{C^2(\re^3)} \leq C,\qquad\forall t\in[0,T].$$
\end{hypothesis}

\begin{definition}\label{def:OCDe} We consider the following optimal control problem
\begin{equation}\label{def:OC_illi}
\text{minimize } J_t(x(\cdot),\alpha(\cdot)):
=\displaystyle\int_t^T\dfrac12|\alpha(s)|^2+f( x(s),s)\,ds+g(x(T))
\end{equation}
subject to $(x(\cdot), \alpha(\cdot))\in \mathcal A_{\varepsilon}(x,t)$, where
\begin{equation} \mathcal A_{\varepsilon}(x,t):=\left\{(x(\cdot), \alpha(\cdot))\in AC([t,T]; \R^3)\times L^2([t,T];\re^3):\, \textrm{\eqref{DYNe} holds a.e. with } x(t)= x\right\}.
\end{equation}
Let $u_{\varepsilon}(x,t)$ be the value function of the optimal control problem~\eqref{def:OC_illi}, namely
\begin{equation}\label{repre}
u_\varepsilon(x,t):=\inf\left\{ J_t(x(\cdot), \alpha(\cdot)):\, (x(\cdot),\alpha(\cdot))\in \mathcal A_{\varepsilon}(x,t)\right\}.
\end{equation}
\end{definition}
From now we denote by $x^*$, $\alpha^*$ the optimal trajectory and the optimal control associated to $u_\varepsilon(x,t)$, i.e. we omit the dependence on $\varepsilon$.
\begin{lemma}\label{3.1}
\begin{enumerate}
\item For any $(x,t)$ there exists a solution $(x^*, \alpha^*)$ of the optimal control problem in Definition~\ref{def:OCDe} such that 
\begin{equation}\label{uniformalfae}
\|\alpha^*\|_2\leq C,\  \|\alpha^*\|_{\infty}\leq C(1+|x_1|+|x_2|),
\end{equation} 
with $C$ independent of $\varepsilon$.
\item Let $(x^*, \alpha^*)$ be optimal for the problem in \eqref{def:OC_illi}. 
Then, there exists an arc $p\in AC([t,T]; \re^3)$, called the costate, such that the pair~$(x^*, p)$ satisfies for a.e. $s\in [t,T]$ 
\begin{equation}\label{tage}
\left\{
\begin{array}{ll}
x_1'= p_1-x_2p_3,\\
x_2'= p_2+x_1p_3,\\
x_3'= (x_1^2+x_2^2+\varepsilon^2)p_3+x_1p_2-x_2p_1,\\
p_1'=-x_1p^2_3-p_2p_3+  f_{x_1}(x,s),\\
p_2'=-x_2p^2_3+p_1p_3+  f_{x_2}(x,s),\\
p_3'=f_{x_3}(x,s),
\end{array}\right.
\end{equation}
with the mixed boundary conditions 
\begin{equation}\label{tag:bcd} x^*(t)=x,\quad p(T)=-D g(x^*(T)).
\end{equation}
\item
The optimal control~$\alpha^*$ verifies
\begin{equation}\label{tagalphae}
\alpha_1^*(s)= p_1-x_2p_3,\qquad
\alpha_2^*(s)= p_2+x_1p_3,\qquad
\alpha_3^*(s)= \varepsilon p_3.
\end{equation}
\item The functions  $x^*$  and $\alpha^*$ are of class $C^1$. In particular equations \eqref{tage} and \eqref{tagalphae} hold for every $s\in [t,T]$.
\item Moreover, if $\varepsilon\neq 0$, the optimal trajectories are unique after the initial time: if $x^*(\cdot)$ is an optimal trajectory for $u_\varepsilon(x,t)$, then for every $t<\tau<T$ there are no other optimal trajectories for $u_\varepsilon(x^*(\tau),\tau)$ other than $x^*(\cdot)$ restricted to $[\tau,T]$. 
\end{enumerate}
\end{lemma}
\begin{proof}
1. The uniform $L^2$ estimate for $\alpha^*$ follows by the same argument of Remark~\ref{2.3}.
For the $L^{\infty}$ norm, we proceed as in the proof of Proposition~\ref{boundalfa}: for the first two components of $x^{\mu}(s)-x(s)$ the arguments are exactly the same and yield inequality~\eqref{claimmu}.
Here, the third component is
\begin{eqnarray*}
x_3^{\mu}(s)-x_3(s)&=&\int_t^s ((x_2^{\mu}(\tau)-x_2(\tau))\alpha_1^{\mu}(\tau)+ x_2(\tau)(\alpha_1^{\mu}(\tau)-\alpha_1(\tau)) +
(x_1^{\mu}(\tau)-x_1(\tau))\alpha_2^{\mu}(\tau)\\ 
&&+x_1(\tau)(\alpha_2^{\mu}(\tau)-\alpha_2(\tau))+\varepsilon (\alpha_3^{\mu}(\tau)-\alpha_3(\tau))
d\tau.
\end{eqnarray*}
Hence, using the same calculations of Proposition \ref{boundalfa}, we infer
\begin{equation*}
|x^{\mu}(s)-x(s)|\leq K(1+|x_1|+|x_2|+\varepsilon)\int_{I_{\mu}} |\alpha(\tau)|d\tau.
\end{equation*}
Still arguing as as in Proposition \ref{boundalfa}, we get
\begin{multline*}
J_t(x^{\mu}(s), \alpha^{\mu}(s))- J_t(x(s),\alpha(s))\leq\\
\leq \int_{I_{\mu}}\left(-\dfrac12|\alpha(s)|^2+ K(L_f (T-t)+L_g)(1+|x_1|+|x_2|+\varepsilon)|\alpha(s)|\right) ds.
\end{multline*}
For $\mu>2K(L_f (T-t)+L_g)(2+|x_1|+|x_2|)$, the last integrand is negative which contradicts the optimality of $\alpha$.
Since the choice of $\mu$ is independent on $\varepsilon$, we get the result.\\
Points 2., 3. and 4. can be obtained as in Propositions \ref{prop:pontriagind} and Corollary \ref{coro:regularityd}.\\
5. The statement can be established adapting the arguments in \cite[Theorem 5.2 (proof)]{MMMT}; hence, we just give the main steps and we refer to that paper for the details.
\\
Let~$y^*$ be an optimal trajectory for~$u(x^*(\tau),\tau)$; the concatenation~$z^*$ of~$x^*$ with~$y^*$ at~$\tau$ is still optimal for~$u(x,t)$. Let~$p$ and~$q$ be respectively the costate of~$x^*$ and of~$z^*$. By point~$(4)$ both~$x^*$ and~$z^*$ are $C^1$.
Since the matrix~$B^{\varepsilon}(x)(B^{\varepsilon}(x))^T$ is invertible, we denote by $\beta(x)$ its inverse and, from the first three lines in~\eqref{tage}, we get
\begin{equation*}
p(s)=\beta(x^*(s))\dot x^*(s)\qquad\textrm{and}\qquad q(s)=\beta(z^*(s))\dot z^*(s)\qquad\forall s\in(t,T).
\end{equation*}
Since $x^*(\cdot)=z^*(\cdot)$ in~$[t,\tau]$, we get: $p(\tau)=q(\tau)$. In conclusion, both~$(x^*,p)$ and~$(z^*,q)$ solve the same Cauchy problem \eqref{tage}
on~$(\tau,T]$ with the same data at time~$\tau$. The Cauchy-Lipschitz theorem ensures that they coincide.
\end{proof}

\begin{lemma}\label{LLS} The value function~$u_\varepsilon$ fulfills the following properties
\begin{enumerate}
\item $u_{\varepsilon}$ is Lipschitz continuous with respect to the spatial variable~$x$ uniformly on $\varepsilon$,
\item $u_{\varepsilon}$ is locally
Lipschitz continuous with respect to the time variable $t$ with a
Lipschitz constant $C(1+|x_1|^2+|x_2|^2)$ where $C$ is a constant independent of~$\varepsilon$.
\item $u_{\varepsilon}$ is semiconcave w.r.t. $x$ with a modulus of semiconcavity independent on $\varepsilon$.
\end{enumerate}
\end{lemma}
\begin{proof}
1. Let
$\alpha(\cdot)$ be the optimal control for $u(x,t)$ and $x(\cdot)$ the optimal trajectory. 
Let $x^*(s)$ be the path starting from $y=(y_1,y_2, y_3)$, with control $\alpha(\cdot)$.
To prove the Lipschitz continuity w.r.t. $x$ uniform on $\varepsilon$ we proceed as in Lemma \ref{L1}; the first two components of the trajectories  $x(s)$ and $x^*(s)$ are as in the previous section and the third components become
\begin{eqnarray*}
x_3(s)&=&x_3-\int_t^s\alpha_1(\tau)x_2(\tau)\,d\tau+
\int_t^s\alpha_2(\tau)x_1(\tau)\,d\tau+
\varepsilon\int_t^s\alpha_3(\tau)\,d\tau\\
x_3^*(s)&=&y_3-\int_t^s\alpha_1(\tau)x_2^*(\tau)\,d\tau+
\int_t^s\alpha_2(\tau)x_1^*(\tau)\,d\tau+
\varepsilon\int_t^s\alpha_3(\tau)\,d\tau;
\end{eqnarray*}
hence we deduce
$$x_3^*(s)=x_3(s)+(y_3-x_3)- (y_2-x_2)\int_t^s\alpha_1(\tau)\,d\tau+
(y_1-x_1)\int_t^s\alpha_2(\tau)\,d\tau,
$$
and we have the same equality as in Lemma \ref{L1}.\\
2. As far as the Lipschitz continuity w.r.t. $t$, still following Lemma \ref{L1}, we get
\begin{eqnarray*}
|x(s)-x|&\leq& C(s-t)(\|\alpha_1\|_{\infty}|x_2|+ \|\alpha_2\|_{\infty}|x_1|+\varepsilon\|\alpha_3\|_{\infty})\\
&\leq& K (1+|x_1|^2+|x_2|^2+\varepsilon)(s-t)\\&\leq& K (2+|x_1|^2+|x_2|^2)(s-t),
\end{eqnarray*}
where $K$ is independent on $\varepsilon$.\\
3. To prove the semiconcavity 
we follow Lemma \ref{semic}, noting that here the third components of $x(s)$ and 
$x_{\lambda}(s)$ become
\begin{eqnarray*}
x_{3}(s)&=&x_{3}-\int_t^s\alpha_1(\tau)x_{2}(\tau)\,d\tau+
\int_t^s\alpha_2(\tau)x_{1}(\tau)\,d\tau+\varepsilon\int_t^s\alpha_3(\tau)d\tau\\
x_{\lambda,3}(s)&=&x_{\lambda,3}-\int_t^s\alpha_1(\tau)x_{\lambda,2}(\tau)\,d\tau+
\int_t^s\alpha_2(\tau)x_{\lambda,1}(\tau)\,d\tau+\varepsilon\int_t^s\alpha_3(\tau)d\tau.
\end{eqnarray*}
Hence 
$x_{3}(s)-x_{\lambda,3}(s)$ is exactly as in \eqref{diff} and rest of the proof is the same as in Lemma \ref{semic}.
Note that, since from \eqref{uniformalfae} $\|\alpha\|_2\leq C$ where $C$ is independent on $\varepsilon$, the modulus of semicontinuity is independent of~$\varepsilon$.
\end{proof}

\begin{lemma}\label{3.2}
The value function~$u_{\varepsilon}$ is the unique bounded viscosity solution to problem
\begin{equation}\label{HJei}
-\partial_t u+H^{\varepsilon}(x, Du)=f(x)\quad \textrm{in }\re^3\times (0,T),\qquad u(x,T)=g(x)\quad \textrm{on }\re^3
\end{equation}
and there exists a constant~$C$ independent of $\varepsilon$ such that
\begin{equation}\label{boundue}
\|u_{\varepsilon}\|_{\infty}\leq C.
\end{equation}
\end{lemma}
\begin{proof}
The proof comes from classical results: see for instance \cite[Proposition III.3.5]{BCD} and \cite[Theorem 3.1]{BCP}).
The bound of $u_{\varepsilon}$ uniform on $\varepsilon$ is obtained taking as admissible control $\alpha=0$ in \eqref{repre}.
\end{proof}

\begin{lemma}[Optimal synthesis]\label{lemma:OSeps}
Let~$u_\varepsilon$ be the unique bounded viscosity solution to~\eqref{HJei} founded in Lemma~\ref{3.2}.
\begin{itemize}
\item[a)] Let $\gamma\in AC([t,T])$ be such that
\begin{equation}\label{hyp_OS}
\textrm{$u_\varepsilon(\cdot, s)$ is differentiable at~$\gamma(s)$ for almost every $s\in(t,T)$}
\end{equation}
and
\begin{equation}\label{1802:rmk2_i}
\dot\gamma(t)=-Du_\varepsilon(\gamma(t),t) B^{\varepsilon}(\gamma(t))(B^{\varepsilon}(\gamma(t)))^T,\qquad \gamma(0)=x. 
\end{equation}
Then, the control law $\alpha(s)=-Du_\varepsilon(\gamma(s),s) B^{\varepsilon}(\gamma(s))$ is optimal for $u_\varepsilon(x,t)$.
\item[b)] For~$\varepsilon\neq 0$, if $u_\varepsilon(\cdot,t)$ is differentiable at $x$, then problem~\eqref{hyp_OS}-\eqref{1802:rmk2_i} has a unique solution corresponding to the optimal trajectory. In particular, for a.e. $x$, there exists a unique optimal trajectory for~$u_\varepsilon(x,0)$.
\end{itemize}
\end{lemma}
\begin{proof}
We shall follow the same arguments as those used in~\cite[Lemma 4.11]{C} (see also \cite[Lemma 3.5]{AMMT} and \cite[Proposition 5.2]{MMMT} for similar arguments). So we only illustrate the main novelties in the proof and we refer the reader to those papers for the details.
Since $\varepsilon$ is fixed, for simplicity we write ``$u$'' instead of ``$u_\varepsilon$''.\\
$(a)$. Let $\gamma$ be a curve as in the statement; we claim that $\gamma$ is bounded and Lipschitz continuous. Indeed, the differential equation~\eqref{1802:rmk2_i} reads
\begin{equation*}
\gamma_1'=-u_{x_1}+\gamma_2 u_{x_3},\quad \gamma_2'=-u_{x_2}-\gamma_1 u_{x_3},\quad \gamma_3'=\gamma_2 u_{x_1}-\gamma_1 u_{x_2}-(\gamma_1^2+\gamma_2^2+\varepsilon^2) u_{x_3}.
\end{equation*}
Summing the first two equations multiplied respectively by~$\gamma_1$ and~$\gamma_2$, we obtain that the function $\xi:=\gamma_1^2+\gamma_2^2$ verifies
\[
\frac{1}{2}\xi'=-\gamma_1 u_{x_1}-\gamma_2 u_{x_2}.
\]
By the Cauchy-Schwarz inequality and the Lipschitz continuity of~$u$ (found in Lemma~\ref{LLS}), there exists a constant~$C$ such that $|\xi'|\leq C(\xi+1)$. We deduce that $\xi$ is bounded and Lipschitz continuous and, consequently, that $\gamma_1$ and $\gamma_2$ are bounded and Lipschitz continuous. Using these properties and again the Lipschitz continuity of~$u$, by the third component of~\eqref{1802:rmk2_i}, we obtain that also $\gamma_3$ is bounded and Lipschitz continuous. Our claim is proved.\\
The Lipschitz continuity of~$\gamma$ and of~$u$ entail that the function $s\mapsto u(\gamma(s),s)$ is Lipschitz continuous. The rest of the proof follows the same arguments of the aforementioned papers.\\
$(b)$. By the same arguments as those in \cite[Lemma 3.5 (proof)]{MMMT} and \cite[Lemma 3.4 (proof)]{AMMT}, one can prove that $Du(x,t)$ exists if, and only if the set $\{\alpha(t),\ \alpha {\text { optimal for } }u(x,t)\}$ is a singleton
and that there holds $\alpha(t)=-Du(x,t) B^{\varepsilon}(x)$. In particular, if $Du(x,t)$ exists, the value of~$p(t)$ is uniquely determined by relations~\eqref{tagalphae} (here, $\varepsilon\ne0$ is needed); hence, system~\eqref{tage} becomes a system of differential equations with condition at time~$t$ and admits a unique solution by the Cauchy-Lipschitz theorem. The rest of the proof follows the same arguments of the aforementioned papers.
\end{proof}

\subsection{The continuity equation}\label{subsect:c_illi}

This section is devoted to study equation \eqref{eq:MFGe}-(ii), namely, to study
\begin{equation}\label{continuitye_illi}
\left\{
\begin{array}{ll} \partial_t m-
\diver(m\, Du B^{\varepsilon}(x)(B^{\varepsilon}(x))^T)=0
&\qquad \textrm{in }\re^3\times (0,T)\\
m(x,0)=m_0(x) &\qquad \textrm{on }\re^3
\end{array}\right.
\end{equation}
where $u$ is the unique bounded (viscosity) solution to problem
\begin{equation}\label{HJe_illi}
\left\{\begin{array}{ll}
-\partial_t u+\frac12 |DuB^{\varepsilon}(x)|^2=F[\overline{m}(t)](x)&\qquad \textrm{in }\re^3\times (0,T)\\
u(x,T)=G[\overline {m}(T)](x)&\qquad \textrm{on }\re^3
\end{array}\right.
\end{equation}
with~$\overline m\in C^{0}([0,T],\mathcal P_1(\re^3))$ (see Lemma~\ref{3.2} for the existence and uniqueness of~$u$) .
Since now on, throughout this section, $\varepsilon\in(0,1]$ and $\overline m$ are fixed.
We perform a vanishing viscosity approach with the Euclidean Laplacian in the whole system and a truncation argument only in the continuity equation. The vanishing viscosity approach permits to exploit the well posedness of uniformly parabolic equations while the truncation argument permits to overcome the issue of coefficients which grow ``too much'' as $x\to\infty$.\\
For $\sigma \in(0,1]$, we consider the problem
\begin{equation}\label{eq:MFGeN}
\left\{\begin{array}{lll}
(i)&\quad-\partial_t u-\sigma \Delta u+H^{\varepsilon}(x, Du)=F[\overline m(t)](x)&\quad \textrm{in }\re^3\times (0,T)\\
(ii)&\quad\partial_t m-\sigma\Delta m-\diver  (mDu B^{\varepsilon,N}(x)(B^{\varepsilon,N}(x))^T)=0&\quad \textrm{in }\re^3\times (0,T)\\
(iii)&\quad m(x,0)=m_0(x), u(x,T)=G[\overline m(T)](x)&\quad \textrm{on }\re^3,
\end{array}\right.
\end{equation}
where
\[
B^{\varepsilon,N}(x):=  \begin{pmatrix}
\!\!1& 0& 0\!\\
\!\!0&1& 0\!\\
\!\!-\psi_N(x_2)&\psi_N(x_1)&\varepsilon \!
\end{pmatrix}, \qquad \psi_N(\xi):=\left\{\begin{array}{ll}
\xi &\quad \textrm{if }|\xi|\leq N\\0 &\quad \textrm{if }|\xi|\geq 2 N
\end{array}\right.
\]
with $\psi_N\in C^2(\re)$, $\|\psi_N\|_{L^\infty}\leq 2N$,
$\|\psi_N'\|_{L^\infty}+\|\psi_N''\|_{L^\infty}\leq K$ ($K$ independent of $N$).

\begin{lemma}\label{1802:lemma1}
There exists a unique bounded classical solution $u_\sigma$ to equation~\eqref{eq:MFGeN}-(i) with terminal condition \eqref{eq:MFGeN}-(iii). Moreover there exists a positive constant~$C$ (independent of $\varepsilon$, $\sigma$, $N$ and $\overline m$) such that
\begin{itemize}
\item[a)] $\|u_\sigma\|_\infty\leq C$
\item[b)] $\|D u_\sigma\|_\infty\leq C$, $|\partial_t u_\sigma(t,x)|\leq C (1+|x_1|^2+|x_2|^2)$
\item[c)] $D^2 u_\sigma\leq C$.
\end{itemize}
\end{lemma}
\begin{proof}
It is enough to invoke classical results on parabolic equations for the existence of a solution and the comparison principle (see for instance, \cite{BU,Lie} and also \cite{DL}) using super- and subsolutions of the form $w^\pm=\pm C(T+1-t)$. Moreover, since the coefficients of 
$B^{\varepsilon}$ have linear growth at infinity, the solution $u_\sigma$ of \eqref{eq:MFGeN}-(i) is the value function of an optimal control problem whose dynamics are given by the stochastic differential equation
\begin{equation}
\label{stocNP}
dY=\alpha(t) B^{\varepsilon}(Y_t)^Tdt +\sqrt{2\sigma} dW_{t}.
\end{equation}
Then estimates in $(b)$ and $(c)$ follow by the same arguments as those in \cite[Lemma 4.1]{AMMT} or \cite[Lemma 3.2]{MMMT}.  
\end{proof}
\begin{lemma}\label{1802:lemma3}
There exists a unique bounded classical solution $m_{\sigma,N}$ to problem \eqref{eq:MFGeN}-(ii) with initial condition \eqref{eq:MFGeN}-(iii). Moreover, there exists a constant~$C_N$ (independent of~$\varepsilon$, $\sigma$ and~$\overline m$) such that $0<m_{\sigma,N}\leq C_N$ in $\re^3\times (0,T)$.
\end{lemma}
\begin{proof}
Since~$\sigma$ is fixed, for simplicity we write ``$u$'' instead of ``$u_\sigma$''. By the regularity of $u$ (see Lemma~\ref{1802:lemma1}), the equation~\eqref{eq:MFGeN}-(ii) can be written as
\[
\partial_t m+\sigma \Delta m-Dm\cdot(Du B^{\varepsilon,N}(B^{\varepsilon,N})^T)-m \diver(Du B^{\varepsilon,N}(B^{\varepsilon,N})^T)=0.
\]
Let us assume for the moment that there exists a positive constant~$k_N$ (independent of $\varepsilon$, $\sigma$ and $\overline m$) such that
\begin{equation}\label{1802:lemma2}
\diver(Du B^{\varepsilon,N}(B^{\varepsilon,N})^T)\leq k_N.
\end{equation}
On the other hand, by Lemma~\ref{1802:lemma1}-(b), we have $\|Du B^{\varepsilon,N}(B^{\varepsilon,N})^T\|_{\infty}\leq CN^2$. 
Hence we can invoke the results by Ikeda~\cite{Ik} for the existence of a solution. Moreover, arguing as in~\cite[Lemma 4.2 and Lemma 4.3]{AMMT}, one can easily obtain the bounds on $m_{\sigma,N}$.

It remains to prove~\eqref{1802:lemma2}. We denote by $I$ the left hand side. There holds
\begin{eqnarray*}
I&=&\partial_{11} u-2\psi_N(x_2)\partial_{13}u+\partial_{22}u+2\psi_N(x_1)\partial_{23}u+(\psi_N(x_2)^2+\psi_N(x_1)^2+\varepsilon^2)\partial_{33}u\\
&=& \xi_1 D^2u\xi_1^T+ \xi_2 D^2u\xi_2^T+ \xi_3 D^2u\xi_3^T
\end{eqnarray*}
where $\xi_1=(1,0,-\psi_N(x_2))$, $\xi_2=(0,1,\psi_N(x_1))$ and $\xi_3=(0,0,\varepsilon)$. By Lemma~\ref{1802:lemma1}-(c), we obtain
\[
I\leq C(2+2\|\psi_N\|_\infty^2 +\varepsilon^2)\leq C(3+8N^2)
\]
where the last inequality is due to our assumption on $\psi_N$. Choosing $k_N=C(3+8N^2)$ we accomplish the proof of~\eqref{1802:lemma2}.
\end{proof}

\begin{lemma}\label{1802:lemma4}
There exists a constant $K_N$, depending on~$N$ (but independent of~$\varepsilon$, $\sigma$ and $\overline m$) such that the function~$m_{\sigma,N}$ found in Lemma~\ref{1802:lemma3} verifies
\begin{itemize}
\item[a)] ${\bf d}_1(m_{\sigma,N}(t_1),m_{\sigma,N}(t_2))\leq K_N(t_2-t_1)^{1/2}$ for any $0\leq t_1<t_2\leq T$
\item[b)] $\int_{\re^3}|x|^2m_{\sigma,N}(x,t)dx\leq K_N(\int_{\re^3}|x|^2dm_0(x)+1)$.
\end{itemize}
\end{lemma}
\begin{proof}
Taking account that $m_{\sigma,N}$ can be interpreted as the law of the following stochastic process
\begin{equation*}
d Y^x_t= -Du (Y^x_t,t) B^{\varepsilon,N}(Y^x_t)(B^{\varepsilon,N}(Y^x_t))^T\, dt +\sqrt{2\sigma} d W_t,\qquad Y^x_0=x,
\end{equation*}
by the same arguments as those in \cite[Lemma 4.3 (proof)]{AMMT}, we obtain the result.
\end{proof}
\vskip .2cm
Now we let $\sigma\rightarrow 0$ and we consider the problem
\begin{equation}\label{continuityeN}
\left\{
\begin{array}{ll} \partial_t m-
\diver(m\, Du B^{\varepsilon,N}(x)(B^{\varepsilon,N}(x))^T)=0
&\qquad \textrm{in }\re^3\times (0,T)\\
m(x,0)=m_0(x) &\qquad \textrm{on }\re^3
\end{array}\right.
\end{equation}
where $u$ is the unique bounded solution to problem~\eqref{HJe_illi}.
\begin{lemma}\label{1802:thm1}
For $N$ sufficiently large, problem~\eqref{continuityeN} admits exactly one solution~$m_{N}$ in the space $C^{1/2}([0,T],\mathcal P_1(\re^3))\cap L^\infty(0,T;\mathcal P_2(\re^3))$.\\
Moreover, the solution~$m_{N}$ has a density in $L^\infty(\re^3\times(0,T))$ and it is the image of the initial distribution through the flow
\begin{equation}\label{1802:rmk2}
\dot\gamma(t)=-Du(\gamma(t),t) B^{\varepsilon}(\gamma(t))(B^{\varepsilon}(\gamma(t)))^T,\qquad \gamma(0)=x
\end{equation}
which is uniquely determined for $m_0$-a.e. $x\in\re^3$.
\end{lemma}
\begin{proof}
Fix $N$ sufficiently large (it will be suitably chosen later on). For any $\sigma\in(0,1]$, let $(u_\sigma,m_{\sigma,N})$ be the unique classical bounded solution to system~\eqref{eq:MFGeN} (see Lemma~\ref{1802:lemma1} and Lemma~\ref{1802:lemma3}).
Letting $\sigma\to0^+$, by the estimates in Lemma~\ref{1802:lemma1}, in Lemma~\ref{1802:lemma3} and in Lemma~\ref{1802:lemma4}, possibly passing to a subsequence that we still denote $(u_\sigma,m_{\sigma,N})$, we get that the functions~$u_\sigma$ converge locally uniformly to the unique solution~$u$ to~\eqref{HJe_illi} while~$m_{\sigma,N}$ converge to a function $m_{N} \in C^{1/2}([0,T],\mathcal P_1(\re^3))\cap L^\infty(0,T;\mathcal P_2(\re^3))$ in the $C^{0}([0,T],\mathcal P_1(\re^3))$-topology and in the weak-$*$ topology of $L^\infty(\re^3\times(0,T))$. By the same arguments as those in~\cite[Proposition 4.1 (proof)]{AMMT}, in particular the uniform semiconcavity, we obtain that $m_{N}$ is a solution to~\eqref{continuityeN}.\\
Let us now establish uniqueness and representation formula for the solution~$m_{N}$. We observe that, by Lemma~\ref{1802:lemma1}-(b), the drift verifies $\|Du B^{\varepsilon,N}(B^{\varepsilon,N})^T\|_{\infty}\leq CN^2$ and in particular condition~\cite[eq. (8.1.20)]{AGS} is fulfilled. Hence, we can apply the superposition principle in~\cite[Theorem 8.2.1]{AGS}: there exists a measure $\eta_N$ on $\re^3\times \Gamma$ such that
\begin{itemize}
\item[1)] $m_N(t)=e_t\#\eta_N$ for all $t\in(0,T)$ (recall: $e_t(x,\gamma)=\gamma(t)$)
\item[2)] $\eta_N =\int_{\re^3}(\eta_N)_x dm_0(x)$ where, for $m_0$-a.e. $x\in\re^3$, the measure $(\eta_N)_x$ is concentrated on the set of pairs $(x,\gamma)\in\re^3\times \Gamma$ where $\gamma$ solves
\begin{equation}\label{1802:rmk2N}
\dot\gamma(t)=-Du(\gamma(t),t) B^{\varepsilon,N}(\gamma(t))(B^{\varepsilon,N}(\gamma(t)))^T\quad\textrm{a.e. }t\in(0,T),\qquad \gamma(0)=x. 
\end{equation}
\end{itemize}
We now claim that, for $m_0$-a.e. $x\in\re^3$, the solutions to~\eqref{1802:rmk2N} coincide with those of~\eqref{1802:rmk2}. Indeed, since $m_0$ has compact support, by arguments similar to those in the proof of Lemma~\ref{lemma:OSeps}-(a), we obtain that any solution~$\gamma$ to~\eqref{1802:rmk2N} is bounded uniformly in~$N$; namely, there exists a positive constant~$k$ (independent of~$N$) such that, for $m_0$-a.e. $x\in\re^3$, any solution~$\gamma$ to~\eqref{1802:rmk2N} verifies $\gamma(t)\in [-k,k]^3$ for any $t\in(0,T)$. Hence, choosing $N\geq k$, we get that problem~\eqref{1802:rmk2N} coincides with~\eqref{1802:rmk2} if $x\in\textrm{supp}(m_0)$.\\
It remains only to prove that, for $m_0$-a.e. $x\in\re^3$, problem~\eqref{1802:rmk2} admits exactly one solution. To this end, it is enough to invoke~Lemma~\ref{lemma:OSeps} and taking into account that~$u(\cdot,0)$ is Lipschitz continuous.
\end{proof}

Now we let $N\rightarrow +\infty$ and we establish the existence of a solution to problem~\eqref{continuitye_illi} exploiting that $m_N$ have compact support independent on $N$ which in turn is due to compactness of~$\textrm{supp} (m_0)$.
\begin{lemma}\label{prp:cont_illi}
The problem~\eqref{continuitye_illi} has a solution~$m$ in the space $C^{1/2}([0,T],\mathcal P_1(\re^3))\cap L^\infty(0,T;\mathcal P_2(\re^3))$ with a density in~$L^\infty(\re^3\times(0,T))$. Moreover, $m$ fulfills~\eqref{ambrosioe}-\eqref{flowe} and there exists a constant~$K$ independent of~$\overline m$ and of~$\varepsilon$ such that
\[\|m\|_\infty\leq K, \qquad {\bf d}_1(m(t_1),m(t_2))\leq K|t_2-t_1|^{1/2}, \qquad \int_{\re^3}|x|^2m(x,t)dx\leq K.
\]
\end{lemma}
\begin{proof}
By Lemma~\ref{1802:thm1} all the solutions~$m_{N}$ to~\eqref{continuityeN} coincide if $N$ is sufficiently large. Hence, passing to the limit as $N\to\infty$, we obtain that problem~\eqref{continuitye_illi} admits a solution $m$ in the space $C^{1/2}([0,T],\mathcal P_1(\re^3))\cap L^\infty(0,T;\mathcal P_2(\re^3))$ with a density in~$L^\infty(\re^3\times(0,T))$. Finally, the estimates follow from the corresponding estimates in Lemma~\ref{1802:lemma4}.
\end{proof}

\subsection{Proof of Proposition \ref{mainproe}}\label{subsect:pr_illi}
\begin{proof}[Proposition \ref{mainproe}]
We achieve the proof through a fixed point argument as in~\cite[Theorem 1.1 (proof)]{MMMT} taking advantage of the results of Lemma~\ref{prp:cont_illi}. We refer the reader to the aforementioned paper for the details.
\end{proof}


\appendix

\section{$\cH$-differentials}\label{app:Hdiff}
In this appendix we introduce the notions of horizontal generalized differentials extending the Euclidean ones~\cite[section 3.1]{CS} (see also \cite[section 6.2]{MMMT} for the Grushin case). We need these notions to study the horizontal regularity of a function $u$. Still following the same arguments as those in 
 \cite{CS, MMMT} we get the proofs of the properties contained in this appendix. 

\begin{definition}\label{def:Hdiff}
A function $u:\re^3\to \re$ is $\cH$-differentiable at $x=(x_1,x_2,x_3)\in\re^3$ if there exists $p\in\re^2$ such that
\begin{equation*}
\lim_{\re^2\ni h=(h_1, h_2)\to 0}\frac{u(x_1+h_1,x_2+h_2, x_3-x_2h_1+x_1h_2)-u(x_1,x_2,x_3) -p\cdot h}{|h|}=0, 
\end{equation*}
and in this case we denote $p=D_\cH u(x)$.
We define the $\cH$-subdifferential and the lower $\cH$-Dini derivative in the direction~$\theta\in\re^2$ respectively as
\begin{eqnarray*}
D^-_\cH u(x)&=&\left\{p\in \re^2 \mid
\liminf_{\re^2\ni h\to 0}\frac{u(x_1+h_1,x_2+h_2, x_3-x_2h_1+x_1h_2)-u(x) -p\cdot h}{|h|}\geq 0
\right\}\\
\partial^-_\cH u(x,\theta)&=&\liminf_{h\to 0,\theta'\to \theta}\frac{u(x_1+h_1\theta'_1,x_2+h_2\theta'_2, x_3-x_2h_1\theta'_1+x_1h_2\theta'_2)-u(x)}{h}.
\end{eqnarray*}
We define the $\cH$-superdifferential  $D^+_\cH u$ and the upper $\cH$-Dini derivative 
$\partial^+_\cH u$ in a similar way.
\end{definition}
\begin{remark}
$D_\cH u(x)$ coincides with the horizontal gradient $(X_1u, X_2u)$ when $u$ is sufficiently regular.
\end{remark}
\begin{lemma}\label{lm:Hdiff}
\begin{itemize}
\item [i)] If $u$ is $\cH$-differentiable at~$x$ then $D_\cH u(x)$ is a singleton and $D_\cH^+ u(x)$ and $D_\cH^- u(x)$ are both nonempty.
\item [ii)] When $u$ is Lipschitz continuous in a neighbourhood of~$x$ the $\cH$-Dini lower derivative reduces to
\begin{equation*}  
\partial^-_\cH u(x,\theta)=\liminf_{h\to 0}\frac{u(x_1+h_1\theta_1,x_2+h_2\theta_2, x_3-x_2h_1\theta_1+x_1h_2\theta_2)-u(x)}{h}.
\end{equation*}
\item[iii)] For any $p= (p_1, p_2, p_3)$ in the Euclidean superdifferential $D^+u(x)$, the vector $pB(x)$ belongs to $D^+_\cH u(x)$.
\end{itemize}
\end{lemma}
\begin{proposition}\label{app:prp3.1.5}
We have
\begin{eqnarray*}
D^+_\cH u(x)&=&\{p\in\re^2:\, \partial^+_\cH u(x,\theta)\leq p\cdot \theta,\quad \forall \theta\in\re^2\}\\
D^-_\cH u(x)&=&\{p\in\re^2:\, \partial^-_\cH u(x,\theta)\geq p\cdot \theta,\quad \forall \theta\in\re^2\}.
\end{eqnarray*}
Moreover, $D^+_\cH u(x)$ and $D^-_\cH u(x)$ are both nonempty if and only if $u$ is $\cH$-differentiable at $x$ and in this case they reduce to the singleton $D_\cH u(x)=D^-_\cH u(x)=D^+_\cH u(x)$.
\end{proposition}
%
%
%
\section{On the uniqueness for second-order Fokker-Planck equation}\label{app:uniqueness}
In this appendix, for the sake of completeness, we collect some results on the uniqueness of the solution to the Cauchy problem for the second-order Fokker-Planck equation~\eqref{eq:FPv} with fixed $\sigma>0$ and dropping the periodicity assumption of the coefficients: for $\sigma>0$, we consider the Cauchy problem
\begin{equation} \label{besala}
\left\{
\begin{array}{ll}
\partial_t m-\sigma\Delta_\cH m+ b\cdot D_\cH m+ cm=0&\qquad \textrm{in }\He^1\times (0,T) \\
m(x,0)=m_0(x)&\qquad \textrm{on }\He^1.
\end{array}\right.
\end{equation}
Let us just underline that in Euclidean setting the above differential equation  becomes
\begin{equation*}
\partial_t m-\sigma\tr\left(D^2 m BB^T\right)+b\cdot (D mB) + c\,m=0
\end{equation*}
which is a degenerate second-order linear equation with unbounded coefficient: the one of the principal part has a quadratic growth while the one of the first-order part has a linear growth and fails to be globally Lipschitz continuous. 
Let us recall that in the paper \cite{CC} a similar result is obtained with a different approach under stronger assumptions on the coefficients.

We shall tackle two different settings: in the former the coefficients~$b$ and~$c$ are bounded and the solution is classical while in the latter the coefficients are possibly unbounded (but more regular) and the solution is weak.

For any domain $U\subset \He^1\times[0,T]$, any $k\in\N$ and any $\delta\in(0,1]$, we denote $C^{k+\delta}_\cH(U)$ (resp. $C^{k+\delta}_{\cH, loc}(U)$) the (resp. local) parabolic H\"older space adapted to the vector fields~$X_1$ and~$X_2$ (for instance, see~\cite[Section 4]{BB07} or~\cite[Definition 10.4]{BBLU}). For $\delta=0$ and $k=0$, we simply denote $C^{k}_\cH(U)$ and respectively $C^{\delta}_\cH(U)$.
%
%
\begin{proposition}\label{prp:!fpv_reg}
Assume that, in~\eqref{besala}, $b$ and $c$ are bounded continuous functions defined in $\He^1\times[0,T]$ and $b$ has a continuous and bounded horizontal gradient. For $i=1,2$, let $m_i\in C^{2}_\cH(\He^1\times(0,T])\cap C^0(\He^1\times[0,T])$ be two classical solution to~\eqref{besala} such that, for some positive constant~$\alpha$,
\begin{equation}\label{exponential}
\iint_{\He^1\times[0,T]} |m_i(x,t)|\exp\{-\alpha(\|x\|_\cH^2+1)\}dxdt<\infty.
\end{equation}
Then, $m_1=m_2$.
\end{proposition}
\begin{remark}
Estimate~\eqref{exponential} is verified by any function which is $Q_\cH$-periodic and belongs to~$\mathbb L^1(Q_\cH)$.
\end{remark}
\begin{proof}
Without any loss of generality, we assume $c\geq 0$; 
We shall adapt the techniques of \cite[Theorem 1]{BU}. To this end, we proceed by contradiction assuming that $m_1\ne m_2$. Let $\tau_0$ be the first time such that $m(\cdot,t)\ne m_2(\cdot,t)$, namely
\[
\tau_0:=\inf\{t\in[0,T]\mid m_1(\cdot,t)\ne m_2(\cdot,t)\}.
\]
By our assumption on the continuity of~$m_i$, $\tau_0$ belongs to~$[0,T)$. The initial condition of~\eqref{besala} (if $\tau_0=0$) and the continuity of $m_1$ and $m_2$ (if $\tau_0>0$) ensure that the function $m:=m_1-m_2$ solves
\[
\partial_t m-\sigma\Delta_\cH m+ b\cdot D_\cH m+ cm=0\quad \textrm{in }\He^1\times(\tau_0,T) ,\qquad m(x,\tau_0)=0\quad \textrm{on }\He^1.
\]
For any $\epsilon>0$, the function~$w:=\sqrt{m^2+\epsilon}$ verifies
\begin{equation*}
\begin{array}{l}
\partial_t w=\frac{m\partial_t m}{w},\, X_i w=\frac{m X_i m}{w},\, 
X_i^2 w=\epsilon\frac{(X_i m)^2}{w^3}+\frac{m}{w}X_i^2m,\, \Delta_\cH w=\epsilon\frac{|D_\cH m|^2}{w^3}+\frac{m}{w}\Delta_\cH m.
\end{array}
\end{equation*}
We multiply the differential equation by $m/w$ and, by these equalities and the sign of~$c$, we obtain
\begin{equation*}
\partial_t w=\sigma \Delta_\cH w -\sigma\epsilon \frac{|D_\cH m|^2}{w^3} -b\cdot D_\cH w - c \frac{m^2}{w}\leq \sigma \Delta_\cH w -b\cdot D_\cH w.
\end{equation*}
We deduce that, for any nonnegative test function $v\in C^{\infty}(\He^1\times[\tau_0,T])$ with bounded support in space and for every $t\in[\tau_0,T]$, there holds
\begin{equation*}
\int_{\He^1}w(x,t)v(x,t)dx-\int_{\He^1}w(x,\tau_0)v(x,\tau_0)dx\leq \iint_{\He^1\times[\tau_0,t]} w[\partial_tv+\sigma \Delta_\cH v+\diver_\cH(v b)]dx ds.
\end{equation*}
Since $w(\cdot,\tau_0)=\epsilon$, letting $\epsilon\to 0^+$, we deduce
\begin{equation}\label{besala1}
\int_{\He^1}|m(x,t)|v(x,t)dx\leq \iint_{\He^1\times[\tau_0,t]} |m|[\partial_tv+\sigma \Delta_\cH v+\diver_\cH(v b)]dx ds.
\end{equation}
Let us state the following technical Lemma whose proof is postponed after this proof. We recall that $\alpha$ is the constant of Proposition \ref{prp:!fpv_reg}.
%
%
\begin{lemma}\label{lm:besala}
For $\alpha_1>\alpha$, the function $\Phi(t,x):=\exp\{-[\alpha_1+\beta (t-\tau_0)](\|x\|_\cH^2+1)\}$ satisfies
\begin{equation*}
\begin{array}{rl}
i) & \quad  \partial_t \Phi+\sigma \Delta_\cH\Phi+ b\cdot D_\cH \Phi+ (\diver_\cH b) \Phi\leq 0\qquad \textrm{in }(\tau_0,\tau)\times\He^1\\
ii) &\quad \iint_{\He^1\times[\tau_0,\tau]} |m_i(x,t)|\Phi(x,t)dxdt<\infty,\qquad \iint_{\He^1\times[\tau_0,\tau]} |m_i(x,t) D_\cH\Phi(x,t)|dxdt<\infty
\end{array}
\end{equation*}
for suitable constants $\beta>0$ and $\tau\in (\tau_0,T]$.
\end{lemma}
We choose $t\in[\tau_0, \tau]$ and $v=\gamma_R \Phi$ where $\tau$ and~$\Phi$ are respectively the constant and the function introduced in Lemma~\ref{lm:besala} while $\gamma_R\in C^\infty(\He^1)$ is a cut-off function such that:
\begin{equation*}
\gamma_R(x)=1\quad \textrm{if }|x|\leq R,\qquad
\gamma_R(x)=0\quad \textrm{if }|x|\geq R+1,\qquad \|D\gamma_R\|_\infty+\|D^2\gamma_R\|_\infty\leq 2.
\end{equation*}
Hence, inequality~\eqref{besala1} becomes
\begin{multline*}
\int_{\He^1}|m(x,t)|\gamma_R(x)\Phi(x,t)dx\leq\\ \iint_{[B(0,R+1)\setminus B(0,R)]\times[\tau_0,t]} |m|\left[\Phi(\sigma \Delta_\cH \gamma_R+b\cdot D_\cH \gamma_R)+2\sigma D_\cH\gamma_R\cdot D_\cH\Phi\right] dx dt.
\end{multline*}
Letting $R\to+\infty$, since the dominated convergence theorem and Lemma~\ref{lm:besala}-(ii) ensure that the right-hand side tends to zero, last inequality yields
\begin{equation*}
\int_{\He^1}|m(x,t)|\Phi(x,t)dx\leq 0\qquad \forall t\in[\tau_0,\tau]
\end{equation*}
which entails $m=0$ in $\He^1\times(\tau_0,\tau)$ contradicting the definition of $\tau_0$.
\end{proof}

\begin{Proofc}{Proof of Lemma~\ref{lm:besala}}
The equalities in Lemma~\ref{Xcalc}-(i), (ii) and (iv) entail respectively that there hold
\begin{equation*}
|X_i(\|x\|_\cH^2)|^2\leq C_1\|x\|_\cH^2,\qquad
|D_\cH(\|x\|_\cH^2)|^2\leq C_1 \|x\|_\cH^2,\qquad
|\Delta_\cH(\|x\|_\cH^2)|^2\leq C_1
\end{equation*}
for a suitable positive constant $C_1$.
Taking into account these estimates, denoting by $\alpha_2:=\alpha_1+\beta (\tau-\tau_0)$, we have
\begin{equation*}
\begin{array}{l}
\partial_t \Phi+\sigma \Delta_\cH\Phi+b\cdot D_\cH\Phi+ (\diver_\cH b) \Phi\\
\quad= \Phi\left[-\beta(\|x\|_\cH^2+1) +\sigma \alpha_2^2|D_\cH(\|x\|_\cH^2)|^2 -\sigma \alpha_2\Delta_\cH(\|x\|_\cH^2)
-\alpha_2 b\cdot D_\cH(\|x\|_\cH^2)+\diver_\cH b\right]\\
\quad \leq \Phi \left[-\beta(\|x\|_\cH^2+1)+\sigma\alpha_2^2C_1\|x\|_\cH^2 +\sigma\alpha_2C_1+\|b\|_\infty\alpha_2C_1\|x\|_\cH+\|\diver_\cH b\|_\infty\right].
\end{array}
\end{equation*}
Choosing $\tau-\tau_0$ sufficiently small and $\beta$ sufficiently large, we accomplish the proof of point~$(i)$.\\ 
Point~$(ii)$ is an easy consequence of our choice of $\alpha_1$ and our assumption~\eqref{exponential}.
\end{Proofc}
%
%
Let us now establish a uniqueness result for weak solution to problem~\eqref{besala}. To this end, it is expedient to introduce the following family of test functions
\begin{equation}\label{def:K_t_beta}
{\mathcal K}_{t,\beta}:=\left\{\phi\in C^{2}(\He^1\times[0,t])\mid \exists C>0:\ 
\begin{array}{rl}
i)& |\phi|\leq C \exp\{\beta\|x\|_\cH^2\}\\
ii)& |{\mathcal A}^*\phi|\leq C\exp\{\beta\|x\|_\cH^2\}\\
\end{array}
\right\}
\end{equation}
where ${\mathcal A}^*\phi:=\partial_t \phi +\sigma \Delta_\cH\phi+\diver_\cH(b\phi)-c\phi$. 
\begin{example}
It is clear that $C^\infty_0(\He^1)\subset {\mathcal K}_{t,\beta}$ for any~$\beta\in\R$. For~$\beta$ non positive, the property~$(i)$ in~\eqref{def:K_t_beta} is satisfied by any $\phi\in\mathbb L^1(\He^1)$. For~$\beta$ negative, ${\mathcal K}_{t,\beta}$ contains all the bounded functions $\phi\in C^{2,1}_\cH$ with ${\mathcal A}^*\phi$ bounded.
\end{example}

%
%
\begin{proposition}\label{prp:!FPv_distr}
Assume that, for some $\delta\in(0,1]$ and some $\beta_0\in \R$, there hold
\begin{itemize}
\item[I)]  $b$, $c$ and their horizontal derivatives up to second order and respectively first order belong to the space $C^{\delta}_\cH(\He^1\times[0,T])$.
\item[II)]  
$\int_{\He^1}|m_0(x)|\exp\{\beta_0\|x\|_\cH^2\}dx<\infty$.
\end{itemize}
Furthermore assume also that, for some fixed constant $\beta\leq \beta_0$, for $i=1,2$ the functions $m_i:[0,T]\to {\mathcal M} (\He^1)$ verify
\begin{equation}\label{eq:FPv_distr}
\int_{\He^1}\phi(x,t)m_i(t)(dx)=\int_{\He^1}\phi(x,0)m_0(x)dx+
\iint_{\He^1\times[0,t]} ({\mathcal A}^*\phi) m_i(s)(dx)ds
\end{equation}
for every $t\in(0,T)$ and every $\phi\in {\mathcal K}_{t,\beta}$.
Then, $m_1=m_2$.
\end{proposition}
\begin{proof}
  We shall argue following a classical method going back to Holmgren (see \cite[pag.340]{Kr} 
and references therein). It suffices to show that, for every $\psi\in C^\infty_0(\He^1)$ with $\|\psi\|_\infty\leq 1$ and $t\in(0,T]$, there holds
\begin{equation*}
\int_{\He^1}\psi(x)m_1(t)(dx)=\int_{\He^1}\psi(x)m_2(t)(dx).
\end{equation*}
To this end, we fix such $\psi$ and $t\in(0,T_0]$, where $T_0$ will be suitably chosen later on and it will only depend on the coefficients $b$ and $c$, and consider the (backward) Cauchy problem
\begin{equation}\label{pbC_duale}
\left\{\begin{array}{ll}
{\mathcal A}^*\phi=\partial_t \phi +\sigma \Delta_\cH\phi+\diver_\cH(b\phi)-c\phi=0&\qquad\textrm{in }\He^1\times(0,T_0) \\
\phi(T_0,x)=\psi(x)&\qquad\textrm{on }\He^1.
\end{array}\right.
\end{equation}
Invoking~\cite[Theorem 10.7-(v)]{BBLU} and~\cite[Theorem 1.1]{BB07}, we obtain that there exists a function $\phi\in C^{2,\delta}_{\cH}$ which is a classical solution to problem~\eqref{pbC_duale}.

Assume for the moment that the function~$\phi$ belongs to ${\mathcal K}_{T_0,\beta}$; then $\int_{\He^1}\phi(x,0)m_0(x)dx$ is finite. Indeed, by point~$(i)$ of~\eqref{def:K_t_beta} and since $\beta\leq \beta_0$, we have
\begin{eqnarray*}
\int_{\He^1}|\phi(x,0)m_0(x)|dx&\leq&
C\int_{\He^1}|m_0(x)|\exp\{\beta_0\|x\|_\cH^2\}\exp\{(\beta-\beta_0)\|x\|_\cH^2\}dx\\
&\leq& C\int_{\He^1}|m_0(x)|\exp\{\beta_0\|x\|_\cH^2\}dx.
\end{eqnarray*}
Moreover, replacing~\eqref{pbC_duale} in~\eqref{eq:FPv_distr} with $i=1,2$  we obtain
\[
\int_{\He^1}\psi(x)m_1(t)(dx)=\int_{\He^1}\phi(x,0)m_0(x)dx=\int_{\He^1}\psi(x)m_2(t)(dx).
\]
By the arbitrariness of $\psi$ and $t$, we get $m_1=m_2$ in $\He^1\times[0,T_0]$. Iterating this argument on time intervals of length $T_0$, we accomplish the proof.

It remains to prove that the function~$\phi$ belongs to ${\mathcal K}_{T_0,\beta}$; in other words, we need to prove that: (a) $\phi$ verifies the bounds in points~(i) and~(ii) in definition~\eqref{def:K_t_beta}, (b) $\phi$ is a $C^{2,1}$ function.\\
(a). Let us prove point (i) in~\eqref{def:K_t_beta}. By~\cite[Theorem 10.7-(v)]{BBLU}, the function~$\phi$ can be written as
\[
\phi(s,x)=\int_{\He^1}h(t-s,x;0,\xi)\psi(\xi)d\xi
\]
for a suitable nonnegative  kernel~$h$.
The final datum $\psi$ in~\eqref{pbC_duale} belongs to $C^\infty_0(\He^1)$; hence  $\textrm{supp} \psi\subset B_\cH (0,K_\psi)$ for some positive constant~$K_\psi$.
Therefore, taking also advantage of the estimates in~\cite[Theorem 10.7-(iv)]{BBLU}, we deduce that, for some constants $C_1$ and $C_2$ (depending only on~$b$ and~$c$), there holds
\begin{eqnarray*}
|\phi(s,x)|&\leq&  C_1\int_{B_\cH (0,K_\psi)}\frac{1}{|B_\cH (x,C_2(t-s)^{1/2})|}\exp\left\{\frac{-d_\cH(x,\xi)^2}{C_2^2(t-s)}\right\}d\xi\\
&\leq&
 C_1\frac{|B_\cH (0,K_\psi)|}{|B_\cH (x, C_2(t-s)^{1/2})|}\exp\left\{-\frac{(\|x\|_\cH^2-R^2)\vee 0}{ C_2^2(t-s)}\right\}.
\end{eqnarray*}
We fix $T_0:=(C_2\beta)^{-1}$ and we obtain point (i) in~\eqref{def:K_t_beta}. The requirement (ii) in~\eqref{def:K_t_beta} can be obtained in a similar manner (taking advantage of the other estimates for~$h$ in~\cite[Theorem 10.7-(iv)]{BBLU}) so we shall omit its proof.

\noindent (ii). We already know: $\partial_t\phi ,X_i\phi, X_iX_j\phi\in C^{\delta}$ so, in particular, they are bounded continuous functions. We shall improve this regularity by a bootstrap argument. 
By equality \eqref{primo} we get that
$$X_1(\diver_\cH(b\phi)) = \diver_\cH(b X_1\phi) +  \diver_\cH(X_1b \phi)+2\partial_{x_3}(b_2 \phi)$$
$$X_2(\diver_\cH(b\phi)) = \diver_\cH(b X_2\phi) +  \diver_\cH(X_2b \phi)-2\partial_{x_3}(b_1 \phi)$$
Hence, taking account of \eqref{secondo} we get that the functions $\Phi_i:=X_i\phi$, $i=1,2$ are distributional solution in $\He^1\times(0,t)$ to
\begin{equation*}
\left\{\begin{array}{ll}
  \partial_t \Phi_1+\sigma\Delta_\cH\Phi_1=-\diver_\cH(b \Phi_1) - \diver_\cH((X_1b) \phi)-2\partial_{x_3}(b_2 \phi)- 4\sigma X_2(\partial_{x_3}\phi)-c\Phi_1-(X_1c)\phi,\\
    \partial_t \Phi_2+\sigma\Delta_\cH\Phi_2=-\diver_\cH(b \Phi_2) -  \diver_\cH((X_2b) \phi)+2\partial_{x_3}(b_1 \phi)+4\sigma X_1(\partial_{x_3}\phi)-c\Phi_2- (X_2c)\phi,\\
\Phi_i(t,x)=X_i\psi(x)\qquad\textrm{on }\He^1
\end{array}\right.
\end{equation*}
The equation satisfied by $\Phi_3:=\partial_{x_3}\phi$ is 
$$ \partial_t \Phi_3+\sigma\Delta_\cH\Phi_3=-\diver_\cH(b\Phi_3)-\diver_\cH(\partial_{x_3}b \phi) +c\Phi_3+(\partial_{x_3}c)\phi.$$
Arguing as in (ii) proof of Lemma 5.1 we get that $\partial_{x_3}\phi\in C_{\cH}^{2+\delta}$. 
Our assumptions and the above bounds for the kernel $h$ and its horizontal derivatives ensures that the right-hand side of the equations satisfied by $\Phi_i:=X_i\phi$, $i=1,2$
belong to $C^\delta$. Therefore, applying~\cite[Theorem 18]{RS} (see also \cite[Theorem 16-(b)]{RS}), we get $\Phi_i\in C^{1+\delta}$ and, consequently, $D^2\phi\in C^0$.
\end{proof}

\section{Probabilistic representation for the continuity equation}

This appendix is devoted to adapt the results in~\cite[Theorem 8.2.1]{AGS} to the case of a continuity equation expressed in terms of the vector fields generating the Heisenberg group and with a drift $D_\cH u$ which is bounded and $Q_\cH$-periodic in the sense of section~\ref{sub:periodicity}. As a matter of facts, in our case, the statement of \cite[Theorem 8.2.1]{AGS} does not apply because the sommability assumption \cite[equation (8.1.21)]{AGS} for the drift (which reads $D_{\cH}u\,B^T$ in Euclidean coordinates) does not hold. To get the probabilistic representation of the solution of the continuity equation
\eqref{eq:MFGintrin}
the key ingredient is a ``superposition principle'' (see \eqref{superA}) which allows to prove that there exists a probability measure concentrated on the solutions of the ODE associated to the optimal synthesis \eqref{OS}.
To get this superposition principle the key results are Lemma \ref{nostro819} and Lemma \ref{8110nostro} where we strongly use the properties of the distance associated to the Heisenberg group and of the pavage to represent $\mathbb{H}^{1}$. 

Throughout this section, we shall only study $Q_\cH$-periodic solution~$m$ to~\eqref{eq:MFGintrin}-(ii) and we shall write ``a.e.'' without specifying the measure when we intend ``a.e. with respect to the Lebesgue measure''.

We observe that $m$ is a $Q_\cH$-periodic solution of \eqref{eq:MFGintrin}-(ii) in the sense of distributions in $\mathbb{H}^{1}$ means
\begin{equation}\label{813nostra}
\int_0^T\int_{\mathbb{H}^{1}}(\partial_t\varphi- D_{\cH}u\cdot D_{\cH}\varphi)dm_t(x)dt=0\qquad \forall \varphi\in C_{c}^{\infty}(\mathbb{H}^{1}\times (0,T)).
\end{equation}
\\
Choosing $\varphi(t,x)=\eta(t)\zeta(x)$ with $\eta\in C_c^{\infty}(0,T)$, by density, we get the following equivalent formulation of  \eqref{813nostra}:
\begin{equation}\label{814nostra}
\frac{d}{dt} \int_{\mathbb{H}^{1}}\zeta(x)dm_t(x)=-\int_{\mathbb{H}^{1}}D_{\cH}u\cdot D_{\cH}\zeta(x)dm_t(x)
\end{equation}
 for any $\zeta\in C_{c}^{\infty}(\mathbb{H}^{1})$, in the sense of distribution in $(0,T)$.\\
Note that, by periodicity, $m$ is a solution of \eqref{eq:MFGintrin}-(ii) in the sense of distributions in $(0,T)$ also over $\Te$, i.e.
\begin{equation}\label{814nostraQH}
\frac{d}{dt} \int_{\Te}\zeta(x)dm_t(x)=-\int_{\Te}D_{\cH}u\cdot D_{\cH}\zeta(x)dm_t(x),\qquad  \forall \zeta\in C^{\infty}(\Te).
\end{equation}
The following lemma ensures that any $Q_\cH$-periodic distributional solution to~\eqref{eq:MFGintrin}-(ii) (or, equivalently to \eqref{814nostra} or to \eqref{814nostraQH}) has a representative in $C([0,T],{\mathcal P}_{per}(\He^1))$ which will be always called $m$.
\begin{lemma}\label{8.1.2}
(Continuous representative). Let $m_{t}$ be a Borel family of probability measures $Q_\cH$-periodic satisfying 
\eqref{814nostraQH}.
Then there exists a narrowly continuous curve $t \in[0, T] \mapsto \tilde{m}_{t} \in \mathcal{P}\left(\Te\right)$ such that $m_{t}=\tilde{m}_{t}$ for a.e. $t \in(0, T).$
Moreover, if $\varphi \in C_{{\cH}}^{1,1}\left(\Te\times[0, T]\right)$ and $t_{1} \leq t_{2} \in[0, T]$ we have
\begin{equation}\label{815nostra}
\int_{\Te} \varphi\left(x, t_{2}\right) d \tilde{m}_{t_{2}}(x) -\int_{\Te} \varphi\left(x, t_{1}\right) d \tilde{m}_{t_{1}}(x)
=\int_{t_{1}}^{t_{2}} \int_{\Te}\left(\partial_{t} \varphi+D_{\cH} \varphi\cdot D_{\cH}u\right) d m_{t}(x) d t.
\end{equation}
\end{lemma}
\begin{proof}
From \eqref{814nostraQH} we get that, for any $\zeta \in C^{\infty}(\Te)$ 
$$
t \mapsto m_{t}(\zeta)=\int_{\Te} \zeta(x) d m_{t}(x) \in W^{1,1}(0, T)
$$
with distributional derivative
\begin{equation*}
\frac{d}{dt}{m}_{t}(\zeta)=-\int_{\Te} D_{\cH} \zeta(x)\cdot D_{\cH}u(x,t) d m_{t}(x) \quad \text { for  a.e. } t \in(0, T);
\end{equation*}
so, since $m_t$ is a measure on~$\Te$, by the boundedness of $D_{\cH} u$, we deduce
\begin{equation*}
\left|\frac{d}{dt}{m}_{t}(\zeta)\right| \leq \|D_{\cH} u\|_{\infty,\Te} \|D_{\cH} \zeta\|_{\infty,\Te}. 
\end{equation*}
Following the proof of \cite[lemma 8.1.2]{AGS}, we get that $m_t$ can be extended in a unique way to a continuous curve $\left\{\tilde{m}_{t}\right\}_{t \in[0, T]}$ in $\mathcal{P}\left(\Te\right)$ and also that  \eqref{815nostra} holds. 
Note that in our case the compactness of $\Te$ yields directly the tightness of the family $m_t$.
\end{proof}

\begin{lemma}\label{8.1.3}
Let $\mathrm{t}:s \in\left[0, T^{\prime}\right] \rightarrow \mathrm{t}(s) \in[0, T]$ be a strictly increasing absolutely continuous map with absolutely continuous inverse $\mathrm{s}:=\mathrm{t}^{-1}$.
Then $m_{t}$ is a distributional solution of \eqref{eq:MFGintrin}-(ii) with drift $D_{\cH}u$ if and only if $\hat{m}:=m \circ \mathrm{t}$,  is a distributional solution of \eqref{eq:MFGintrin}-(ii) on $\left(0, T^{\prime}\right)$ with drift $\mathrm{t}^{\prime} D_{\cH}u \circ \mathrm{t}$.
\end{lemma}
\begin{proof}
The proof is analogous to that of \cite[Lemma 8.1.3]{AGS} by replacing $D\hat{\varphi}$ with $D_{\cH}\hat{\varphi}$, where $\hat{\varphi}\in C_{\cH, c}^{1,1}\left(\mathbb{H}^{1} \times\left(0, T^{\prime}\right)\right)$. 
\end{proof}
When the drift $v_{t}$ in equation \eqref{eq:MFGintrin}-(ii) satisfies 
\begin{equation}\label{818nostra}
\int_{0}^{T}\operatorname{Lip}\left(v_{t}, K\right)d t<+\infty
\end{equation}
where $K$ is any compact set of $\He^1$, we can obtain an explicit solution of \eqref{eq:MFGintrin}-(ii) by
the classical method of characteristics as proved in Proposition~\ref{818}. To obtain the needed regularity we approximate $v_t$ and $m_t$ with $v_t^{\epsilon}$ and $m_t^{\epsilon}$ by means of a family of mollifiers (see Section \ref{sec:conv}). For $m_t^{\epsilon}$ solution of the continuity equation \eqref{eq:MFGintrin}-(ii) with drift $v_t^{\epsilon}$,
we can get a representation formula. 
The following two Lemma provide the approximation with the needed regularity to obtain the explicit formula proved in Proposition \ref{818}.

\begin{lemma}\label{nostro819}(Approximation by regular curves)
Let $m_{t}$ be a time continuous solution of \eqref{eq:MFGintrin}-(ii). 
Let $\left(\rho_{\varepsilon}\right) \subset C^{\infty}(\mathbb{R}^{3})$ be the family of strictly positive mollifiers in the $x$ variable, defined in \ref{rhoeps} and set, by the convolution defined in \eqref{conH}-\eqref{rhoeps}
$$
m_{t}^{\varepsilon}:=m_{t} \ast \rho_{\varepsilon}, \quad E_{t}^{\varepsilon}:=\left(D_{\cH}u\, m_{t}\right) \ast\rho_{\varepsilon}, \quad v_{t}^{\varepsilon}:=\frac{E_{t}^{\varepsilon}}{m_{t}^{\varepsilon}}.
$$
Then $m_{t}^{\varepsilon}$, $E_{t}^{\varepsilon}$ and $v_{t}^{\varepsilon}$ are $Q_\cH$-periodic. Moreover $m_{t}^{\varepsilon}$ is a continuous solution of \eqref{eq:MFGintrin}-(ii) with drift $v_{t}^{\varepsilon}$: 
\begin{equation}\label{8128nostra} 
\partial_{t} m_{t}^{\varepsilon}-\diver_{\cH}(v_{t}^{\varepsilon}\, m_{t}^{\varepsilon})=0,\qquad in\ \mathbb{H}^{1} \times(0, T),
\end{equation}
where $v_{t}^{\varepsilon}$ fulfills the regularity property~\eqref{818nostra} and the uniform integrability bound
\begin{equation}\label{8126nostra}
\int_{\Te}\left|v_{t}^{\varepsilon}(x)\right|^{p} d m_{t}^{\varepsilon}(x) \leq C, \quad \forall t \in(0, T),\ p \geq 1.
\end{equation}
Moreover, as $\varepsilon\to0^+$, $E_{t}^{\varepsilon} \rightarrow v_{t} m_{t}$ narrowly and
\begin{equation}\label{8127nostra}
\lim _{\varepsilon \to 0}\left\|v_{t}^{\varepsilon}\right\|_{L^p\left(m_{t} ; \Te\right)}=\left\|D_{\cH}u(\cdot,t)\right\|_{L^p\left(m_{t} ; \Te\right)} \qquad \forall t \in(0, T).
\end{equation}
\end{lemma}
\begin{proof}
Note that, from Proposition~\ref{propconvH}-(i), $m_{t}^{\varepsilon}$, $E_{t}^{\varepsilon}$ and $v_{t}^{\varepsilon}$ are $Q_\cH$-periodic.
From Proposition~\ref{propconvH}-(v) and the continuity of $m_{t}^{\varepsilon}(x)$ w.r.t. $x$ and $t$, we get  
  $$m_{t}^{\varepsilon}(x)>0,\ {\text{ for any }} x\in \Te\ {\text{ and any }} t\in[0,T].$$
From the definition of $\rho_ {\varepsilon}$, since $m_{t}$ is bounded then $\left|m_{t}^{\varepsilon}\right|(t, \cdot)$ is bounded. From the definition of the $\He^1$-norm \eqref{norm}
we get that
$$D\rho_ {\varepsilon}(x)=C(\varepsilon)e^{-(\|\frac{x}{\epsilon}\|^4_\cH)}
\left(
\frac{4x_1(x_1^2+x_2^2)}{\varepsilon^4}, \frac{4x_2(x_1^2+x_2^2)}{\varepsilon^4}, \frac{2x_3}{\varepsilon^4}
\right).
$$
Hence, in $\Te$, the spatial gradient of $m_{t}^{\varepsilon}(t, \cdot)$ is bounded with a constant depending on $\varepsilon$.
Analogously, in $\Te$, $E^{\epsilon}(t, \cdot)$ and its spatial gradient are uniformly bounded in space by the product of $\left\|D_{\cH}u\right\|_{L^{1}\left(m_{t}\right)}$ with a constant depending on $\varepsilon$. 

Moreover, from the positivity of $m_{t}^{\varepsilon}$, 
the local regularity assumptions \eqref{818nostra} for $v_{t}^{\varepsilon}=E_{t}^{\varepsilon} / m_{t}^{\varepsilon}$ hold. 
Lemma \ref{8110nostro} shows that \eqref{8126nostra} holds.\\
From proposition~\eqref{propconvH}-(v), noting that
$X_i((m_{t} \,X_iu)\ast \rho_{\varepsilon})=X_i(m_{t} \,X_iu)\ast \rho_{\varepsilon}$, we get
$$\diver_{\cH}(v_{t} m_{t})\ast \rho_{\varepsilon}=\diver_{\cH}((v_{t}m_{t})\ast \rho_{\varepsilon})=
\diver_{\cH}E_{t}^{\varepsilon}=\diver_{\cH}(v_{t}^{\varepsilon} m_{t}^{\varepsilon}).$$
Since $m_{t}$ solves \eqref{eq:MFGintrin}-(ii), then
$$ \partial_t (m_{t} \ast \rho_{\varepsilon})+\diver_{\cH}(v_{t} m_{t})\ast \rho_{\varepsilon}=0.$$
Hence
$m_{t}^{\varepsilon}$ solves the continuity equation \eqref{8128nostra}.
Finally,
general lower semicontinuity results on integral functionals defined on measures of the form
$$
(E, m) \mapsto \int_{\Te}\left|\frac{E}{m}\right|^{p} d m
$$
and the following Lemma \ref{8110nostro}
give \eqref{8127nostra}.
\end{proof}

\begin{lemma}\label{8110nostro}
 Let $m, E \in \mathcal{P}(\Te)$, $E\in L^{\infty}(\Te)$ and absolutely continuous with respect to $m$.  Let $p \geq 1,$ Then
$$
\int_{\Te}\left|\frac{E * \rho}{m * \rho}\right|^{p} m * \rho\, d x \leq \int_{\Te}\left|\frac{E}{m}\right|^{p} d m
$$
for any positive convolution kernel $\rho$ (see Section \ref{sec:conv}).
\end{lemma}

\begin{proof}
Arguing as in the proof of \cite[Lemma 8.1.10]{AGS}, in particular by the Jensen inequality, for any $x\in \He^1$ we get
\begin{eqnarray*}
\left|\frac{E \ast\rho(x)}{m \ast \rho(x)}\right|^{p} m \ast \rho(x) 
&\leq&\int_{\mathbb{H}^{1}}\left|\frac{E}{m}\right|^{p}(y) \rho(x\ominus y) d m(y)
=\sum_{n\in\Z^3}\int_{n\oplus Q_\cH}\left|\frac{E}{m}\right|^{p}(y) \rho(x\ominus y)dm(y)\\
&=&\sum_{n\in\Z^3}\int_{Q_\cH}\left|\frac{E}{m}\right|^{p}(n\oplus z) \rho(x\ominus n\ominus z)dm(z)\\
&=&\int_{Q_\cH}\left|\frac{E}{m}\right|^{p}(z) \sum_{n\in\Z^3}\rho(x\ominus n\ominus z)dm(z)
\end{eqnarray*}
where we used the $\Te-$periodicity of $m$ and of ${E}/{m}$.
Integrating with respect to $x$ in $\Te$ we get
\begin{eqnarray*}
&&\int_{\Te}\left|\frac{E \ast\rho(x)}{m \ast\rho(x)}\right|^{p} m\ast \rho(x) dx
\leq 
\int_{\Te}\int_{\Te}\left|\frac{E}{m}\right|^{p}(z) \sum_{n\in\Z^3}\rho(x\ominus n\ominus z)dm(z)\,dx\\
&&=\int_{\Te}\left|\frac{E}{m}\right|^{p}(z)\bigg(\sum_{n\in\Z^3}\int_{\Te}\rho(x\ominus n\ominus z)dx\bigg) dm(z)= \int_{\Te}\left|\frac{E}{m}\right|^{p}(z)dm(z).
\end{eqnarray*}
The last equality comes from  
$$\sum_{n\in\Z^3}\int_{\Te}\rho(x\ominus n\ominus z)dx=\int_{\He^1}\rho(y)dy=1$$
and this equality
is due to the fact that, fixed $z\in \Te$,
$$\mathbb{H}^{1}= \cup_{n\in\Z^3} \Te\ominus n\ominus z.$$
To prove it we have to show that for any $y\in \mathbb{H}^{1}$ there exists an unique $n\in\Z^3$ such that there exists $x\in \Te$ such that $y= x\ominus n\ominus z$. We recall that, from the property of pavage defined at the beginning of Section \ref{sub:periodicity}, for any $a\in \mathbb{H}^{1}$ we denote by 
$n_{\cH}(a)$ the unique $n\in\Z^3$ such that there exists a unique point $x_a\in \Te$ such that $a=n_{\cH}(a)\oplus x_a$.
Hence there exists an unique $(n_{\cH, 1}, n_{\cH, 2}, n_{\cH, 3}) = n_{\cH}(y\ominus z)\in \Z^3$ such that there exists 
$x=(x_1,x_2,x_3)\in\Te$ such that $y\ominus z= n_{\cH}(y\ominus z)\oplus x$, i.e. 
$y_1-z_1= n_{\cH, 1}+x_1$, $y_2-z_2= n_{\cH, 2}+x_2$, $y_3-z_3+ y_1z_2-y_2z_1= n_{\cH, 3}+x_3-
n_{\cH, 1}x_2+n_{\cH, 2}x_1$.
To find the unique $n=(n_1, n_2, n_3)\in\Z^3$  such that $y= x\ominus n\ominus z$
we take $n_1= -n_{\cH, 1}$, $n_2= -n_{\cH,2}$ and 
$n_3=-n_{\cH, 3}+2(n_2x_1-n_1x_2)$.
%
\end{proof}


%
Now using an elementary result of the theory of ODEs, we obtain a maximal existence and uniqueness result for the characteristic system associated to equation~\eqref{8128nostra}.

\begin{lemma}\label{8.1.4}
Let~$v^\varepsilon$ be the field introduced in Lemma~\ref{nostro819}. Then for any $x\in\He^1$ and $s\in[0,T]$, the ODE
\begin{equation}\label{819nostra}
\frac{d}{d t} Y_{t}(x, s)=v^{\epsilon}_{t}\left(Y_{t}(x, s)\right)\, B^T\left(Y_{t}(x, s)\right),\qquad Y_{s}(x, s)=x
\end{equation}
admits a unique maximal solution which is defined in~$[0,T]$.
\end{lemma}
\begin{proof} The results in~\cite[Lemma 8.1.4]{AGS} ensure that there exists a unique maximal solution to~\eqref{819nostra}, defined on some interval~$I$, relatively open in~$[0,T]$ and containing~$s$ as relatively internal point. Moreover,~\eqref{819nostra} reads
\[
Y'_{1,t}= v^{\epsilon}_{1,t},\qquad
Y'_{2,t}= v^{\epsilon}_{2,t},\qquad
Y'_{3,t}= -Y_{2,t}v^{\epsilon}_{1,t}+Y_{1,t}v^{\epsilon}_{2,t}.
\]
By the boundedness of $v^\varepsilon$, we get that the first two components of $Y^{\epsilon}_{t}(x, s)$ are bounded in~$I$ and, afterwards, we deduce the boundedness of the third component. Applying again~\cite[Lemma 8.1.4]{AGS}, we conclude that~$I$ coincides with the whole interval~$[0,T]$.
\end{proof}

For simplicity, we set $Y_{t}(x):=Y_{t}(x, 0)$ in the particular case $s=0$.

\begin{remark}\label{rk815}
Characteristics provide a useful representation formula for classical solutions of the equation which is formally the adjoint to \eqref{eq:MFGintrin}-(ii):
\begin{equation}\label{8112nostra}
\partial_{t} \varphi - v^{\epsilon}_t\cdot D_{\cH} \varphi=\psi \quad \text { in } \mathbb{H}^{1} \times(0, T), \quad \varphi(x, T)=\varphi_{T}(x) \quad x \in \mathbb{H}^{1}
\end{equation}
with $\psi \in C_{b, \cH}^{1}\left(\mathbb{H}^{1} \times(0, T)\right)$, $\varphi_{T} \in C_{b, \cH}^{1}\left(\mathbb{H}^{1}\right)$.
 A direct calculation shows that, if $Y^{\epsilon}_{s}(x, t)$ solves~\eqref{819nostra}, then
 \begin{equation}\label{8113nostra}
\varphi(x, t):=\varphi_{T}\left(Y^{\epsilon}_{T}(x, t)\right)-\int_{t}^{T} \psi\left(Y^{\epsilon}_{s}(x, t), s\right) d s
\end{equation}
solves \eqref{8112nostra}. Indeed $Y^{\epsilon}_{s}\left(Y^{\epsilon}_{t}(x, 0), t\right)=Y^{\epsilon}_{s}(x, 0)$ yields
$$
\varphi\left(Y^{\epsilon}_{t}(x, 0), t\right)=\varphi_{T}\left(Y^{\epsilon}_{T}(x, 0)\right)-\int_{t}^{T} \psi\left(Y^{\epsilon}_{s}(x, 0), s\right) d s
$$
and differentiating both sides with respect to $t$ we obtain
$$
\left[\frac{\partial \varphi}{\partial t}-  v^{\epsilon}_tB^T \cdot D \varphi\right] \left(Y^{\epsilon}_{t}(x, 0), t\right)=\psi\left(Y^{\epsilon}_{t}(x, 0), t\right).
$$
Noting $v^{\epsilon}_tB^T \cdot D \varphi = v^{\epsilon}_t \cdot D_{\cH}\varphi$, by the arbitrariness of $x$ (and then $Y_{t}(x, 0)$), we conclude that \eqref{8112nostra} is fulfilled. 
\end{remark}
Now we use characteristics to prove the existence, the uniqueness, and a representation formula of the solution of the continuity equation \eqref{8128nostra}. 

\begin{lemma}\label{8.1.6}
For any $m_{0} \in$ $\mathcal{P}_{per}(\He^1)$, let  $m_{0}^{\epsilon}$ denote $m_0*\rho_\epsilon$ where the kernel $\rho_\epsilon$ has been introduced in~\eqref{rhoeps}. 
Let $Y^{\epsilon}_{t}$ be the solution of~\eqref{819nostra} (corresponding to $s=0$).
Then $t \mapsto m^{\epsilon}_{t}:=Y^{\epsilon}_{t}\# m^{\epsilon}_{0}$ is a continuous (in the topology of $C([0,T],{\mathcal P}_{per}(\He^1))$)  solution of \eqref{8128nostra} in $[0, T]$.
\end{lemma}
\begin{proof}
Note that, from the boundedness of $D_{\cH}u$, the velocity field $v^{\epsilon}_{t}$ satisfies~\eqref{818nostra} and \eqref{8126nostra}.
The continuity of $m^{\epsilon}_{t}$ follows easily since $\lim _{s \rightarrow t} Y^{\epsilon}_{s}(x)=Y^{\epsilon}_{t}(x)$ for $m^{\epsilon}_{0}$-a.e.
$x \in \mathbb{H}^{1}$: 
thus for every continuous and bounded function $\zeta: \mathbb{H}^{1} \rightarrow \mathbb{R}$ the dominated convergence theorem gives
$$
\lim _{s \rightarrow t} \int_{\mathbb{H}^{1}} \zeta d m^{\epsilon}_{s}=\lim _{s \rightarrow t} \int_{\mathbb{H}^{1}} \zeta\left(Y^{\epsilon}_{s}(x)\right) d m^{\epsilon}_{0}(x)=\int_{\mathbb{H}^{1}} \zeta\left(Y^{\epsilon}_{t}(x)\right) d m^{\epsilon}_{0}(x)=\int_{\mathbb{H}^{1}} \zeta d m^{\epsilon}_{t}.
$$
For any $\varphi \in C_c^{\infty}\left(\Te\times(0, T)\right)$ and for $m^{\epsilon}_{0}$-a.e. $x \in \Te$ the maps $t \mapsto \phi_{t}(x):=$
$\varphi\left(Y^{\epsilon}_{t}(x), t\right)$ are absolutely continuous in $(0, T)$ and
\begin{eqnarray*}
&&\dot{\phi}_{t}(x)=\partial_{t} \varphi\left(Y^{\epsilon}_{t}(x), t\right)+\left\langle D\varphi\left(Y^{\epsilon}_{t}(x), t\right), v^{\epsilon}_{t}\left(Y^{\epsilon}_{t}(x)\right)B\left(Y^{\epsilon}_{t}(x)\right)\right\rangle=\\
&&\partial_{t} \varphi\left(Y^{\epsilon}_{t}(x), t\right)+\left\langle D_{\cH}\varphi\left(Y^{\epsilon}_{t}(x), t\right), v^{\epsilon}_{t}\left(Y^{\epsilon}_{t}(x)\right)\right\rangle=\Lambda(\cdot, t) \circ Y^{\epsilon}_{t}
\end{eqnarray*}
where $\Lambda(x, t):=\partial_{t} \varphi(x, t)+\left\langle D_{\cH} \varphi(x, t), v^{\epsilon}_{t}(x)\right\rangle .$ We thus have
$$
\begin{aligned}
\int_{0}^{T} \int_{\He^1}\left|\dot{\phi}_{t}(x)\right| d m^{\epsilon}_{0}(x) d t &=\int_{0}^{T} \int_{\He^1}\left|\Lambda\left(Y_{t}(x), t\right)\right| d m^{\epsilon}_{0}(x) d t \\
&=\int_{0}^{T} \int_{\He^1}|\Lambda(x, t)| d m^{\epsilon}_{t}(x) d t \\
& \leq \operatorname{Lip_{\cH}}(\varphi)\left(T+\int_{0}^{T} \int_{\He^1}\left|v^{\epsilon}_{t}(x)\right| d m^{\epsilon}_{t}(x) d t\right)<+\infty
\end{aligned}
$$
where the boundedness of the last integral comes from the fact that we can cover the compact support of $\varphi$  with a finite number of elements of the pavage where $|v^{\epsilon}_{t}|$ is bounded.
Therefore
$$
\begin{aligned}
0 &=\int_{\He^1} \varphi(x, T) d m^{\epsilon}_{T}(x)-\int_{\He^1} \varphi(x, 0) d m^{\epsilon}_{0}(x)=\int_{\He^1}\left(\varphi\left(Y^{\epsilon}_{T}(x), T\right)-\varphi(x, 0)\right) d m^{\epsilon}_{0}(x) \\
&=\int_{\He^1}\left(\int_{0}^T \dot{\phi}_{t}(x) d t\right) d m^{\epsilon}_{0}(x)=\int_{0}^T \int_{\He^1}\left(\partial_{t} \varphi+ D_{\cH} \varphi \cdot v^{\epsilon}_{t}\right) dm^{\epsilon}_{t} d t
\end{aligned}
$$
by a simple application of Fubini's theorem, i.e. \eqref{8128nostra} holds.
\end{proof}
We want to prove that any solution of \eqref{8128nostra} can be represented as in Lemma \ref{8.1.6}. 
\begin{proposition}\label{817} (Uniqueness and comparison for the continuity equation). Let
$\sigma_{t}$ be a narrowly continuous family of signed $\Te$-periodic measures solving $\partial_{t} \sigma_{t}+\diver_{\cH} \cdot\left(v^{\epsilon}_{t} \sigma_{t}\right)=0$ in $\He^1 \times(0, T),$ with $\sigma_{0} \leq 0$.
Then $\sigma_{t} \leq 0$ for any $t \in[0, T]$.
\end{proposition} 
\begin{proof}
The proof is the same as the one for \cite[Proposition 8.1.7]{AGS} where we replace $\mathbb{R}^{d}$ with $\He^1$ and the Euclidean gradient $D$ with $D_{\cH}$. Observe that, from the boundedness of the field $v_t^{\epsilon}$, we have 
$\int_{0}^{T} \int_{\Te}\left|v^{\epsilon}_{t}\right| d\left|\sigma_{t}\right| d t<+\infty.$
Moreover covering any compact set $C$ with a finite number of elements of the pavage, we get
$$
\int_{0}^{T}\left(\left|\sigma_{t}\right|(C)+\sup _{C}\left|v^{\epsilon}_{t}\right|+\operatorname{Lip}\left(v^{\epsilon}_{t}, C\right)\right) d t<+\infty
$$
for any bounded closed set $C\subset \mathbb{H}^{1}$.
\end{proof}
\begin{proposition}\label{818} (Representation formula for the continuity equation). Let $m^{\epsilon}_{t}\in {\mathcal P}_{per}(\He^1)$, $t \in[0, T]$, be a family of narrowly continuous measures solving the continuity equation~\eqref{8128nostra}.
Then for $m_{0}$-a.e. $x \in \He^1$ the characteristic system \eqref{819nostra} admits a globally defined solution $Y^{\epsilon}_{t}(x)$ in $[0, T]$ and
\begin{equation}\label{8120nostra}
m^{\epsilon}_{t}=Y^{\epsilon}_{t}{\#} m^{\epsilon}_{0}, \quad \forall t \in[0, T].
\end{equation}
\end{proposition}
\begin{proof} 
Recall that
$v^{\epsilon}_{t}$ satisfies \eqref{818nostra}.
Moreover, by Lemma~\ref{8.1.4}, $Y^{\epsilon}_{t}$ is globally defined in $[0, T]$ for $m_{0}$-a.e. in $\mathbb{H}^{1}$. Applying Lemma \ref{8.1.6} and Proposition \ref{817} we obtain $\eqref{8120nostra}$.
\end{proof}
Now we want to extend Proposition \ref{818} to the continuity equation \eqref{eq:MFGintrin}-(ii), where the vector field $D_\cH u$ does not satisfy the local regularity assumptions \eqref{818nostra} but it is still bounded and $Q_\cH$-periodic. In this situation we consider suitable probability measures in the space $\Gamma$ of the absolutely continuous maps 
from $[0,T]$ to $\mathbb{H}^{1}$, see definition \eqref{gamma}.

Our representation formula for the periodic solutions $m_{t}^{\eta}$ of the continuity equation
\eqref{eq:MFGintrin}-(ii) is given by
\begin{equation}\label{821nostra}
\int_{\Te} \varphi d m_{t}^{\eta}:=\int_{\Te \times \Gamma} \varphi(\gamma(t)) d \eta(x, \gamma) \quad \forall \varphi \in C^{0}(\Te), t \in[0, T],
\end{equation}
where $\eta$ is a suitable probability measure in $\Te \times \Gamma$.
With a slight abuse of notations, we denote $e_t$ as in \eqref{eval} also the evaluation map $e_t:\Te\times \Gamma\rightarrow \Te$ with $e_t(x,\gamma)=\gamma(t)$. Hence,  \eqref{821nostra} can be written as
\begin{equation}\label{823nostra}
m_{t}^{\eta}=e_{t}{\#} \eta.
\end{equation}
\begin{theorem}\label{821} (Probabilistic representation). 
Let $m:[0, T] \rightarrow \mathcal{P}(\Te)$ be a narrowly continuous solution of the continuity equation \eqref{eq:MFGintrin}-(ii). Then there exists a probability measure $\eta$ in $\Te\times \Gamma$, 
 such that\\
(i) $\eta$ is concentrated on the set of pairs $(x, \gamma)$ such that $\gamma \in \Gamma$ is a solution of the differential equation 
\begin{equation}\label{ODE}
\dot{\gamma}(t)=-D_{\cH}u(\gamma(t), t)B^T(\gamma(t))\  for\ a.e.\ t \in(0, T),\ \gamma(0)=x.
\end{equation}
(ii) $m_{t}=m_{t}^{\eta}$ for any $t \in[0, T],$ with $m_{t}^{\eta}$ is defined in \eqref{821nostra}.\\
Conversely, any $\eta$ satisfying (i) induces via \eqref{821nostra} a solution of the continuity equation, with $m_{0}=e_0 \# \eta$.
\end{theorem}
\begin{proof}
We adapt the arguments of the proof of \cite[Theorem 8.2.1]{AGS}. We first prove the converse implication. Notice that due to (i), we have
$$
\dot{\gamma}(t)=D_{\cH}u(\gamma(t), t)B^T(\gamma(t)) \quad \eta-a.e., \text { for } \text {a.e. } t \in(0, T).
$$
From \eqref{821nostra} we deduce that $t \mapsto m_{t}^{\eta}$ is narrowly continuous; actually, for every $\varphi \in C^{0}(\Te)$ and $t \in[0, T]$, there holds
\begin{equation}\label{narrcont}
\int_{\Te} \varphi d m_{t}^{\eta}-\lim_{s\to t}\int_{\Te} \varphi d m_{s}^{\eta}
=\lim_{s\to t}\left(\int_{\Te \times \Gamma} \varphi(\gamma(t)) d \eta(x, \gamma) - \int_{\Te \times \Gamma} \varphi(\gamma(s)) d \eta(x, \gamma)\right)=0.
\end{equation}

Now we check that $t \mapsto \int \zeta d m_{t}^{\eta}$ is absolutely continuous for $\zeta \in C_{\cH}^{1}(\Te)$ bounded and with a bounded horizontal gradient $D_{\cH}\zeta$. Indeed, from \eqref{narrcont}, since $ D\zeta\cdot D_{\cH}uB^T= D_{\cH}\zeta\cdot D_{\cH}u$, for $s<t$ in $(0,T)$, we have
\begin{eqnarray*}
\left|\int_{\Te} \zeta d m_{s}^{\eta}-\int_{\Te} \zeta d m_{t}^{\eta}\right| &\leq& \int_{s}^{t} \int_{\Te  \times \Gamma}| D\zeta(\gamma(\tau))\cdot \dot{\gamma}(\tau)| d \eta\, d \tau \\
&=& \int_s^t\int_{\Te \times \Gamma}| D\zeta(\gamma(\tau))\cdot D_{\cH}u(\gamma(\tau), \tau)B^T(\gamma(\tau))| d \eta\, d \tau\\
&=&\int_s^t\int_{\Te \times \Gamma}| D_{\cH}\zeta(\gamma(\tau))\cdot D_{\cH}u(\gamma(\tau), \tau))| d \eta\, d \tau\\
&&\leq\|D_{\cH} \zeta\|_{\infty} \int_{s}^{t} \int_{\Te \times \Gamma}\left|D_{\cH}u(\gamma(\tau), \tau))\right| d \eta\, d \tau.
\end{eqnarray*}
Since $D_{\cH}u$ is bounded, the inequality gives the absolute continuity of the map. 
We have also
$$\frac{d}{d t} \int_{\Te} \zeta d m_{t}^{\eta}=\frac{d}{d t} \int_{\Te \times \Gamma} \zeta(\gamma(t)) d \eta 
=\int_{\Te\times \Gamma}  D \zeta(\gamma)\cdot \dot{\gamma}(t)\, d \eta=\int_{\Te} D_{\cH} \zeta\cdot D_{\cH}u \,d m_{t}^{\eta},
$$
for a.e. $t \in(0, T)$. Since this pointwise derivative is also a distributional one, this proves that $\eqref{814nostraQH}$ holds for test function $\varphi$ of the form $\zeta(x) \psi(t)$ and therefore for all test functions.

Conversely, for $m_{t}$ as in the statement, let us apply Lemma~\ref{nostro819} finding $Q_\cH$-periodic approximations $m_{t}^{\varepsilon}, v_{t}^{\varepsilon}$ satisfying the continuity equation~\eqref{8128nostra}.
Therefore, we can apply Proposition~\ref{818}, obtaining the representation formula $m_{t}^{\varepsilon}=Y_{t}^{\varepsilon}\# m^{\varepsilon}_{0},$ where $Y_{t}^{\varepsilon}$ is the flow of maximal solution of~\eqref{819nostra} with $s=0$. \\
Since $Y^\varepsilon$ induces naturally a map from~$\Te$ to~$\Gamma$, we define the measure $\eta^\varepsilon\in {\mathcal P}(\Te\times \Gamma)$ as $\eta^\varepsilon:= (i\times Y^\varepsilon)\#m_0^\varepsilon$ where $(i\times Y^\varepsilon):\Te\to \Te\times \Gamma$ with $(i\times Y^\varepsilon)(x):=(x,Y^\varepsilon_\cdot (x,0))$ where $Y^\varepsilon_\cdot (x,0)$ denotes the maximal solution to~\eqref{819nostra} with $Y^\varepsilon_0 (x,0)=x$. In other words, for any Borel function~$\phi$ defined in~$\Te\times\Gamma$, the measure~$\eta^\varepsilon$ verifies
\begin{equation}\label{i1}
\int_{\Te\times\Gamma}\phi(x,\gamma)d \eta^\varepsilon(x,\gamma)=\int_{\Te}\phi(x,Y^\varepsilon_\cdot (x,0))d m_0^\varepsilon(x).
\end{equation}
Now we claim that $(\eta^\varepsilon)$ is a relatively compact family of measures on $\Te\times\Gamma$. Indeed, we set
\[
C:=\{(x,\gamma)\in \Te\times\Gamma\mid \gamma(0)=0,\quad\|\gamma'\|_\infty\leq \beta\}
\]
where $\beta$ is a positive constant such that the solution to~\eqref{819nostra} with~$x\in\Te$ and $s=0$ satisfies $\|\dot Y_t^\varepsilon(x,0)\|_\infty\leq\beta$.
We observe that
\begin{eqnarray*}
\eta^\varepsilon(C)=\int_{\Te} \chi_C(x,Y^\varepsilon_\cdot(x,0))d m^\varepsilon_0(x)=\int_{\{x\in\Te\mid \ \|Y^\varepsilon_\cdot(x,0)\|\leq \beta\}}d m^\varepsilon_0(x)=1.
\end{eqnarray*}
Invoking Prokhorov theorem, there exists a subsequence of~$\{\eta^\varepsilon\}_{\varepsilon\in(0,1)}$ which narrowly converges. Hence our claim is completely proved.

Now, let $\eta$ be a narrow cluster point of~$\{\eta^{\varepsilon}\}_{\varepsilon}$. We claim $m_t=e_t\#\eta$ and that $m_0$ is the first marginal of~$\eta$. Indeed, by the definition of~$e_t$ (recall: $e_t:\Te\times\Gamma\to \Te$ with $e_t(x,\gamma)=\gamma(t)$) and~\eqref{i1}, for every $\phi\in C^0_{b}(\Te)$ and $t\in[0,T]$, there holds
\begin{eqnarray*}
\int_{\Te} \varphi(x) d(e_t\#\eta^\varepsilon)(x)&=& \int_{\Te\times\Gamma} \varphi(\gamma(t)) d\eta^\varepsilon(x,\gamma) = \int_{\Te} \varphi(Y^\varepsilon_t(x,0))d m^\varepsilon_0(x)\\
& =& \int_{\Te} \varphi(x)d m^\varepsilon_t(x)
\end{eqnarray*}
where the last equality is due to $m^\varepsilon_t=Y^\varepsilon_t\#m^\varepsilon_0$. Passing to the limit in the previous equality, we obtain $m_t=e_t\#\eta$ namely
\begin{equation}\label{eta}
\int_{\Te\times \Gamma} (\varphi\circ e_t) d \eta(x, \gamma)=\int_{\Te} \varphi(x) d m_{t}(x), \quad \forall \varphi \in C_b^{0}(\Te).
\end{equation}
Moreover, again by~\eqref{i1}, we have
\[
\int_{\Te\times\Gamma}\varphi(x)d \eta^\varepsilon(x,\gamma)=\int_{\Te}\varphi(x)d m^\varepsilon_0(x)
\]
and, passing to the limit as $\varepsilon\to0$, we get
\[
\int_{\Te\times\Gamma}\varphi(x)d \eta(x,\gamma)=\int_{\Te}\varphi(x)d m_0(x)
\]
namely $m_0$ is the first marginal of~$\eta$. So our claim is completely proved.\\
%
%
Now we have to show that $\eta$ is concentrated on solutions of the differential equation \eqref{ODE}.
We claim the following ``superposition principle''
\begin{equation}\label{superA}
\int_{\Te\times \Gamma}\left|\gamma(t)-x-\int_{0}^{t} D_{\cH}u(\gamma(\tau), \tau)\, B^T(\gamma(\tau))d \tau\right| d \eta(x, \gamma)=0 \quad \forall t \in[0, T].
\end{equation}
If the claim is true then we disintegrate $\eta$ with respect to its first marginal $m_0$ (see~\cite[pag 122]{AGS} or \cite[Theorem 8.5]{C}):
\begin{equation}\label{dis}
d \eta(x, \gamma)=d\eta_x(\gamma)\,dm_0(x)
\end{equation}
and from \eqref{superA} we get for $m_0$-a.e. $x\in \Te$, $\eta_x$-a.e. $\gamma$ is a solution of the \eqref{ODE}.\\
It remains to prove the claim~\eqref{superA}.
First of all we prove 
\begin{equation}\label{41}
\int_{\Te \times \Gamma}\left|\gamma(t)-x-\int_{0}^{t} w(\gamma(\tau), \tau)\, B^T(\gamma(\tau)) d \tau\right| d \eta(x, \gamma)\leq C\int_{0}^{T} \int_{\Te}\left|(D_{\cH}u-w)\right| d m_{t} d \tau,
\end{equation}
where $w(x,t)$ is a $Q_\cH$-periodic vector field, bounded and continuous w.r.t. $x$. We have 
\begin{eqnarray*}
&& \int_{\Te\times \Gamma}\left|\gamma(t)-x-\int_{0}^{t} w(\gamma(\tau), \tau)\, B^T(\gamma(\tau))d \tau\right| d \eta^{\varepsilon}(x, \gamma)\\
&&= \int_{\Te}\left|Y_{t}^{\varepsilon}(x)-x-\int_{0}^{t} w\left(Y_{\tau}^{\varepsilon}(x),\tau\right)\, B^T(Y_{\tau}^{\varepsilon}(x)) d \tau\right|d m_{0}^{\varepsilon}(x)\\ 
&&= \int_{\Te}\left|\int_{0}^{t}\left(v^{\varepsilon}-w\right)\left(Y_{\tau}^{\varepsilon}(x), \tau\right)\,B^T(Y_{\tau}^{\varepsilon}(x))d \tau\right|d m_{0}^{\varepsilon}(x)\\
&&\leq
 \int_{\Te}\int_{0}^{t}\left|\left(v^{\varepsilon}-w\right)\left(Y_{\tau}^{\varepsilon}(x), \tau\right)\,B^T(Y_{\tau}^{\varepsilon}(x))\right| d \tau d m_{0}^{\varepsilon}(x)\\
&&=\int_{0}^{t} \int_{\Te}\left|(v^{\varepsilon}-w)\,B^T\right| d m_{\tau}^{\varepsilon} d \tau, 
\end{eqnarray*}
where $Y_{t}^{\varepsilon}(x)$ is the solution of \eqref{819nostra}.
Setting $w^{\epsilon}:=\frac{\left(w m\right) * \rho_{\varepsilon}}{m^{\varepsilon}}$ we obtain
\begin{eqnarray*}
&&\int_{0}^{t} \int_{\Te}\left|(v^{\varepsilon}-w)\, B^T\right| d m_{\tau}^{\varepsilon} d \tau\\
&&\leq C\int_{0}^{t} \int_{\Te}\left|v^{\varepsilon}-w^{\varepsilon}\right| d m_{\tau}^{\varepsilon} d \tau+ C\int_{0}^{t} \int_{\Te}\left|w^{\varepsilon}-w\right|d m_{\tau}^{\varepsilon} d \tau \\
&&\leq C\int_{0}^{T} \int_{\Te}\left| D_{\cH}u-w\right| d m_{\tau} d \tau+C\int_{0}^{T} \int_{\Te} \int_{\He^1}\rho_{\varepsilon}(z)|w(x\oplus z)-w(x)| d z d \tau,
\end{eqnarray*}
where for the last inequality we used Lemma \ref{8110nostro} with $E= (D_{\cH}u-w)\, m$, $p=1$ and the definition of convolution \eqref{conH}. If $\epsilon \to 0$, from the continuity of $w$ we get \eqref{41}.
To complete the proof of the claim \eqref{superA}
 we just take a sequence $w_n$ of $\Te$-periodic functions, uniformly bounded continuous w.r.t. $x$ such that $w_n\to D_{\cH}u$ in $L^1(m_t,\Te)$.
 Applying \eqref{41} to $w_n$ and noting that 
 $m_{t}^{\eta}=m_{t}$ we get \eqref{superA}.
 \end{proof}

\noindent{\bf Acknowledgments.} 
The first and the second authors are members of GNAMPA-INdAM and were partially supported also by the research project of the University of Padova ``Mean-Field Games and Nonlinear PDEs'' and by the Fondazione CaRiPaRo Project ``Nonlinear Partial Differential Equations: Asymptotic Problems and Mean-Field Games''. The third author has been partially funded by the ANR project ANR-16-CE40-0015-01.

{\small{
 }

\end{document}